\documentclass[graybox]{svmult}

\usepackage{mathptmx}       % selects Times Roman as basic font
\usepackage{helvet}         % selects Helvetica as sans-serif font
\usepackage{courier}        % selects Courier as typewriter font
\usepackage{type1cm}        % activate if the above 3 fonts are
\usepackage{makeidx}         % allows index generation
\usepackage{graphicx}        % standard LaTeX graphics tool
\usepackage{multicol}        % used for the two-column index
\usepackage[bottom]{footmisc}% places footnotes at page bottom

\makeindex             % used for the subject index
\usepackage{graphicx}
\usepackage{color}
\usepackage{enumerate}
\usepackage{tikz}
\usetikzlibrary{shapes}
\usetikzlibrary{backgrounds}
\usetikzlibrary{patterns,fadings}
\usetikzlibrary{arrows,decorations.pathmorphing}
\usetikzlibrary{calc}

\definecolor{light-gray}{gray}{0.95}

\def\sqw{\hbox{\rlap{\leavevmode\raise.3ex\hbox{$\sqcap$}}$%
\sqcup$}}

\def\cqfd{\ifmmode\sqw\else{\ifhmode\unskip\fi\nobreak\hfil
\penalty50\hskip1em\null\nobreak\hfil\sqw
\parfillskip=0pt\finalhyphendemerits=0\endgraf}\fi}

\usepackage{caption}

\usepackage{mathrsfs}
\usepackage{amssymb}
\usepackage{amsfonts}
\usepackage{amsmath}

\usepackage{pdfsync}
\usepackage{hyperref}

\newtheorem{assumption}{Assumption}[section]

\newcommand{\mc}[1]{{\mathcal #1}}
\newcommand{\mf}[1]{{\mathfrak #1}}
\newcommand{\mb}[1]{{\mathbf #1}}
\newcommand{\bb}[1]{{\mathbb #1}}
\newcommand\bG{{\mathbf G}}
\newcommand\bH{{\mathbf H}}

\newcommand\bq{{\mathbf q}}
\newcommand\bp{{\mathbf p}}
\newcommand\be{{\mathbf e}}

\newcommand\N{{\mathbb N}}

\newcommand\RR{{\mathbb R}}

\newcommand\T{{\mathbb T}}
\newcommand\TT{{\mathbb T}}

\newcommand\PP{{\mathbb P}}
\newcommand\EE{{\mathbb E}}
\newcommand\LL{{\mathbb L}}
\newcommand\ZZ{{\mathbb Z}}
\newcommand\ve{\varepsilon}

\newcommand\br{{\mathbf r}}

\newcommand{\sfrac}[2]{{\smash{\frac{#1}{#2}}}}

\newcommand{\tnorm}{\vert \kern -1pt\vert\kern -1pt\vert}

\newcommand{\plus}{\!+\!}
\newcommand{\minus}{\!-\!}

\begin{document}

\title*{Diffusion of energy in chains of oscillators with conservative noise}
% Use \titlerunning{Short Title} for an abbreviated version of
% your contribution title if the original one is too long
\author{C\'edric Bernardin}
% Use \authorrunning{Short Title} for an abbreviated version of
% your contribution title if the original one is too long
\institute{C\'edric Bernardin \at Universit\'e de Nice Sophia-Antipolis, Laboratoire J.A. Dieudonn\'e, UMR CNRS 7351, Parc Valrose, 06108 Nice cedex 02- france, \email{cbernard@unice.fr}}
%
% Use the package "url.sty" to avoid
% problems with special characters
% used in your e-mail or web address
%
\maketitle

\abstract{These notes are based on a mini-course given during the conference Particle systems and PDE's - II which held at the Center of Mathematics of the University of Minho in December 2013. We discuss the problem of normal and anomalous diffusion of energy in systems of coupled oscillators perturbed by a stochastic noise conserving energy. }

\section*{}

%The foundations of a mathematical theory of non-equilibrium statistical mechanics is a long standing problem. In the occasion of the International Congress of Mathematicians held in Paris in 1900 Hilbert introduced this problem as the sixth and stated it as follows: \\
%
%%``Quant aux principes de la M\'ecanique, nous poss\'edons d\'eja au point de vue
%physique des recherches d'une haute port\'ee; je citerai, par exemple, les \'ecrits
%de MM. Mach [81], Hertz [64], Boltzmann [14] et Volkmann [107]. Il serait
%aussi tr\`es d\'esirable qu'un examen approfondi des principes de la M\'ecanique
%f\^ut alors tent\'e par les math\'ematiciens. Ainsi le Livre de M. Boltzmann sur les
%Principes de la M\'ecanique nous incite \`a \'etablir et \`a discuter au point de vue
%math\'ematique d'une mani\`ere compl\`ete et rigoureuse les m\'ethodes bas\'ees sur
%l'id\'ee de passage \`a la limite, et qui de la conception atomique nous conduisent
%aux lois du mouvement des continua. Inversement on pourrait, au moyen de
%m\'ethodes bas\'ees sur l'id\'ee de passage \`a la limite, chercher \`a d\'eduire les lois
%du mouvement des corps rigides d'un syst\`eme d'axiomes reposant sur la notion
%d'\'etats d'une mati\`ere remplissant tout l'espace d'une mani\`ere continue,
%variant d'une mani\`ere continue et que l'on devra d\'efinir param\'etriquement.
%Quoi qu'il en soit, c'est la question de l'\'equivalence des divers syst\`emes
%d'axiomes qui pr\'esentera toujours l'int\'er\^et le plus grand quant aux principes."
%
The goal of statistical mechanics is to elucidate the relation between the microscopic world and the macroscopic world. {\textit{Equilibrium statistical mechanics}} assume the microscopic systems studied to be in equilibrium. In this course we will be concerned with {\textit{non-equilibrium statistical mechanics}} where time evolution is taken into account: our interest will not only be in the relation between the microscopic and the macroscopic scales in space but also in time. 

By microscopic system we refer to molecules or atoms governed by the classical Newton's equations of motion. The question is then to understand how do these particles manage to organize themselves in such a way as to form a coherent structure on a large scale. The ``structure'' will be described by few variables (temperature, pressure \ldots) governed by autonomous equations (Euler's equations, Navier-Stokes's equation, heat equation \ldots). The microscopic specificities of the system will appear on this scale only through the thermodynamics (equation of state) and through the transport coefficients. Unfortunately, we are very far from understanding how to derive such macroscopic equations for physical relevant interactions.

One of the main ingredients that we need to obtain the macroscopic laws is that the particles, which evolve deterministically, have a behavior that one can consider almost as being random. The reason for this is that the dynamical system considered is expected to have a very sensitive dependence on the initial conditions and therefore is chaotic. This `` deterministic chaos'' is a poorly understood subject for systems with many degrees of freedom and even a precise consensual formulation is missing.    

A first simplification to attack these problems consists in replacing the deterministic evolution of particles {\textit{ab initio}} by purely stochastic evolutions. Despite this simplification we notice that the derivation of the macroscopic evolution laws is far from being trivial. For example, we do not have any derivation of a system of hyperbolic conservation laws from a stochastic microscopic system after shocks.  Nevertheless, since the pioneering work of Guo, Papanicolaou,Varadhan (\cite{GPV1}) and Yau (\cite{yau1}), important progresses have been performed in several well understood situations by the development of robust probabilistic and analytical methods (see \cite{KL} and \cite{Sp} for reviews).  

In this course we will be mainly (but not only) interested in hybrid models for which the time evolution is governed by a combination of deterministic and stochastic dynamics. These systems have the advantage to be mathematically tractable and conserve some aspects of the underlying deterministic evolution. The stochastic noise has to be chosen in order to not destroy the main features of the Hamiltonian system that we perturb. 

The central macroscopic equation  of these lecture notes is the heat equation:
\begin{equation*}
\begin{cases}
\partial_t u = \partial_{x} ( D(u) \partial_x u), \quad x \in {\mathring U}, \quad t > 0,\\
\quad u(0, x) = u_{0} (x), \quad x \in U,\\
\quad u(t,x)=b(x), \quad x \in \partial U, \quad t>0.
\end{cases}
\end{equation*}
Here $u(t,x)$ is a function of the time $ t \ge 0$ and the space $x \in U \subset \RR^d$, $d \ge 1$, starting from the initial condition $u_0$ and subject to boundary conditions prescribed by the function $b$. The advantage of the heat equation with respect to other macroscopic equations such as the Euler or Navier-Stokes equations is that the notion of solution is very well understood. The dream would be to start from a system of $N \gg 1$ particles whose interactions are prescribed by Newton's laws and to show that in the large $N$ limit, the empirical energy converges in the diffusive time scale $t=\tau  N^2$ to $u$ ($\tau$ is the microscopic time and $t$ the macroscopic time). In fact, this picture is expected to be valid only under suitable conditions and to fail for some low dimensional systems. In the case where the heat equation (or its variants) holds we say that the system has a normal behavior. Otherwise anomalous behavior occurs and the challenging question (even heuristically) is to know by what we shall replace the heat equation and what is the time scale over which we have to observe the system in order to see this macroscopic behavior ( \cite{BLRB}, \cite{Dhar},\cite{LLP1} for reviews).

The course is organized as follows. In Chapter \ref{ch:models} we introduce the models studied. Chapter \ref{ch:normal} is concerned with models which have a normal diffusive behavior. In Chapter \ref{ch:anomalous} we are interested in systems producing an anomalous diffusion. An important issue not discussed here is the effect of disorder on diffusion problems. In order to deal with lecture notes of a reasonable size, many of the proofs have been suppressed or only roughly presented.

%%%%%%%%%%%%%%%%%%%%%%%%%%%%%%%%%%%%%%%%%%%
% Chapter 1 : Models
%%%%%%%%%%%%%%%%%%%%%%%%%%%%%%%%%%%%%%%%%%%

\section[Models]{Chains of oscillators}
\label{ch:models}

\subsection{Chains of oscillators with bulk noise}

Chains of coupled oscillators are usual microscopic models of heat conduction in solids. Consider a finite box $\Lambda_N =\{ 1 ,\ldots, N\}^d \subset \ZZ^d$, $d\ge 1$, whose boundary $\partial \Lambda_N$ is defined as $\partial \Lambda_N = \{ x \notin \Lambda_N \, ; \, \exists y \in \Lambda_N, \; |x-y|=1\}$. Here $|\cdot|$ denotes the Euclidian norm in $\RR^d$ and $`` \cdot"$ the corresponding scalar product. Let us fix  a nonnegative pair interaction potential $V$ and a pinning potential $W$ on $\RR$. The atoms are labeled by $x \in \Lambda_N$. The momentum of atom $x$ is $p_x \in {\RR}$ and its displacement from its equilibrium position {\footnote{We restrict us to the case where $q_x \in \RR^n$ with $n=1$ because the relevant dimension of the system is the dimension $d$ of the lattice. Most of the results stated in this manuscript can be generalized to the case $n \ge 1$.}} is ${q_x} \in \RR$. The energy ${\mc E}_x$ of the atom $x$ is the sum of the kinetic energy, the pinning energy and the interaction energy: 
\begin{equation}
\label{eq:1-energy-site-x}
{\mc E}_x= \frac{|p_x|^2}2 + W(q_x) + \frac{1}{2} \sum_{\substack{|y-x|=1,\\ y\in \Lambda_N}} V(q_x - q_y). 
\end{equation}
The Hamiltonian is given by
\begin{equation}
  \label{eq:1-hamilt}
  \mathcal H_N =  \sum_{x \in \Lambda_N}  {\mc E}_x + \partial \, {\mc H}_N
\end{equation}
where ${\partial} \, {\mc H}_N$ is the part of the Hamiltonian corresponding to the boundary conditions which are imposed. 

%%%%%%%%%%%%%%picture%%%%%%%%%%%%
%%%%%%%%%%%%%%%%%%%%%%%%%%%%%%%
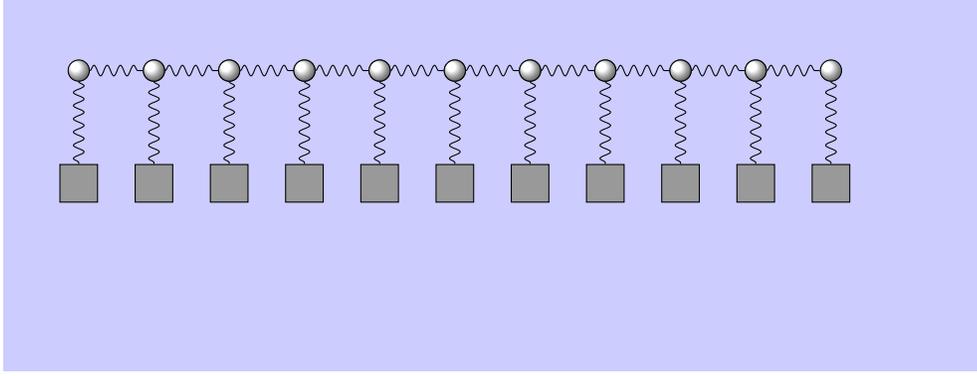
\begin{figure}[htbp]
\centering
\begin{tikzpicture}[scale=1]
\tikzfading[name=fade out, inner color=transparent!45,
         outer color=transparent!100]
\fill [blue!20, path fading=fade out] (-1,-4) rectangle (12,1);

\tikzstyle{atome}=[draw,circle,ball color=white];
\foreach \y in {0,...,9}{
\node[atome] (\y) at (\y,0) {};
\node[atome] (\y+1) at (\y+1,0) {};
\node[draw,rectangle, minimum width=0.5cm, minimum height=0.5cm,fill=black!40] (p\y) at (\y,-1.5){};
}
\foreach \y in {0,...,9}{
\draw [decorate,decoration={snake,amplitude=2pt,segment length=5pt}] (\y) -- (\y+1);
\draw [decorate,decoration={snake,amplitude=2pt,segment length=5pt}] (\y) -- (p\y);
}
\node[draw,rectangle, minimum width=0.5cm, minimum height=0.5cm,fill=black!40] (plast) at (9+1,-1.5){};
\draw [decorate,decoration={snake,amplitude=2pt,segment length=5pt}] (9+1) -- (plast);
\end{tikzpicture}
\caption{A one-dimensional chain of pinned oscillators with free boundary conditions}
\end{figure}

%%%%%%%%%%%%%%%%%%%%%%%%%%%%%%%
%%%%%%%%%%%%%%%%%%%%%%%%%%%%%%%%

We will consider the following cases:
\begin{itemize}
\item Periodic boundary conditions: we identify the site $1$ to the site $N$ and denote the corresponding box by ${\TT}_N$, the discrete torus of length $N$ (then $\partial {\mc H}_N =0$).
\item Free boundary conditions: this corresponds to the absence of boundary conditions, i.e. to $\partial {\mc H}_N =0$.
\item Fixed boundary conditions: introduce the positions $q_y=0$, $y\in \partial \Lambda_N$, of some fictive walls. We add to the Hamiltonian ${\mc H}_N$ a boundary term $\partial {\mc H}_N = \partial^{\rm{f}}\,  {\mc H}_N$ given by
\begin{equation*}
\partial^{\rm{f}}\,  {\mc H}_N =  \sum_{\substack{|y-x|=1,\\ x \in \Lambda_N, y\in \partial\Lambda_N}} V(q_x - q_y)= \sum_{\substack{|y-x|=1,\\ x \in \Lambda_N, y\in \partial\Lambda_N}} V(q_x).
\end{equation*}
\item Forced boundary conditions: site ${\bf 1}=(1,\ldots,1)$ is in contact with a wall at position $q_0=0$ and each site $y \in \partial \Lambda_N \backslash \{ 0\}$ is driven by a constant force $\tau_y$. This results in a boundary term $\partial {\mc H}_N = \partial^{\rm{\tau}} \, {\mc H}_N$ given by  
\begin{equation}
\partial^{\rm{\tau}} \, {\mc H}_N = \sum_{\substack{|y-x|=1,\\ x \in \Lambda_N, y\in \partial\Lambda_N}} V(q_x - q_y) - \sum_{ y\in \partial\Lambda_N \backslash \{ 0 \} } \tau_y q_y.
\end{equation}
\end{itemize}
The equations of motion of the atoms are 
\begin{equation}
\label{eq:1-Newton}
{\dot {q_x}}= \partial_{p_x} {\mc H}_N, \quad {\dot p_x} = - \partial_{q_x} {\mc H}_N
\end{equation}
and the generator ${\mc A}_N$ of the system is given by the Liouville operator
\begin{equation*}
{\mc A}_N = \sum_{x \in {\Lambda}_N} \left\{  \partial_{p_x} {\mc H}_N \, \partial_{q_x} \, - \,  \partial_{q_x} {\mc H}_N \, \partial_{q_x} \right\}.
\end{equation*}

It will be also useful to consider the chain of oscillators in infinite volume, i.e. replacing $\Lambda_N$ by $\ZZ^d$, $d\ge 1$, in the definitions above. The formal generator ${\mc A}_N$ is then denoted by ${\mc A}$. The dynamics can be defined for a large set of initial conditions if $V$ and $W$ do not behave too badly (\cite{LLL}, \cite{MPP}, \cite{BO-livre1}). We define the set $\Omega$ as the subset of ${\RR}^{\ZZ^d}$ given by 
\begin{equation}
\Omega= \bigcap_{\alpha >0} \left\{ \xi \in \RR^{\ZZ^d} \, ; \, \sum_{x \in \ZZ^d} e^{-\alpha|x|} |\xi_x|^2 < +\infty \right\}
\end{equation}
and ${\tilde \Omega}= \Omega \times \Omega$.  We equip $\Omega$ with its natural product topology 
and its Borel $\sigma$-field and $\tilde \Omega$ by the corresponding product topology. For $X=\Omega$ or $X={\tilde \Omega}$, the set of Borel probability measures on $X$ is 
denoted by ${\mathcal P} (X)$. A function $f: X \to \RR$ is said to be \emph{local} 
if it depends of $\xi$ only through the coordinates $\{ \xi_x\, ; \, x \in \Lambda_f\}$, 
$\Lambda_f$ being a finite box of $\ZZ$. We also introduce the sets $C_0^k(X)$, $k \ge 1$ (resp. $k=0$),  
of bounded local functions on $X$ which are differentiable up to order $k$ with bounded partial derivatives (resp. continuous and bounded).

In the rest of the manuscript, apart from specific cases, we will assume that one of the following conditions hold:
\begin{itemize}
\item  The potentials  $V$ and $W$ have bounded second derivatives.  Then the infinite dynamics $(\omega(t))_{t \ge 0}$ can be defined for any initial condition $\omega^0=(\bq^0, \bp^0) \in {\tilde \Omega}$. Moreover ${\tilde \Omega}$ is invariant by the dynamics. This defines a semigroup $(P_t)_{t \ge 0}$ on $C_0^0 ({\tilde \Omega})$ and the Chapman-Kolmogorov equations 
\begin{equation}
(P_t f)(\omega) - f(\omega) = \int_0^t (P_s {\mc A} f ) (\omega) \, ds \; = \;   \int_0^t ({\mc A} P_s  f ) (\omega) \, ds 
\end{equation}
are valid for any $f \in C_0^1 ({\tilde \Omega})$.

\item The potential $W=0$ and the interaction potential $V$ has a second derivative uniformly bounded from above and below. It is more convenient to go over the deformation field $\eta_{(x,y)}= q_y -q_x$, $|x-y|=1$, which by construction is constrained to have zero curl. In $d=1$ we will denote $\eta_{(x-1,x)}= q_{x}- q_{x-1}$ by $r_x$. The dynamics (\ref{eq:1-Newton}) can be read as a dynamics for the deformation field and the momenta. Given say $q_0$, the scalar field $\bq=\{ q_x\}_{x \in \ZZ^d}$ can be reconstructed from $\eta$. In the sequel, when $W=0$, we will use these coordinates without further mention. The dynamics for the coordinates $\omega=(\eta, \bp)= (\eta_{(x,x+\be)}, p_x)_{|\be|=1, x \in \ZZ}$  can be defined if the initial condition satisfies $\omega^0 \in {\tilde \Omega}$. Moreover the set ${\tilde \Omega}$ is invariant by the dynamics. This defines a semigroup $(P_t)_{t \ge 0}$ on $C_0^0 ({\tilde \Omega})$ and the Chapman-Kolmogorov equations 
\begin{equation}
\label{eq:CK1}
(P_t f)(\omega) - f(\omega) = \int_0^t (P_s {\mc A} f ) (\omega) \, ds \; = \;   \int_0^t ({\mc A} P_s  f ) (\omega) \, ds 
\end{equation}
are valid for any $f \in C_0^1 ({\tilde \Omega})$.{\footnote{The generator ${\mc A}$ has to be written in terms of the deformation field.}}
\end{itemize}

Let us first consider the problem related to the characterization of equilibrium states. For simplicity we take  the finite volume dynamics with periodic boundary conditions. Then it is easy to see that the system conserves one or two physical quantities depending on whether the chain is pinned or not. The total energy ${\mc H}_N$ is always conserved. If $W=0$ the system is translation invariant and the total momentum $\sum_{x} p_x$ is also conserved. Notice that because of the periodic boundary conditions the sum of the deformation field $\sum_{x} \eta_{(x,x+e_i)}$ is automatically fixed equal to $0$ for any $i=1, \ldots, d$. 

Liouville's Theorem implies that the uniform measure $\lambda^N$ on the manifold $\Sigma^N$ composed of the configurations with a fixed total energy (and possibly  a fixed total momentum) is invariant for the dynamics. The micro canonical ensemble is defined as the probability measure $\lambda^N$. The dynamics restricted to $\Sigma^N$ is not necessarily ergodic. Two examples for which one can show it is not the case are the harmonic lattice ($V$ and $W$ quadratic) and the Toda lattice ($d=1$, $W=0$, $V ( r )=e^{-r} -1 +r $) which is a completely integrable system (\cite{Toda}). In fact what is really needed for our purpose is not the ergodicity of the finite dynamics but of the infinite dynamics. We expect that even if the finite dynamics are never ergodic the fraction of $\Sigma^N$ corresponding to non ergodic behavior decreases as $N$ increases, and probably disappears as $N=\infty$ (apart from very peculiar cases). Therefore a good notion of ergodicity has to be stated for infinite dynamics. The definition of a conserved quantity is not straightforward in infinite volume (the total energy of the infinite chain is usually equal to $+ \infty$). To give a precise definition we will use the notion of space-time invariant probability measures for the infinite dynamics defined above.

The infinite volume Gibbs grand canonical ensembles are such probability measures. They form a set of probability measures indexed by one (pinned chains) or $d+2$ (unpinned chains) parameters  and are defined by the so-called Dobrushin-Landford-Ruelle's equations. To avoid a long discussion we just give a formal definition (see e.g. \cite{Giac1} for a detailed study).
\begin{itemize}
\item Pinned chains ($W \ne 0$): the infinite volume Gibbs grand canonical ensemble $\mu_{\beta}$ with inverse temperature $\beta>0$ is the probability measure on ${\tilde \Omega}$ whose density with respect to the Lebesgue measure is
\begin{equation*}
Z^{-1} (\beta) \exp \left( - \beta \sum_{x \in \ZZ^d} {\mc E}_x \right).
\end{equation*} 
\item Unpinned chains ($W=0$): the infinite volume Gibbs grand canonical ensemble {\footnote{They are defined with respect to the gradient fields $\eta_{(x,y)}$. It would be more coherent to call them {\textit{gradient Gibbs measures}}. }} $\mu_{\beta, {\bar p}, \tau}$ with inverse temperature $\beta>0$, average momentum ${\bar p} \in \RR$ and tension $\tau=\beta^{-1} \lambda \in \RR^d$ is the probability measure on ${\tilde \Omega}$ whose density with respect to the Lebesgue measure is
\begin{equation}
\label{eq.Gibbsmeasures}
Z^{-1} (\beta, {\bar p}, \tau) \;  \exp \left( - \beta \, \sum_{x \in \ZZ^d} \{ {\mc E}_x - {\bar p}\,  p_x - \sum_{i=1}^d \tau_i \, \eta _{(x, x+e_i)}\} \right).
\end{equation} 
\end{itemize} 
Observe that in the one dimensional unpinned case we have simply product measures and that the tension $\tau$ is equal to the average of $V'(r_x)$.

Fix an arbitrary Gibbs grand canonical ensemble $\mu$. A probability measure $\nu$ is said to be $\mu$-regular if for any finite box $\Lambda \subset \ZZ^d$ whose cardinal is denoted by $|\Lambda|$,  the relative entropy of $\nu |_{\Lambda}$ w.r.t. $\mu |_{\Lambda}$ is bounded above by $C | \Lambda|$ for a constant $C$ independent of $\Lambda$. We recall that the relative entropy $H(\nu|\mu)$ of $\nu \in {\mc P} (X)$ with respect to 
$\mu \in {\mc P} (X)$, $X$ being a probability space, is defined as
\begin{equation}
\label{eq:ent009}
H(\nu | \mu) = \sup_{\phi} \left\{ \int \phi \, d\nu - 
\log \left( \int e^{\phi} \, d\mu \right) \right\},
\end{equation}
where the supremum is carried over all bounded measurable functions $\phi$ on $X$.  

For any arbitrary Gibbs grand canonical ensembles $\mu$ and $\mu'$, $\mu$ is $\mu'$-regular and $\mu'$ is $\mu$-regular. Therefore $\nu$ is $\mu$-regular is equivalent to $\nu$ is $\mu'$-regular and we simply say that $\nu$ is regular.  

A notion of ergodicity for infinite dynamics which is suitable to derive rigorously large scale limits of interacting particle systems is the following.

\begin{definition}[Macro-Ergodicity] {\footnote{The name has been proposed by S. Goldstein.}}  We say that the dynamics generated by ${\mc A}$ is {\textit{macro-ergodic}} if and only if the only space-time invariant {\footnote{Observe that a probability measure $\nu$ is time invariant for the infinite dynamics if and only if $\int {\mc A} f \, d\nu =0$ for any $f \in C_0^1 ({\tilde \Omega})$. This is a consequence of the Chapman-Kolmogorov equations (\ref{eq:CK1}).}} regular measures $\nu$ for ${\mc A}$ are mixtures (i.e. generalized convex combinations)  of Gibbs grand canonical ensembles. 
\end{definition}

If the microscopic dynamics is macro-ergodic, then, by using the relative entropy method developed in \cite{OVY}, we can derive the hydrodynamic equations {\footnote{The notion of hydrodynamic limits is detailed in Section \ref{subsec:hl-vf} and Section \ref{sec:hydroLimEuler}.}} in the Euler time scale of the chain before the appearance of the shocks, at least in $d=1$ (\cite{BO-livre1}). These limits form a triplet of compressible Euler equations (for energy ${\mf e}$, momentum ${\mf p}$ and deformation ${\mf r}$) of the form
\begin{equation}
\label{eq:Euler-general0}
\begin{cases}
\partial_t {\mf r} = \partial_q {\mf p}\\
\partial_t {\mf p} = \partial_q {\mf \tau}\\
\partial_t {\mf e} = \partial_q ({\mf p}  \tau)
\end{cases}
\end{equation}
where the pressure $\tau:= \tau (\mf r, {\mf e} -\tfrac{{\mf p}^2}{2})$ is a suitable thermodynamic function depending on the potential $V$. A highly challenging open question is to extend these results after the shocks. The proof can be adapted to take into account the presence of mechanical boundary conditions  (\cite{Even-Olla}).

We do not claim that the macro-ergodicity is a necessary condition to get Euler equations for purely Hamiltonian systems. We could imagine that weaker or different conditions are sufficient but in the actual state of the art the macro-ergodicity is a clear and simple mathematical statement of what we could require from deterministic systems in order to derive Euler equations rigorously. We refer the interested reader to \cite{Bricmont-Chance} and \cite{Szasz-ergodic} for interesting discussions about the role of ergodicity in statistical mechanics.

\subsubsection{Conserving noises}

In \cite{FFL}, Fritz, Funaki and Lebowitz prove a weak form of macro-ergodicity for a chain of anharmonic oscillators under generic assumptions on the potentials $V$ and $W$ that we do not specify here (see \cite{FFL}).

\begin{theorem}[\cite{FFL}] {\footnote{The proof given in \cite{FFL} assumes $W \ne 0$ but it can be adapted to the unpinned one dimensional case (see {\cite{BO-livre1}}). It would be interesting to extend this theorem to the general unpinned case.}} Consider the pinned chain $W \ne 0$ generated by ${\mc A}$ or an unpinned chain $W=0$ in $d=1$. The only regular time and space invariant measures for ${\mc A}$ which are such that conditionally to the positions configuration $\bq:=\{q_x \, ;\, x \in \ZZ^d\}$ the law of the momenta $\bp:=\{ p_x \, ;\, x \in \ZZ^d\}$ is exchangeable are given by mixtures of Gibbs grand canonical ensembles. 
\end{theorem}

They also proposed to perturb the dynamics by a stochastic noise that consists in exchanging at random exponential times, independently for each pair of nearest neighbors site $x,y \in \ZZ^d$, $|x-y|=1$, the momenta $p_x$ and $p_y$. The formal generator ${\mc L}$ of this dynamics, that we will call the {\textit{stochastic energy-momentum conserving dynamics}},  is given by ${\mc L} ={\mc A} +\gamma {\mc S}$, $\gamma >0$, where ${\mc A}$ is the Liouville operator and ${\mc S}$ is defined for any local function $f: {\tilde \Omega} \to \RR$ by
\begin{equation}
\label{eq:cons-noise-exch-mom}
({\mc S} f)(\bq,\bp) = \sum_{\substack{x,y \in \ZZ^d\\ |x-y|=1}} \left[ f(\bq,\bp^{x,y}) -f(\bq, \bp) \right].
\end{equation}
Here the momenta configuration $\bp^{x,y}$ is the configuration obtained from $\bp$ by exchanging $p_x$ with $p_y$. The previous discussion about existence of the dynamics on ${\tilde \Omega}$ for the deterministic case and its relation with its formal generator is also valid for this dynamics and the other dynamics defined in this section.   

With some non-trivial entropy estimates we get the following result.

\begin{theorem}[\cite{FFL}]
\label{th:1:ergexc} 
Consider the pinned ($W\ne 0$) or the one-dimensional unpinned ($W \ne 0$) stochastic energy-momentum conserving dynamics. The only regular time and space invariant measures for these dynamics are given by mixtures of Gibbs grand canonical ensembles, i.e. the stochastic energy-momentum conserving dynamics is macro-ergodic.
\end{theorem}

Consequently the stochastic energy-momentum conserving dynamics is macro-ergodic. By using the relative entropy method developed in \cite{OVY},  one can show it has in the Euler time scale and before the appearance of the shocks the same hydrodynamics (\ref{eq:Euler-general0}) as the deterministic model. This is because the noise has some macroscopic effects only in the diffusive time scale (\cite{BO-livre1}). 

We consider now a different stochastic perturbation. Let us define the flipping operator $\sigma_x: \bp \in \Omega \to {\bp}^x \in \Omega$ where ${\bp}^x$ is the configuration such that $({\bp}^x)_z =p_z$ for $z \ne x$ and $({\bp}^x)_x =-p_x$.   In \cite{FFL} is also proved that the only time-space regular stationary measures for the Liouville operator ${\mc A}$ such that conditionally to the positions the momenta distribution is invariant by any flipping operator $\sigma_{x}$ are mixtures of Gibbs grand canonical ensembles with zero momentum average. Then we consider the dynamics on ${\tilde \Omega}$ generated by ${\mc L} ={\mc A} +\gamma {\mc S}$, $\gamma >0$, with ${\mc S}$ the noise defined by
\begin{equation}
\label{eq:cons-noise-flip}
({\mc S} f)(\bq, \bp)= \cfrac{1}{2} \sum_{x \in \ZZ^d} \left[ f(\bq, \bp^x) -f(\bq,\bp)\right]
\end{equation}
for any local function $f: {\tilde \Omega} \to \RR$. This dynamics conserves the energy and the deformation of the lattice but destroys all the other conserved quantities. We call this system the \textit{velocity-flip model} (sometimes the \textit{stochastic energy conserving model}). 

\begin{theorem}[\cite{FFL}]
\label{th:1:ergflip} 
Consider the pinned $d$-dimensional velocity-flip model or the one-dimensional unpinned velocity-flip model.  The only regular time and space invariant measures are given by mixtures of Gibbs grand canonical ensembles. In other words the velocity-flip model is macro-ergodic.
\end{theorem}

Since the velocity flip-model does not conserve the momentum its Gibbs invariant measures are given by (\ref{eq.Gibbsmeasures}) with ${\bar p}=0$. In particular the average currents with respect to theses measures is zero. Therefore assuming propagation of local equilibrium in the Euler time scale we get that it has trivial hydrodynamics in this time scale: initial profile of energy does not evolve. This is only in the diffusive scale that an evolution should take place. 

\subsubsection{NESS of chains of oscillators perturbed by an energy conserving noise}

The models defined in the previous sections can also be considered in a non-equilibrium stationary state (NESS) by letting them in contact with thermal baths at different temperatures and imposing various mechanical boundary conditions. Let us only give some details for the NESS of the one-dimensional velocity-flip model.  

Consider a chain of $N$ unpinned oscillators where the particle $1$ (resp. $N$) is subject to a constant force $\tau_\ell$ (resp. $\tau_r$). Moreover we assume that the particle $1$ (reps. $N$) is in contact with a Langevin thermal bath at temperature $T_{\ell}$ (resp. $T_r$). The generator ${\mc L}_N$ of the dynamics on the phase space $\Omega_N=\RR^{N-1} \times \RR^N$ is given by
\begin{equation}
\label{eq:1:ness-gen}
{\mc L}_N= {\mc A}^{\tau_\ell, \tau_r}_N +\gamma {\mc S}_N + \gamma_\ell {\mc B}_{1,T_\ell} + \gamma_r {\mc B}_{N,T_r}, \quad \gamma>0,
\end{equation}
where ${\mc A}^{\tau_\ell, \tau_r}_{N}$ is the Liouville operator, 
${\mc B}_{j,T}$ the generator of the Langevin bath at temperature $T$
acting on the $j$--th particle and ${\mc S}_N$ the generator of the
noise. The strength of noise and thermostats are regulated by
$\gamma$, $\gamma_{\ell}$ and $\gamma_r$ respectively. The Liouville operator is
defined by 
\begin{equation}
  \label{eq:ss51}
   \begin{split}
    {\mathcal A }^{\tau_\ell, \tau_r}_{N}= 
    \sum_{x=2}^{N} \left(p_{x} - p_{x-1}\right) \partial_{r_x} +
        \sum_{x=2}^{N-1}\left(V'(r_{x+1}) - V'(r_{x})\right)
      \partial_{p_{x}}\\
    - \left(\tau_\ell - V'(r_2)\right) \partial_{p_1} 
    + \left(\tau_r - V'(r_N)\right) \partial_{p_N}.
  \end{split}
\end{equation}
The generators of the thermostats are given by
\begin{equation}
  \label{eq:ss49}
  {\mc B}_{j,T} = T \partial_{p_j}^2 - p_j \partial_{p_j}.
\end{equation}
The noise corresponds to independent velocity change of sign, i.e.
\begin{equation}
\label{eq:ss41}
({\mc S}_N f ) (\br,\bp) = \frac{1}{2} \sum_{x=2}^{N-1} \left(f(\br,\bp^{x}) - f(\br,\bp)\right), \quad f: \Omega_N \to \RR .
\end{equation}
We will also consider the case where the chain has fixed boundary conditions. 

\begin{proposition}[\cite{BO1}, \cite{BO2},\cite{Car}]
Consider a finite chain of pinned or unpinned oscillators with fix, free or forced boundary conditions in contact with two thermal baths at different temperatures and perturbed by one of the energy conserving noises defined above. Then, there exists a unique non-equilibrium stationary state for this dynamics which is absolutely continuous w.r.t. Lebesgue measure. 
\end{proposition}

\begin{proof}
The proof of the existence of the invariant state can be obtained from the knowledge of a suitable Liapounov function. To prove the uniqueness of the invariant measure it is sufficient to prove that the dynamics is irreducible and has the strong-Feller property. Some hypoellypticity, control theory and conditioning arguments are used to achieve this goal. 
\cqfd
\end{proof}

\subsection{Simplified perturbed Hamiltonian systems}
\label{sec:intro-phd}

Introducing a noise into the deterministic dynamics help us to solve some ergodicity problems. Nevertheless, as we will see, several challenging problems remain open for chains of oscillators perturbed by a conservative noise. In \cite{BS} we proposed to simplify still these models and the main message addressed there is that the models introduced in \cite{BS} have qualitatively the same behaviors as the unpinned chains. For simplicity we define only the dynamics in infinite volume. 

Let $U$ and $V$ be two potentials on $\RR$ and consider the Hamiltonian system $(\omega (t) )_{t \ge 0} = ( \, {\br} (t) , {\bp} (t) \,)_{t \ge 0}$ described by the equations of motion
\begin{equation}
\label{eq:1:generaldynamics}
\frac{dp_x}{dt} = V'(r_{x+1}) -V'(r_x), 
\qquad \frac{dr_x}{dt} = U' (p_x) -U' (p_{x-1}), 
\qquad x \in \ZZ, 
\end{equation}
where $p_x$ is the momentum of particle $x$, $q_x$ its position and $r_x=q_{x} -q_{x-1}$ the ``deformation''. Standard chains of oscillators are recovered for a quadratic kinetic energy $U(p)=p^2 /2$. The dynamics conserves (at least) three physical quantities: the total momentum $\sum_{x} p_{x}$, the total deformation $ \sum_{x} r_{x}$ and the total energy $\sum_x {\mc E_x}$ with ${\mc E}_x= V(r_x) + U(p_x)$. Consequently, every Gibbs grand canonical ensemble ${\nu}_{\beta,\lambda, \lambda'}$ defined by
\begin{equation}
\label{eq:invmeas007}
d{\nu}_{\beta,\lambda,\lambda'} (\eta) = \prod_{x \in \ZZ} {\mc Z} (\beta,\lambda,\lambda')^{-1} 
\exp\left\{ -\beta {\mc E}_x -\lambda p_{x} -\lambda' r_{x} \right\} \, d r_x \, d p_x 
\end{equation}
is invariant under the evolution.  To simplify we assume that the potentials $U$ and $V$ are smooth potentials with second derivatives bounded by below and from above.

To overcome our ignorance about macro-ergodicity of the dynamics, as before, we add a stochastic conserving perturbation. In the general case $U \ne V$, the Hamiltonian dynamics can be perturbed by the energy-momentum conserving noise acting on the velocities (as proposed in \cite{FFL}) but conserving the three physical invariants mentioned above. Then the infinite volume dynamics can be defined on the state space ${\tilde \Omega}$. Its generator ${\mc L}$ is given by ${\mc L} = {\mc A} +\gamma {\mc S}$, $\gamma >0$, where
\begin{equation}
\label{eq:1:ASdynum}
\begin{split}
&({\mc A} f) (\br ,\bp) =\sum_{x \in \ZZ} \left\{\,  (V'(r_{x+1}) -V'(r_x)) \partial_{p_x} f\,  + \,  (U' (p_x) -U' (p_{x-1})) \partial_{r_x} f \right\}\,  (\br,\bp)\\
& ({\mc S} f) = \sum_{x \in \ZZ} \left[ f(\br, \bp^{x,x+1}) - f (\br, \bp) \right] 
\end{split}
\end{equation} 
for any $f \in C_0^1 ({\tilde \Omega})$.

\begin{theorem}[\cite{BS}]
Assume that the potentials $U$ and $V$ are smooth potentials with second derivatives bounded by below and from above. The dynamics generated by ${\mc L}={\mc A} +\gamma {\mc S}$ with $\gamma>0$ and ${\mc A}, {\mc S}$ given by (\ref{eq:1:ASdynum})  is macro-ergodic. Consequently, before the appearance of the shocks, in the Euler time scale, the hydrodynamic limits are given by a triplet of compressible Euler equations.
\end{theorem}

Our motivation being to simplify as much as possible the dynamics considered in~\cite{BBO1,BBO2} without destroying the anomalous behavior of the energy diffusion, we mainly focus on the symmetric case $U=V$. Then the $\bp$'s and $\br$'s play a symmetric role so there is no reason that momentum conservation is more important than deformation conservation. We propose thus to add a noise conserving only the energy and $\sum_{x} [ r_x + p_x]$.  It is more convenient to use the variables $\{ \eta_x \, ; \, x \in \ZZ \} \in \RR^{\ZZ}$ defined by $\eta_{2x}= p_x$ and $\eta_{2x -1} =r_x$ so that (\ref{eq:1:generaldynamics}) becomes
\begin{equation}
\label{eq:1:dyneq}
d\eta_x = \left[ V' (\eta_{x+1}) -V' (\eta_{x-1}) \right] dt, \quad x \in \ZZ.
\end{equation}
We might also interpret the dynamics for the $\eta$'s as the dynamics of an interface whose height (resp. energy) at site $x$ is $\eta_x$ (resp. $V(\eta_x)$). It is then quite natural to call the quantity $\sum_x \eta_x$ the ``volume".

Hence, we introduce the so-called {\textit{stochastic energy-volume  conserving dynamics}}, which is still described by (\ref{eq:1:dyneq}) between random exponential times where two nearest neighbors heights $\eta_x$ and $\eta_{x+1}$ are exchanged. Observe that in the momenta-deformation picture this noise is less degenerate than the momenta exchange noise since exchange between momenta and positions is now allowed. The generator ${\mc L}$ of the infinite volume dynamics, well defined on the state space $\Omega$, is given by ${\mc L} ={\mc A} +\gamma {\mc S}$, $\gamma>0$, where for any $f \in C_0^1 (\Omega)$, 
\begin{equation}
\label{eq:1:ASdynum2}
\begin{split}
& ({\mc A}f)(\eta) = \sum_{x \in \ZZ} \left[ V' (\eta_{x+1}) -V' (\eta_{x-1}) \right] (\partial_{\eta_x} f )(\eta),\\
& ({\mc S}f)(\eta) = \sum_{x \in \ZZ} \left[ f(\eta^{x,x+1}) -f(\eta)  \right].
\end{split}
\end{equation}

The noise still conserves the total energy and the total volume but destroys the conservation of momentum and deformation. Therefore, only two quantities are conserved and the invariant Gibbs grand canonical measures of the stochastic dynamics correspond to the choice $\lambda=\lambda'$ in~(\ref{eq:invmeas007}). We denote ${\nu}_{\beta,\lambda,\lambda}$ (resp. ${\mc Z}(\beta,\lambda,\lambda)$) by $\mu_{\beta, \lambda}$ (resp. $Z (\beta,\lambda)$).

%%%%%%
%
%  CHAPTER 3: NORMAL DIFFUSION
%
%%%%%%%%%%%%%%%%%%%%%%%%%%%%%%%%%%%%

\section{Normal diffusion}
\label{ch:normal}

Normal diffusion of energy in purely deterministic homogeneous chains of oscillators is expected to hold in high dimension ($d\ge 3$) or if momentum is not conserved, i.e. in the presence of a pinning potential. The problem of anomalous diffusion will be discussed in the next chapter. In this chapter we consider the case of normal diffusion.

The first step to show such normal behavior is to prove that the transport coefficient, the thermal conductivity, is well defined.  Once it has been achieved, the following non-equilibrium problems can be considered:
\begin{itemize}
\item Hydrodynamic limits in the diffusive time scale $t \ve^{-2}$, $\ve$ being the scaling parameter: if the system has trivial hydrodynamics in the time scale $t \ve^{-1}$, i.e. if momentum is not conserved, we would like to show that in the diffusive time scale, the macroscopic energy profile evolves according to a diffusion equation. If the system has non-trivial hydrodynamics given by the Euler equations in the hyperbolic scaling (i.e. if momentum is conserved), in the diffusive time scale, we would like to derive the incompressible Navier-Stokes equations. These would be obtained by starting with some initial momentum macroscopic profile of order ${\mc O}(\ve)$ but an energy profile of order ${\mc O} (1)$.  
%\item Study the energy fluctuation field in the diffusive time scale $t \ve^{-2}$: this would consist in proving that the latter converges to an infinite dimensional Ornstein-Uhlenbeck process as $\ve$ goes to zero. If a macroscopic transport occurs in the Euler time scale, the energy fluctuation field has to be considered in a moving frame defined by this macroscopic transport. 
\item Validity of Fourier's law: we consider the NESS of the system in contact at the boundaries with thermal baths at different temperatures. Fourier's law expresses that the average of the energy current in the NESS is proportional to the gradient of the local temperature. The proportionality coefficient is called the thermal conductivity.
\end{itemize}

Assume for simplicity that $d=1$ and that the energy is the only conserved quantity. The corresponding microscopic current, denoted by  $j^{e}_{x,x+1}$, is defined by the local energy conservation law
\begin{equation}
{\mc L} {\mc E}_x = - \nabla j^{e}_{x-1,x}
\end{equation}
where ${\mc L}$ is the generator of the infinite dynamics under investigation and $\nabla$ is the discrete gradient defined for any $(u_x)_x \in \RR^{\ZZ}$ by $\nabla u_x = u_{x+1} -u_x$. In the current state of the art, in all the problems mentioned above, the usual approach consists to prove that there exist functions $\varphi_x =\theta_x \varphi_0$ and $h_x =\theta_x h_0$ (actually only approximate solutions are needed) such that the following decomposition
\begin{equation}
\label{eq:2:fluc-diss}
j^e_{x,x+1} = \nabla \varphi_x +{\mc L} h_x
\end{equation}
holds. Here $\theta_x$ denotes the shift by $x \in \ZZ^d$. Equation (\ref{eq:2:fluc-diss}) is called a {\textit{microscopic fluctuation-dissipation equation}}. Then, taking arbitrary large integer $\ell \ge 1$, by using a multi-scale analysis we replace the block averaged function $\tfrac{1}{2\ell +1} \sum_{|y-x| \le \ell} \nabla \varphi_y$ by $D({\mc E}^{\ell}_x) \nabla {\mc E}^\ell_x$ where the function  $D$ is identified to a diffusion coefficient which depends on the empirical energy ${\mc E}_x^\ell = \tfrac{1}{2\ell +1} \sum_{|y-x| \le \ell} {\mc E}_y$ in the mesoscopic box of length $(2\ell +1)$ centered around $x$. Intuitively, ${\mc L} h_x$ represents rapid fluctuation (integrated in time, it is a martingale) and the term $\nabla \varphi_x$ represents the dissipation. Gradient models are systems for which the current is equal to the gradient of a function ($h_x=0$ with the previous notations).

There are at least two reasons for which the problems listed above are difficult:
\begin{itemize}
\item The existence of a microscopic fluctuation-dissipation equation has been given for the first time for reversible systems. It has been extended to asymmetric systems satisfying a {\textit{sector condition}}. Roughly speaking this last condition means that the antisymmetric part of the generator is a bounded perturbation of the symmetric part of the generator {\footnote{The antisymmetric (resp. symmetric) part of the generator ${\mc L}$ is given by $\tfrac{{\mc L}-{\mc L}^*}{2}$ (resp. $\tfrac{{\mc L}+{\mc L}^*}{2}$) where ${\mc L}^*$ is the adjoint of ${\mc L}$ in $\bb L^2 (\mu)$, $\mu$ being any Gibbs grand canonical measure. For the models considered in this course, the antisymmetric part is ${\mc A}$ and due to the deterministic dynamics, the symmetric part is ${\mc S}$ and due to the noise. }}.  Later this condition has been relaxed into the so-called {\textit{graded sector condition}}: there exists a gradation of the space where the generator is defined and the asymmetric part is bounded by the symmetric part on each graded part (see \cite{KLO-book}, \cite{HIToth} and references therein). The Hamiltonian systems perturbed by a noise are non-reversible and since the noise (the symmetric part of the generator) is very degenerate, none of these conditions hold. 
\item The system evolves in a non compact space and one needs to show that energy cannot concentrate on a site. This technical problem turns out to be difficult since no general techniques are available. For deterministic nonlinear chains the bounds on the average energy moments are usually polynomial in the size $N$ of the system. Typically we need bounds of order one with respect to $N$.
\end{itemize}

\subsection{Anharmonic chain with velocity-flip noise}
\subsubsection{Linear response theory: Green-Kubo formula}
\label{subsec:lrGK}
The Green-Kubo formula is one of the most important formulas of non-equilibrium statistical mechanics. In the two problems mentioned in the introduction of the chapter (hydrodynamic limits and Fourier's law) the limiting objects are defined via some macroscopic coefficients which can be expressed by a Green-Kubo formula. The latter is a formal expression and showing that it is indeed well defined is a difficult problem. It is usually introduced in the context of the linear response theory that we describe below.

Consider a one dimensional unpinned chain of $N$ harmonic oscillators with forced boundary conditions and perturbed by the velocity-flip noise. The two external constant forces are denoted by  $\tau_\ell$ and $\tau_r$. Furthermore on the boundary particles $1$ and $N$, Langevin thermostats are acting at different temperature $T_\ell=\beta_{\ell}^{-1}$ and $T_r=\beta_r^{-1}$. The generator ${\mc L}_N$ of the dynamics is given by (\ref{eq:1:ness-gen}) and we denote the unique non-equilibrium stationary state by $\mu_{\rm{ss}}$. The expectation w.r.t. $\mu_{\rm{ss}}$ is denoted by $\langle \cdot \rangle_{ss}$.

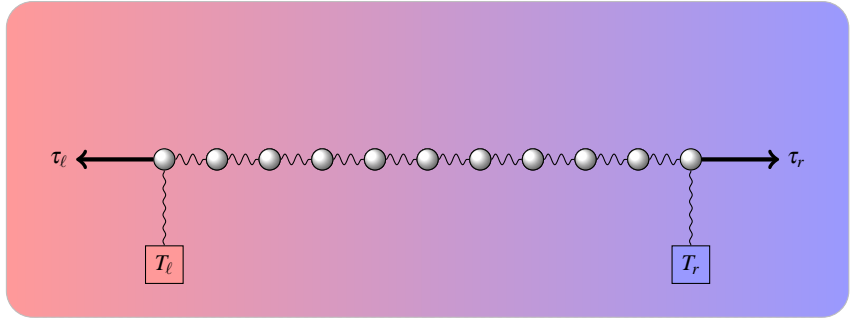
\begin{figure}[htbp]
\centering
\begin{tikzpicture}[scale=0.7]
\filldraw[gray!50,left color=red!40, right color=blue!40,rounded corners=10pt] 
(0,-3)--(13,-3)--(13,3)--(-0,3)--(-3,3)--(-3,-3)--(0,-3) ;

\tikzstyle{atome}=[draw,circle,ball color=white];
\foreach \y in {0,...,9}{
\node[atome] (\y) at (\y,0) {};
\node[atome] (\y+1) at (\y+1,0) {};
}
\foreach \y in {0,...,9}{
\draw [decorate,decoration={snake,amplitude=2pt,segment length=5pt}] (\y) -- (\y+1);
}
\node [draw,rectangle, minimum width=0.5cm, minimum height=0.5cm,fill=red!40] (bath-left) at (0,-2){$T_\ell$};
\node [draw,rectangle, minimum width=0.5cm, minimum height=0.5cm,fill=blue!40] (bath-right) at (10,-2){$T_r$};
\draw [decorate,decoration={snake,amplitude=0.5pt,segment length=5pt}] (bath-left) -- (0);
\draw [decorate,decoration={snake,amplitude=0.5pt,segment length=5pt}] (bath-right) -- (9+1);
\node (force-right) at (12,0){$\tau_r$};
\draw [->,ultra thick] (9+1)--(force-right);
\node (force-left) at (-2,0){$\tau_\ell$};
\draw [->,ultra thick] (0)--(force-left);
%\node [draw,rectangle, minimum width=0.5cm, minimum height=2cm,fill=black!70] (wall) at (-1,0){};
%\draw [decorate,decoration={snake,amplitude=2pt,segment length=5pt}] (wall) -- (0);
\end{tikzpicture}
\caption{The unpinned chain with boundary thermal reservoirs and forced boundary conditions.}
\end{figure}

The energy {\footnote{The definition of the energy is slightly modified w.r.t. (\ref{eq:1-energy-site-x}). It is more convenient since the energies are then independent  random variables in the Gibbs grand canonical ensemble.}} of atom $x$ is defined by
\begin{equation*}
{\mc E}_1 =\cfrac{p_1^2}{2}, \quad {\mc E}_x =\cfrac{p_x^2}{2} + V(r_x), \quad x=2, \ldots,N.
\end{equation*}

The local conservation of
energy is expressed by the microscopic continuity equation 
\begin{equation*}
{\mc L}_N ({\mc E}_x) = -\nabla j^e_{x-1,x}, \quad x=1, \ldots, N,
\end{equation*}
where the energy current $j^e_{x,x+1}$ from site $x$ to site $x+1$ is
given by 
\begin{equation}
\begin{split}
j^e_{0,1} &= -\tau_{\ell} p_1 + \gamma_\ell (T_{\ell} -p_1^2), \\
j^e_{N,N+1} &=-\tau_r p_N -\gamma_r (T_r -p_N^2), \\
j^e_{x,x+1} &= - p_x V' (r_{x+1}), \, x=1,\ldots,N-1.
\end{split}
\end{equation}

The energy current $j^e_{0,1}$ (and similarly for $j^e_{N,N+1}$) is
composed of two terms: the term $-\tau_\ell p_1$ corresponds to the
work done on the first particle by the linear force and the term $
\gamma_\ell (T_{\ell} -p_1^2)$ is the heat current due to the left
reservoir. 

Let $P_s$ be the velocity of the center
of mass of the system and $J_{s}$ be the average energy current, which are defined by 
$$P_s=\langle p_x \rangle_{ss} \quad \text{and}\qquad 
J_s=\langle j^e_{x,x+1} \rangle_{ss} .
$$ 

%Observe that we are in Lagrangian coordinates and that $J_s$ is really the interparticle exchange of energy, that does not take into account the trivial energy flux of the Eulerian coordinates due to the center of mass movement.

We have the simple relation between these two quantities
\begin{equation}
\label{eq:JV}
J_s = -\tau_\ell P_s + \gamma_\ell (T_\ell -\langle p_1^2 \rangle_{ss} ), \quad J_s = -\tau_r P_s - \gamma_r (T_r -\langle p_N^2 \rangle_{ss} ).
\end{equation}

The value of $P_s$ can be determined exactly and is independent of the nonlinearities present in the system. By writing that $\langle {\mc L}_N (p_x) \rangle_{ss} =0$  for any $x=1, \ldots,N$ we get that the tension profile, defined by $\tau_x = \langle V'(r_x) \rangle_{ss}$, satisfies 
\begin{equation*}
\begin{split}
\tau_2 -\tau_\ell =\gamma_\ell P_s, \quad \tau_r -\tau_N = \gamma_r P_s,\\
\tau_{x+1} -\tau_x = \gamma P_s, \quad x=2, \ldots, N-1.
\end{split}
\end{equation*}
We have then:
 
\begin{lemma}[\cite{BO2}]
\label{lem:velo}
The velocity $P_s$ of the center of mass is given by
\begin{equation}
\label{eq:V}
P_s = \cfrac{\tau_r -\tau_\ell}{\gamma (N-2) + \gamma_\ell + \gamma_r}
\end{equation}
and the tension profile is linear:
\begin{equation}
  \label{eq:1}
  \tau_x = \cfrac{\gamma (x-2) + \gamma_\ell}{\gamma (N-2) + \gamma_\ell +
    \gamma_r} (\tau_r - \tau_\ell) + \tau_\ell.
\end{equation}
Consequently
\begin{equation}
\lim_{n \to \infty} \tau_{[Nu]} = 
\tau_\ell + (\tau_r -\tau_\ell) u, \quad u \in [0,1].
\end{equation}
\end{lemma}

For purely deterministic chain ($\gamma=0$), the
velocity $P_{s}$ is of order $1$, while the tension profile is flat
at the value $\left({\gamma_\ell + \gamma_r}\right)^{-1}\left[ {\gamma_\ell \tau_r + \gamma_r \tau_\ell }\right]$. The first effect of the
noise is to make $P_s$ of order $N^{-1}$ and to give a nontrivial macroscopic tension profile.

%From the equality $J_s = \langle j^e_{0,1} \rangle_{ss} =\langle j^e_{N, N+1} \rangle_{ss} $ we get
%\begin{equation*}
%\gamma_\ell \langle p_1^2 \rangle_{ss} + \gamma_r \langle p_N^2 \rangle_{ss} = \gamma_\ell T_\ell + \gamma_r T_r + \cfrac{(\tau_r -\tau_\ell)^2}{ \gamma (N-2)  + \gamma_\ell +\gamma_r}.
%\end{equation*}
%This shows that there exists a constant $C$ depending on the parameters of the model ($T_\ell, T_r, \tau_\ell, \tau_r , \gamma_\ell, \gamma_r$) but not on $N$ such that
%\begin{equation}
%\label{eq:ubp}
%\langle p_1^2 \rangle_{ss} \le C, \quad \langle p_N^2 \rangle_{ss} \le C.
%\end{equation}

It is expected that there exists a positive constant $C$ independent of the size $N$ such that $\langle {\mc E}_x \rangle_{ss} \le C$ for any $x=1, \ldots,N$. Apart from the harmonic case we do not know how to prove such a bound. 

We shall denote by ${\tilde f_{ss}}$ the derivative of the stationary state $\mu_{ss}$ with respect to the
local Gibbs equilibrium state $\mu_{lg}$ defined by $\mu_{lg} (d\br, d\bp) = g(\br,\bp) d \br d \bp$ with
\begin{equation}
\label{eq:GLES-ness}
 g (\br,\bp) = \prod_{x=1}^N \frac{e^{-\beta_x (\mathcal E_x - \tau_x
    r_x)}}{Z(\tau_x\beta_x, \beta_x)}, 
\end{equation}
where $\beta_x=\beta_{\ell} + \frac{x}{N} (\beta_r -\beta_{\ell})$ and $\tau_x=\tau_{\ell} + \frac{x}{N} (\tau_r -\tau_{\ell})$. In the formula above we have introduced  $r_1=0$ to avoid annoying notations.

The function ${\tilde f}_{ss}$ is solution, in the sense of distributions, of the equation
\begin{equation}
 \label{eq:ss48}
{\tilde{\mc L}}_N^{*} \, \tilde f_{ss} = 0
\end{equation}
where ${\tilde{\mc L}}_N^*$ is the adjoint of ${\mc L}_N$ in ${\mathbb L}^2 (\mu_{lg})$. We assume that $T_r=T +\delta T, T_\ell= T$ and $\tau_r = \tau -\delta \tau, \tau_\ell= \tau$ with $\delta T, \delta \tau$ small. At first order in $\delta T$ and $\delta \tau$, we have
\begin{equation*}
  \begin{split}
{ \tilde {\mc L}}_N^{*} 
%& = -{\mathcal A}^{\tau_\ell, \tau_r}_{N} + \gamma {\mc S}_N + \gamma_\ell {\mc B}_{1, T} + \gamma_r {\mc B}_{n, T+ \delta T} \\
%    &- \frac {\delta T}{T^2 n} \sum_{x=1}^{N-1} \left(j^e_{x,x+1} + \tau p_x \right) -\frac{\delta \tau}{NT} \sum_{x=1}^{N-1} p_x + o(\delta T, \delta \tau)\\
%    &=  -{\mathcal A}^{\tau_\ell, \tau_r}_{N} +\gamma {\mc S}_N+ \gamma_\ell {\mc B}_{1, T_\ell} + \gamma_r {\mc B}_{n, T_r}\\
%    &-\frac {\delta T}{T^2 n} \sum_{x=1}^{N-1} \left(j^e_{x,x+1} + \tau p_x \right) -\frac{\delta \tau}{NT} \sum_{x=1}^{N-1} p_x  + \gamma_r \delta T \partial_{p_N}^2 + o(\delta T, \delta \tau)\\
    &=  {\mc L}_{N, {\rm{eq.}}}^{*} +  \gamma_r \delta T \partial_{p_N}^2 -\delta \tau \partial_{p_N} -\frac {\delta T}{T^2 N} \sum_{x=1}^{N-1} \left(j^e_{x,x+1} + \tau p_x \right) -\frac{\delta \tau}{NT} \sum_{x=1}^{N-1} p_x+ o(\delta T, \delta \tau)
  \end{split}
\end{equation*}
where $ {\mc L}_{ N, {\rm{eq.}}}^{*}=-{ \mathcal A}^{\tau,\tau}_{N} + \gamma {\mc S}_N + \gamma_\ell {\mc B}_{1, T} + \gamma_r {\mc B}_{N, T} $ is the adjoint in ${\mathbb L}^{2} (\mu^N_{\tau,T})$ of
\begin{equation}
  \label{eq:ss43}
   {\mc L}_{N,{\rm{eq.}}} ={ \mathcal A}_N^{\tau,\tau} + \gamma {\mc S}_N + \gamma_\ell {\mc B}_{1, T} + \gamma_r {\mc B}_{N, T} 
\end{equation}
and $\mu_{\tau,T}^N$ is the finite volume Gibbs grand canonical ensemble with tension $\tau$ and temperature $T$. We now expand $\tilde f_{ss}$ at the linear order in $\delta T$ and $\delta \tau$:
\begin{equation}
  \label{eq:ss55tau}
  \tilde f_{ss} = 1 + \tilde u \, \delta T +\tilde v \,  \delta \tau +
  o(\delta T, \delta \tau)
\end{equation}
and we get that $\tilde u$ and $\tilde v$ are solution of
\begin{equation}
  \label{eq:ss51}
  \begin{split}
    \mc L_{N,{\rm{eq.}}}^{*} \tilde u &= \frac 1{T^2 N} \sum_{x=1}^{N-1}
    \left(j^e_{x,x+1} + \tau p_x \right),%  = \frac 1{T^2 n}
%     \sum_{x=1}^{n-1} p_x\left(V'(r_{x+1}) - \tau\right)
    \\
     \mc L_{N, {\rm{eq.}}}^{*} \tilde v &= \frac{1}{NT}
     \sum_{x=1}^{N-1} p_x . 
  \end{split}
\end{equation}

It is clear that the function $h_x$ appearing in the microscopic fluctuation-dissipation equation (\ref{eq:2:fluc-diss}) is closely related (up to a time reversal) to the functions ${\tilde u}$, ${\tilde v}$, i.e. to the {\textit{first order correction to local equilibrium}}.

We can now compute the average energy current at the first order in $\delta T$ and $\delta \tau$ as $N \to \infty$ but we need to introduce some notation. We recall that the generator of the infinite dynamics is given by ${\mc L} ={\mc A} +\gamma {\mc S}$ where, for any $f \in C_0^1 ({\tilde \Omega})$,
\begin{equation*}
\begin{split}
&({\mc A} f)(\br, \bp) =  \sum_{x \in \ZZ} \left[ \left(p_x - p_{x-1}\right) \partial_{r_x} f + \left(V'(r_{x+1}) - V'(r_{x})\right) \partial_{p_{x}} f \right] (\br,\bp),\\
&({\mc S} f) (\br, \bp) = \frac{1}{2}\sum_{x \in \ZZ} \left[ f(\br, \bp^x) -f (\br, \bp) \right].
\end{split}
\end{equation*}
Let ${\bb H}:={\bb H}_{\tau,T}$ be the completion of the vector space of bounded local functions w.r.t. the semi-inner product $\ll \cdot, \cdot \gg$ defined for bounded local functions $f,g : {\tilde \Omega} \to \RR$, by 
\begin{equation}
\label{eq:scalarproduct>>}
\ll f, g \gg = \sum_{x \in \ZZ} \left\{ \mu_{\tau,T} (f \theta_x g) - \mu_{\tau,T} (f) \mu_{\tau,T} (g)\right\}.
\end{equation}
Observe that in ${\bb H}$ every constant $c\in \RR$ and discrete gradient $\psi=\theta_1 f - f$ is equal to zero since for any local bounded function $h$ we have $\ll c, h \gg =0$ and $\ll \psi, h \gg =0$. 

Assuming they exist let ${\tilde J}_s$ and ${\hat P_s}$ be the limiting average energy current and velocity:
\begin{equation}
\label{eq:PJJJ}
{\tilde J}_s = \lim_{N \to \infty} N \langle j^e_{0,1} \rangle_{ss}, \quad {\hat P}_{s} = \lim_{N \to \infty} N \langle p_0 \rangle_{ss},
\end{equation}
and define ${\hat J}_{s} = {\tilde J}_s + \tau {\hat P}_s$. We expect that as $N$ goes to infinity and, at first order in $\delta T$ and $\delta \tau$,
%\begin{equation}
%\begin{split}
%&{\hat J}_s = -{\kappa}^{e} \delta T - {\kappa}^{e,r} \delta \tau\\
%&{\hat V}_s =- {\kappa}^{r} \delta \tau -{\kappa}^{r,e} \delta T
%\end{split}
%\end{equation}
\begin{equation*}
\left(
\begin{array}{c}
{\hat J}_s\\
{\hat P}_s
\end{array}
\right)
= - \, 
\kappa (T,\tau) \, 
\left(
\begin{array}{c}
\delta T\\
\delta \tau
\end{array}
\right)
\end{equation*}
with 
\begin{equation}
\label{eq:1:cond-mat}
\kappa (T,\tau)=
\left(
\begin{array}{cc}
\kappa^e& \kappa^{e,r}\\
\kappa^{r,e}&\kappa^r
\end{array}
\right)
\end{equation}
the {\textit{ thermal conductivity}} matrix. Assume for simplicity that $N=2k$ is even. By (\ref{eq:ss55tau}) and (\ref{eq:ss51}), we get that
\begin{equation*}
\begin{split}
N \langle p_0 \rangle_{ss}=N \langle p_k \rangle_{ss}& = N \int p_k \,  {\tilde f}_{ss} \; d\mu_{lg}\\
&= N\,  \delta T \, \int  p_k \, {\tilde u} \; d\mu_{lg} \; + \; N \, \delta \tau \, \int  p_k \, {\tilde v} \; d\mu_{lg} \;+ \; o(\delta T, \delta \tau)\\
&= -\cfrac{\delta T}{T^2} \, \int  p_k \; (-{\mc L}_{N, {\rm{eq.}}}^*)^{-1} \Big( \sum_{x=1}^{N-1} (j_{x,x+1}^e + \tau p_x) \Big) \; d\mu_{lg} \\
&\;  - \; \cfrac{\delta \tau}{T}  \, \int  p_k \;   (-{\mc L}_{N, {\rm{eq.}}}^*)^{-1} \Big( \sum_{x=1}^{N-1} p_x \Big) \; d\mu_{lg} \;+ \; o(\delta T, \delta \tau).
\end{split}
\end{equation*}
Since $\tfrac{d\mu_{lg}}{d\mu_{\tau, T}^N}$ is equal to $1+ O(\delta T, \delta \tau)$, we can replace $\mu_{\lg}$ by $\mu_{\tau,T}^N$ in the last terms of the previous expression. Using that ${\mc L}_{N, {\rm{eq.}}}^*$ is the adjoint of ${\mc L}_{N, {\rm{eq.}}}$ in ${\bb L}^2 (\mu_{\tau,T}^N)$ and denoting by $\langle \cdot, \rangle_{\tau,T}$ the scalar product in ${\bb L}^2 (\mu_{\tau,T}^N)$, we obtain that
\begin{equation*}
\begin{split}
N \langle p_0 \rangle_{ss}&= -\cfrac{\delta T}{T^2} \, \left\langle (-{\mc L}_{N, {\rm{eq.}}})^{-1}  p_k \; , \; \sum_{x=1}^{N-1} (j_{x,x+1}^e + \tau p_x)  \; \right\rangle_{\tau,T} \\
&\;  - \; \cfrac{\delta \tau}{T}  \, \left\langle (-{\mc L}_{N, {\rm{eq.}}})^{-1}  p_k \; , \;  \sum_{x=1}^{N-1} p_x \; \right\rangle_{\tau,T} \;+ \; o(\delta T, \delta \tau)\\
&= -\cfrac{\delta T}{T^2} \, \left\langle  (-{\mc L}_{2k, {\rm{eq.}}})^{-1}  p_k \; , \; \sum_{y=-k+1}^{k-1} (j_{y+k,y+k+1}^e + \tau p_{y+k})  \; \right\rangle_{\tau,T} \\
&\;  - \; \cfrac{\delta \tau}{T}  \, \left\langle (-{\mc L}_{2k, {\rm{eq.}}})^{-1}  p_k  \; , \;   \sum_{y=-k+1}^{k-1} p_{y+k} \; \right\rangle_{\tau,T} \;+ \; o(\delta T, \delta \tau)
\end{split}
\end{equation*} 
In the first order terms of the previous expression we can recenter everything around $k$ by a translation of $-k$ and we get
\begin{equation*}
\begin{split}
N \langle p_0 \rangle_{ss}&= -\cfrac{\delta T}{T^2} \, \left\langle  (-{\mc L}_{\Lambda_k, {\rm{eq.}}})^{-1}  p_0  \; , \;  \sum_{y=-k+1}^{k-1} (j_{y,y+1}^e + \tau p_{y})  \; \right\rangle _{\tau,T} \\
&\;  - \; \cfrac{\delta \tau}{T}  \, \left\langle  (-{\mc L}_{\Lambda_k, {\rm{eq.}}})^{-1}  p_0  \; , \;    \sum_{y=-k+1}^{k-1} p_{y} \; \right\rangle_{\tau,T} \;+ \; o(\delta T, \delta \tau)
\end{split}
\end{equation*} 
where $\Lambda_k= \{-k+1, \ldots,k\}$ and 
\begin{equation*}
{\mc L}_{\Lambda_k, {\rm{eq.}}}={\mc A}_{\Lambda_k}^{\tau,\tau} +\gamma {\mc S}_{\Lambda_k} + \gamma_\ell {\mc B}_{-k,T} +{\gamma_r} {\mc B}_{k, T}
\end{equation*}
with
\begin{equation*}
\begin{split}
    {\mathcal A }^{\tau, \tau}_{\Lambda_k} = 
    \sum_{x=-k+2}^{k} \left(p_{x} - p_{x-1}\right) \partial_{r_x} +
        \sum_{x=-k+2}^{k-1}\left(V'(r_{x+1}) - V'(r_{x})\right)
      \partial_{p_{x}}\\
    - \left(\tau- V'(r_{-k+2})\right) \partial_{p_{-k+1}} 
    + \left(\tau - V'(r_{k})\right) \partial_{p_{k}}
  \end{split}
\end{equation*}
and
\begin{equation*}
({\mc S}_{\Lambda_k} f ) (\br,\bp) = \frac{1}{2} \sum_{x=-k+2}^{k-1} \left(f(\br,\bp^{x}) - f(\br,\bp)\right).
\end{equation*}

A similar formula can be obtained for $N\langle j_{0,1}^e \rangle_{ss}$. As $k \to \infty$, the finite volume Gibbs measure converges to the infinite volume Gibbs measure. Moreover, we expect that since $k \to \infty$ the effect of the boundary operators ${\mc B}_{\pm k,T}$ around the site $0$ disappears so that $(-{\mc L}_{\Lambda_k, {\rm{eq.}}})^{-1} p_0$ converges to $(-{\mc L})^{-1} p_0$. Therefore, in the thermodynamic limit $N \to \infty$ ({\textit{i.e.}} $k \to \infty$), the transport coefficients are given by the Green-Kubo formulas
\begin{equation}
\label{eq:1:ke}
\begin{split}
\kappa^e = T^{-2} \ll j^e_{0,1} +\tau p_0\,,\, (-{\mc L})^{-1}\, (j^e_{0,1} +\tau p_0) \gg,\\
\kappa^{e,r} =T^{-1} \ll p_0\,,\, (-{\mc L})^{-1} \,( j^e_{0,1}+\tau p_0) \gg, 
\end{split}
\end{equation} 
and
\begin{equation}
\label{eq:1:kr}
\begin{split}
\kappa^r =T^{-1} \ll p_{0} \,,\,(- {\mc L})^{-1}\, (p_0) \gg,\\
\kappa^{r,e} =T^{-2} \ll j^e_{0,1} + \tau p_0\,,\, (-{\mc L})^{-1} \,( p_0) \gg. 
\end{split}
\end{equation} 

The argument above is formal. In fact even proving the existence of the transport coefficients defined by (\ref{eq:1:ke}), (\ref{eq:1:kr}) is a non-trivial task. The existence of ${\hat P}_s$ defined by the second limit in (\ref{eq:PJJJ}) can be made rigorous since we have the exact expression of $P_s$. From Lemma \ref{lem:velo}, we have, even for $\delta \tau ,\delta T$ that are not small,
$${\hat P}_s = -\cfrac{\delta \tau}{\gamma}.$$
On the other hand we show in Theorem \ref{th:GK} that the quantities $\kappa^r, \kappa^{r,e}$, formally given by  (\ref{eq:1:kr}), can be defined in a slightly different but rigorous way, and are then equal to
\begin{equation}
\label{eq:krr}
{\kappa^r} = \gamma^{-1}, \quad {\kappa}^{r,e} =0.
\end{equation}  
Thus we can rigorously establish the validity of the linear response theory for the velocity ${\hat P}_s$.

\subsubsection{Existence of the Green-Kubo formula}
\label{subsec:GK0}

One of the main results of \cite{BO2} is the existence of the Green-Kubo formula for the conductivity matrix. Let ${\bb H}^a$ (resp. ${\bb H}^s$) be the set of functions $f:{\tilde \Omega} \to \RR$ antisymmetric (resp. symmetric) in $\bp$, i.e. $f(\br,\bp)=-f(\br,-\bp)$ (resp. $f(\br,\bp)=f(\br,-\bp)$)  for every configuration $(\br,\bp) \in {\tilde \Omega}$. For example, the functions $j^e_{0,1}$, $p_0$ and every linear combination of them are antisymmetric in $\bp$.

\begin{theorem}[\cite{BO2}, \cite{BHLLO}]
\label{th:GK}
Let $f,g \in {\bb H}^a$. The limit 
\begin{equation*}
\sigma(f, g) =\lim_{\substack{z \to 0\\ z>0}} \ll f\, ,\, (z - {\mc L})^{-1} \, g \gg
\end{equation*}
exists and $\sigma(f,g)=\sigma(g,f)$. Therefore, the conductivity matrix $\kappa (T,\tau)$ is well defined in the following sense: the limits 
\begin{equation}
\begin{split}
\kappa^e =\lim_{\substack{z \to 0\\ z>0}}  T^{-2} \ll j^e_{0,1} +\tau p_0\,,\, (z-{\mc L})^{-1}\, (j^e_{0,1} +\tau p_0) \gg,\\
\kappa^{e,r} =\lim_{\substack{z \to 0\\ z>0}} T^{-1} \ll p_0\,,\, (z-{\mc L})^{-1} \,( j^e_{0,1}+\tau p_0) \gg,\\
\kappa^r =\lim_{\substack{z \to 0\\ z>0}}  T^{-1} \ll p_{0} \,,\,(z- {\mc L})^{-1}\, (p_0) \gg=\gamma^{-1},\\
\kappa^{r,e} =\lim_{\substack{z \to 0\\ z>0}}  T^{-2} \ll j^e_{0,1} + \tau p_0\,,\, (z-{\mc L})^{-1} \,( p_0) \gg
\end{split}
\end{equation} 
exist and are finite. Moreover Onsager's relation $\kappa^{e,r}=\kappa^{r,e} (=0) $ holds.
\end{theorem}

We have a nice thermodynamical consequence of the previous result. If $\delta T$ and $\delta \tau$ are small and of the same order, the system cannot be used as a refrigerator or a
boiler: at the first order, a gradient of tension does
not contribute to the heat current ${\hat J}_s$. 
The argument above says nothing about the possibility to
realize a heater or a refrigerator if $\delta \tau $ is not of the
same order as $\delta T$. For the harmonic chain, we will see that it is possible to get a heater if $\delta
\tau $ is of order $\sqrt {\delta T}$.     

\begin{remark}

\begin{enumerate}
\item The existence of the Green-Kubo formula is also valid for a pinned or unpinned chain in any dimension. 
\item Observe that with respect to the establishment of a microscopic fluctuation-dissipation equation (\ref{eq:2:fluc-diss}) the computation of the Green-Kubo formula is less demanding since only the knowledge of $\sum_x h_x$ is necessary. 
\end{enumerate}
\end{remark}

The proof of Theorem \ref{th:GK} is based on functional analysis arguments. The first main observation is that there exists a spectral gap for the operator ${\mc S}$ restricted to the space ${\bb H}^a$. 

\begin{lemma}
 \label{lem:sg}
The noise operator ${\mc S}$ lets ${\bb H}^a$ and ${\bb H}^s$ invariant. For any local function $f \in {\bb H}^a$ we have that
\begin{equation}
\label{eq:1:sg1}
 \ll f, f \gg \; \le \; \ll f, -{\mc S} f \gg.    
\end{equation}
Moreover, for any local function $f \in {\bb H}^a$, there exists a local function $h \in {\bb H}^a$ such that
$${\mc S}h =f.$$ 
\end{lemma} 

\begin{proof}
Since the Gibbs states are Gaussian states in the $p_x$'s it is convenient to decompose the operator ${\mc S}$ (which acts only on the $p_x$'s)  in the orthogonal basis of Hermite polynomials. The the lemma follows easily.
\cqfd\end{proof}

\begin{proof}[Theorem \ref{th:GK}] 

We observe first that ${\bb H}^a$ and ${\bb H}^s$ are orthogonal Hilbert spaces such that ${\bb H} ={\bb H}^a \oplus {\bb H}^s$. It is also convenient to define the following semi-inner product
\[
\ll u,w\gg_1=\ll u, (-\mc S) w\gg.
\]
Let $\bb H^1$ be the associated Hilbert space. We also define the Hilbert space $\bb H^{-1}$ via the duality
given by the ${\bb H}$ norm, that is  
\[
\|u\|_{-1}^2=\sup_{w} \{ \, 2 \ll u,w\gg -\ll w,w\gg_1 \, \} 
\]
where the supremum is taken over local bounded functions $w$. By Lemma \ref{lem:sg} we have that ${\bb H}^a\subset {\bb H}^{-1}$. Thus $g \in {\bb H}^{-1}$. 

Let $w_z$ be the solution of the resolvent equation $(z- {\mc L}) w_z = g$. We have to show that $\ll f, w_z \gg$ converges as $z$ goes to $0$. We decompose $w_z$ into $w_z = w^-_z +w^+_z$, $w^-_z \in {\bb H}^a$ and $w^+_z \in {\bb H}^s$. Since ${\bb H}^a$ is orthogonal to ${\bb H}^s$ and $f \in {\bb H}^a$ we have $\ll f, w_z \gg = \ll f, w_z^- \gg$. It is thus sufficient to prove that $(w_z^-)_{z>0}$ converges weakly in ${\bb H}$ as $z \to 0$.

Since ${\mc A}$ inverts the parity and $\mc S$ preserves it and ${\bb H}^a \oplus {\bb H}^s =\bb H$ and $g \in {\bb H}^a$, we have, for any $\mu,\nu>0$, 
\begin{equation}
\label{eq:cfgst}
\begin{split}
&\nu w_{\nu}^+ - {\mc A} w_{\nu}^- -\gamma {\mc S} w_{\nu}^+ =0,\\
&\mu w_{\mu}^{-} - {\mc A} w_\mu^+ - \gamma {\mc S} w_{\mu}^- =g.
\end{split}
\end{equation}
Taking the scalar product with $w_\mu^+$ (resp. $w_{\nu}^-$) on both sides of the first (resp. second) equation of (\ref{eq:cfgst}), we get
\begin{equation}
\label{eq:oldsplit}
\begin{split}
&\nu\ll w_{\mu}^+,w_{\nu}^+\gg -\ll w_{\mu}^+, {\mc A} w_{\nu}^-\gg+\gamma\ll w_{\mu}^+,w_{\nu}^+\gg_1=0,\\
&\mu\ll w_{\nu}^-,w_{\mu}^-\gg -\ll w_{\nu}^-, {\mc A} w_{\mu}^+\gg+\gamma \ll w_{\mu}^-,w_{\nu}^-\gg_1=\ll w_{\nu},g \gg.
\end{split}
\end{equation}
Summing the above equations we have
\begin{equation}
\label{eq:baseone}
\nu\ll w_{\mu}^+,w_{\nu}^+\gg+\mu\ll w_{\nu}^-,w_{\mu}^-\gg+\gamma\ll w_{\mu},w_{\nu}\gg_1=\ll w_{\nu},g\gg
\end{equation}
Putting $\mu=\nu$ we get
\[
\nu\ll w_{\nu} , w_\nu \gg +\gamma \ll w_{\nu},w_{\nu}\gg_1\leq \| w_{\nu}\|_1\| g\|_{-1}.
\]
Hence $(w_\nu)_{\nu >0}$ is uniformly bounded in $\bb H^1$ and by the spectral gap property so is $( w_\nu^- )_{\nu >0}$ in $\bb H$. Moreover, $( \nu w_{\nu} )_{\nu>0}$ converges strongly to $0$ in ${\bb H}$ as $\nu \to 0$. We can then extract weakly convergent subsequences. Taking first the limit, in \eqref{eq:baseone}, $\nu\to 0$ and then $\mu\to 0$ along one such subsequence (converging to $w_*$) we have 
\[
\gamma \ll w_{*},w_{*}\gg_1=\ll w_{*},g\gg.
\]
%Taking again the limit along such subsequence, with $\mu=\nu$, we have then
%\begin{equation}
%\label{eq:smallL2norm}
%\lim_{\nu\to 0} \nu\ll w_{\nu} , w_\nu \gg=0.
%\end{equation}
Next, taking the limit along different weakly convergent subsequences (let $w^*$ be the other limit) we have
\[
\gamma \ll w_{*},w^{*}\gg_1=\ll w^{*},g\gg
\]
and, exchanging the role of the two sequences
\[
2\gamma \ll w_{*},w^{*}\gg_1=\ll w_{*},g\gg +\ll w^{*},g\gg =\gamma \ll w_{*},w_{*}\gg_1+\gamma\ll w^{*},w^{*}\gg_1
\]
which implies $w_*=w^*$, that is all the subsequences have the same limit. Thus $(w_\nu)_{\nu>0}$ converges weakly in ${\bb H}^1$ as well as $(w_{\nu}^-)_{\nu>0}$ in $\bb H$ by Lemma \ref{lem:sg}. 
\cqfd\end{proof}

In the harmonic case, $V ( r )=r^2/2$, much more is known. Indeed one easily checks that the exact microscopic fluctuation-dissipation equation (\ref{eq:2:fluc-diss}) holds with
\begin{equation}
\label{eq:fluct-diss-eq-harm}
h_x =\frac{1}{2\gamma} r_{x+1} (p_x +p_{x+1}) - \frac{r_{x+1}^2}{4}, \quad \varphi_x =- \frac{1}{2\gamma} (r_x r_{x+1} + p_x^2).
\end{equation}

It follows that we can compute explicitly $(z-{\mc L})^{-1}j^{e}_{0,1}$ and obtain that the value of the conductivity matrix:
\begin{equation*}
\kappa (\tau, T)= 
\left(
\begin{array}{cc}
\frac{1}{2 \gamma} & 0\\
0& \frac{1}{\gamma}
\end{array}
\right).
\end{equation*} 
This value will be recovered by considering the hydrodynamic limits of the system (Theorem \ref{thm:Sim1}) and also by establishing the validity of Fourier's law (see Theorem \ref{thm:BO1}).

\subsubsection{Expansion of the Green-Kubo formula in the weak coupling limit}

In the previous subsection we proved the existence of the Green-Kubo formula showing that the transport coefficient is well defined if some noise is added to the deterministic dynamics. We are now interested in the behavior of the Green-Kubo formula as the noise vanishes. We investigate this question in the weak coupling limit, i.e. assuming that the interaction potential is of the form $\ve V$ where $\ve \ll 1$ is the (small) coupling parameter. For notational simplicity we consider the one dimensional infinite pinned system but the arguments given below are easily generalized to the (pinned or unpinned)  $d \ge 1$-dimensional case {{\footnote{If $W=0$ the variables $q_x$ have to take values in a compact manifold.}}}. The expansion presented in this section is formal but we will precise at the end of the section what has been rigorously proved. In order to emphasize the dependence of $\kappa^e$ (denoted in the sequel by $\kappa$) in the coupling parameter $\ve$ and the noise intensity $\gamma$, we denote $\kappa$ by $\kappa(\ve, \gamma)$. Here we propose a formal expansion of the conductivity $\kappa$ in the form 
\begin{equation}
\label{eq:expK}
\kappa (\ve,\gamma)= \sum_{n \ge 2} \kappa_n (\gamma) \ve^n.
\end{equation}
Then we study rigorously the first term of this expansion $\kappa_2 (\gamma)$. It is intuitively clear that the expansion starts from $\ve^2$ since the Green-Kubo formula is a quadratic function of the energy current and that the latter is of order $\ve$ (see (\ref{current:  general formula})). 

When the system is uncoupled ($\ve=0$), the dynamics is given by the generator ${\mc L}_0={\mc A}_0 +\gamma S$ with ${\mc S}$ the flip noise defined by (\ref{eq:cons-noise-flip}) and 
\begin{equation*}
{\mc A}_0 =\sum_{x \in \ZZ} p_x \partial_{q_x} - W' (q_x) \partial_{p_x}. 
\end{equation*}
When $\ve >0$, the generator of the coupled dynamics is denoted by 
\begin{equation}
\label{eq:geneve}
  {\mc L}_{\ve} = {\mc L}_0 + \ve {\mc G}
\end{equation}
where 
$$ {\mc G} =\sum_{x \in \ZZ} \,  V'(q_x -q_{x-1})(\partial_{p_{x-1}} - \partial_{p_x}).$$

The energy of each cell, which is the sum of the internal energy and
of the interaction energy, is defined by 
\begin{equation}
  \label{eq:11}
  {\mc E}_x^\ve = {\mc E}_x + \frac \ve 2 \left(V(q_{x+1} - q_{x}) + V(q_{x} -
    q_{x-1})\right), \quad {\mc E}_x = \frac{p_x^2}{2} + W(q_x).
\end{equation}
Observe that ${\mc E}_x ={\mc E}_x^0$ is the energy of the
isolated system $x$. The dynamics generated by ${\mc L}_0$ preserves all the
  individual energies ${\mc E}_x$. The dynamics generated by ${\mc L}_\ve$ conserves the total energy. The
corresponding energy currents $\ve j_{x,x+1}$, 
defined by the local conservation law
$$
{\mc  L}_{\ve} {\mc E}_x^\ve = \ve \left(j_{x-1,x}-j_{x,x+1}\right)
$$
are given by
\begin{equation}
\label{current:  general formula}
 \ve\, j_{x,x+1} =  -\frac {\ve}{2}\,  (p_x + p_{x+1}) \cdot V'(q_{x+1} -q_{x}).
\end{equation}

Let us denote by $\mu_{\beta,\ve}=\langle \cdot \rangle_{\beta,\ve}$
the canonical Gibbs measure at temperature $\beta^{-1}>0$ defined by the Dobrushin-Lanford-Ruelle equations, which of course depends on the interaction $\ve V$. We shall assume in all the cases considered
that $\mu_{\beta,\ve}$ 
is analytical in $\ve$ for sufficiently small $\ve$ (when applied to
local functions). In particular we assume that the potentials $V$ and
$W$ are such that the Gibbs state is unique and has spatial exponential decay of correlations (this holds under great general conditions on $V$ and $W$, see \cite{GEO}).

In order to emphasize the dependence in $\ve$ we reintroduce some notation. For any given local functions $f,g$, define the semi-inner product
\begin{equation}
 \label{eq:sipt}
  \ll f,g \gg_{\beta, \ve} \ =\ \sum_{x\in \ZZ} [\langle\theta_x f, g\rangle_{\beta,\ve} -
  \langle f\rangle_{\beta, \ve} \langle g\rangle_{\beta,\ve}]. 
\end{equation}
We recall that  $\theta_x$ is the shift operator by $x$. The sum is finite in the case $\ve =0$, and converges for $\varepsilon>0$ thanks to the exponential decay of the spatial correlations.  Denote by $\bb H_{\ve} = {\bb L}^2( \ll\cdot ,\cdot \gg_{\beta,\varepsilon} )$ the corresponding closure. We define the subspace of antisymmetric functions in the velocities
\begin{equation}
 \label{eq:antisymm}
 \bb H_\ve^a = \left\{ f \in \bb H_\ve: f({\bf q},-{\bf p}) = -f({\bf q},{\bf p})\right\}.
\end{equation}
Similarly we define the subspace of symmetric functions in  $\bf p$ as
$\bb H_\ve^s$.  On local functions this decomposition of a function into symmetric and antisymmetric parts is
independent of $\ve$. Let us denote by $\mc P_{\ve}^a$ and $\mc P_{\ve}^s$
the corresponding orthogonal projections, whose definition in fact
does not depend on $\ve$. Therefore we sometimes omit the index $\ve$ in the notation. Finally, for any function $f\in {\bb L}^2(\mu_{\beta,\ve})$,  define
$$
(\Pi_\ve f) ({\mathbb {\mc E}}) =  \mu_{\beta,\ve} (f|{\mathbb{\mc E}}),\quad Q_\ve =\rm{Id}-\Pi_\ve
$$
where ${\mathbb{\mc E}}:=\{ {\mc E}_x \, ; \, x \in \ZZ \}$.
According to Theorem \ref{th:GK} the conductivity is defined by 
\begin{equation}
\label{eq:ghj1}
\kappa (\ve, \gamma)= \ve^2 \lim_{\nu \to 0} \ll j_{0,1} \, , \, (\nu -{\mc L}_{\ve})^{-1} j_{0,1} \gg_{\beta,\ve}. 
\end{equation}
It turns out that, for calculating the terms in the expansion (\ref{eq:expK}), it is convenient to choose $\nu = \ve^2\lambda$ in (\ref{eq:ghj1}), for a $\lambda >0$, and solve the resolvent equation 
\begin{equation}\label{eq:poisson}
(\lambda\ve^2-\mc L_\ve) u_{\lambda,\ve} = \ve j_{0,1}
\end{equation}
for the unknown function $u_{\lambda,\ve}$. The factor $\ve^2$ is the natural scaling in view of
the subsequent computations. 
%We have already remarked that $\mc P^a_\ve u_{\lambda,\ve}$ converges in
%$\bb H_\ve$, for any fixed $\ve>0$.
% Unfortunately $u_{\lambda,\ve}$ is not a local function, so the
% decomposition in $\mc H^a_\ve \otimes \mc H^s_\ve$ depends on $\ve$.
We assume that a solution of \eqref{eq:poisson} is in the form 
\begin{equation}
  \label{eq:6}
 u_{\lambda,\ve} =\sum_{n \ge 0} U_{\lambda,n} \ve^n= \sum_{n \ge 0} (v_{\lambda,n} + w_{\lambda,n}) \ve^n,
\end{equation}
where $\Pi v_{\lambda,n} =Q w_{\lambda,n}=0$, i.e. $w_{\lambda,n} = \Pi U_{\lambda,n}$ and $v_{\lambda,n} =Q U_{\lambda,n}$. Here $\Pi=\Pi_0$ and $Q=
Q_0$ refer to the uncoupled measure $\mu_{\beta,0}${\footnote{The reason to use the orthogonal  decomposition of $U_{\lambda,n}= v_{\lambda,n} + w_{\lambda,n}$ is that at some point we will have to consider, for a given function $f$, the solution $h$ to the Poisson equation  ${\mc L}_0 h= f $ . The minimal requirement for the existence of $h$ is that $\Pi f =0$.}}. Given such an expression we can, in principle, use it in \eqref{eq:ghj1} to write 
\begin{equation}
\label{eq:kcoef}
\begin{split}
\kappa(\ve, \gamma)&=\lim_{\lambda\to 0}\sum_{n\geq0}\ve^{n+1}\ll j_{0,1},v_{\lambda,n}+w_{\lambda,n}\gg_{\beta,\ve}\\
&=\sum_{n\geq 1}\lim_{\lambda\to 0}\ve^{n}\ll j_{0,1},v_{\lambda,n-1}\gg_{\beta,\ve}
\end{split}
\end{equation}
where we have used the fact that that $\ll j_{0,1},  w_{\lambda,\ve} \gg_{\beta,\ve}=0$ and we have, 
arbitrarily, exchanged the limit with the sum. Note that this is not yet
of the type \eqref{eq:expK} since the terms in the expansion depend
themselves on $\ve$. To identify the coefficients $\kappa_n$ we
would need to expand in $\ve$ also the expectations. This is not
obvious since the functions $v_{\lambda,n}$ are non local.  

Let us consider the operator $\mf L =   \Pi \mc G \mc P^a (-\mc L_0)^{-1} \mc G\Pi.$ We show below that the operator $\mf L$ is a generator of a Markov process so that
$(\lambda- {\mf L})^{-1}$ is well defined for $\lambda>0$. Pluging (\ref{eq:6}) in (\ref{eq:poisson}) we obtain the following hierarchy
\begin{equation} 
\label{eq:w0}
\begin{split}
&v_{\lambda,0}=0,\\
&w_{\lambda,0}=(\lambda-\mf L)^{-1}\Pi \mc G \mc P^a (-\mc L_0)^{-1}j_{0,1},\\
&v_{\lambda,1}= (-\mc L_0)^{-1}\left[j_{0,1}+\mc G w_{\lambda,0}\right],\\
&w_{\lambda,n}=(\lambda-\mf L)^{-1}\Pi \mc G {\mc P}^a (-\mc L_0)^{-1}\left[-\lambda
  v_{\lambda, n-1} + Q \mc G v_{\lambda,n}\right], \qquad n\ge 1\\
&v_{\lambda,n+1}= (-\mc L_0)^{-1}\left[-\lambda v_{\lambda, n-1} + \mc G
  w_{\lambda,n} + Q\mc G v_{\lambda,n}\right], \qquad n\ge 1.
\end{split}
\end{equation}
Observe that in the previous equations the (formal) operator $(-{\mc L}_0)^{-1}$ is always applied to functions $f$ such that $\Pi f =0$ (this is the minimal requirement to have consistent equations). This is however not sufficient to make sense of the functions $v_{\lambda, n}$ and $w_{\lambda,n}$. Nevertheless, by using an argument similar to the one given in Theorem \ref{th:GK}, we have that the local operator ${\mc T}_0$ on ${\bb H}_0^a$ defined by 
$${\mc T}_0 f = \lim_{\nu \to 0} {\mc P}^a (\nu -{\mc L}_0)^{-1} f, \quad f \in \bb H_{0}^a,$$ is well defined. Therefore, it is possible to make sense, as a distribution, of 
\begin{equation}
\label{eq:alpha01}
\alpha_{01}= \Pi \mc G \mc P^a (-\mc L_0)^{-1}j_{0,1}:=\Pi \mc G {\mc T}_0 j_{0,1}.
\end{equation}
Nevertheless, the function $w_{\lambda,0}$ is still not well defined since we are not sure that $ {\mc T}_0 j_{0,1}$ is in the domain of ${\mc G}$. 

Even if the previous computations are formal a remarkable fact is that the operator ${\mf L}$, when applied to functions of the internal energies, coincides with the Markov generator ${\mf L}_{\rm{GL}}$ of a reversible Ginzburg-Landau dynamics on the internal energies. Let us denote by $\rho_{\beta}$ the distribution of the internal energies ${\mathbf{\mc E}}=\{ {\mc E}_x \, ; \, x \in \ZZ \}$ under the Gibbs measure $\mu_{\beta,0}$. It can be written in the form  
$$
d\rho_{\beta} ({\mathbf{\mc E}})= \prod_{x \in \ZZ} Z_\beta^{-1} \exp( - \beta {\mc E}_x -U({\mc E}_x)) d{\mc E}_x
$$
% \marginnotes{added normalization constant}
for a suitable function $U$. We denote the formal sum $\sum_x U({\mc E}_x)$
by ${\mc U}:={\mc U} ({\mathbf{\mc E}})$. We denote also, for a given value of the
internal energy ${\tilde {\mc E}}_x$ in the cell $x$,  by $\nu_{{\tilde {\mc E}}_x}^{x}$ the microcanonical probability measure in the cell $x$. i.e. the uniform probability measure on the manifold
$$
\Sigma_{{\tilde {\mc E}}_x}:=\{ (q_x,p_x) \in \Omega \, ; \, {\mc E}_x (q_x,p_x)
\, =\, {\tilde {\mc E}}_x \}.
$$ 
Then, the generator ${\mf L}_{\rm{GL}}$ is given by
 \begin{equation}
 \label{eq:12}
 {\mf L}_{\rm{GL}} = \sum_x e^{\mc U} (\partial_{{\mc E}_{x+1}} - \partial_{{\mc E}_x} )
 \left[e^{-\mc U} \gamma^2 ({\mc E}_x, {\mc E}_{x+1}) 
(\partial_{{\mc E}_{x+1}}  - \partial_{{\mc E}_x} ) \right],
  \end{equation}
where
\begin{equation}
\label{gamma square}
 \gamma^2 ({\tilde {\mc E}}_0, {\tilde {\mc E} }_{1})= \int_{\Sigma_{{\tilde {\mc E}}_0} \times
    \Sigma_{{\tilde {\mc E}}_{1}}} \big(  j_{0,1}\;  {\mc T}_0\,   j_{0,1} \big) \; d\nu_{{\tilde {\mc E}}_0}^0 d\nu_{{\tilde {\mc E}}_{1}}^1. 
\end{equation}
The operator ${\mf L}_{\rm{GL}} $ is well defined only if the function $\gamma^2$ has some regularity properties, that are actually proven in specific examples \cite{LO-0, DL}. We can show that the Dirichlet forms {\footnote{They are well defined even if $\gamma^2$ is not regular.}} associated to $\mf L$ and ${\mf L}_{{\rm{GL}}}$ coincide. Then in the cases where $\gamma^2$ is proven to be smooth \eqref{eq:12} is well defined and $\mf L={\mf L}_{{\rm{GL}}}$. 

\begin{proposition}[\cite{BHLLO}]
\label{prop:Lweak}
For each local smooth functions $f,g$ of the internal energies only we have
\begin{equation}\label{dir-form}
\ll g ,(-\mf L) f \gg_{\beta,0} = \ll g, (-{\mf L}_{\rm{GL}}) f \gg_{\beta,0}. 
  \end{equation}
\end{proposition}

The operator ${\mf L}_{\rm{GL}}$ is the generator of a Ginzburg-Landau dynamics
which is reversible with respect to $\rho_{\beta}$, for any $\beta>0$. It is conservative
in the energy $\sum_x  {\mc E}_x$ and the corresponding currents are given by $\theta_x \alpha_{0,1}$ where $\alpha_{0,1}$ has been defined in (\ref{eq:alpha01}). The corresponding finite size
dynamics appears in \cite{LO-0,DolgoL} as the weak coupling limit of a
finite number $N$ (fixed) of cells weakly coupled by  a potential
$\varepsilon V$ in the limit $\varepsilon \to 0$ when time $t$ is rescaled
as $t \ve^{-2}$. Moreover, the hydrodynamic limit of the Ginzburg-Landau dynamics is then given (in the diffusive time scale $tN^2$, $N \to + \infty$), by a heat equation with diffusion coefficient which coincides with $\kappa_2$ as given by \eqref{eq:gke} below (\cite{Var0}). This is summarized in Figure \ref{fig:wclimit000}. 

%\resizebox{13.0cm}{!}{
\begin{center}
\begin{figure}
\begin{tikzpicture}[node distance = 4cm, auto, scale=0.75]
 \node[draw, rectangle,fill=blue!25] (MS) at (-5,5) {$N$ cells coupled by $\ve V$}; 
\node[draw,rectangle,fill=blue!25] (HE) at (5,5) {$\partial_t T = \nabla ({\kappa} (\ve,\gamma) \nabla T)$};
\draw[->,>=latex,dashed] (MS) -- (HE) node[sloped,midway]{$tN^2$, $N \to \infty$, $\varepsilon \sim 1$};
\node[draw,rectangle,fill=blue!25] (GL) at (-5,0) {Ginzburg-Landau dynamics for $N$ particles};
 \draw[->,>=latex] (MS)  -- (GL) node[midway]{${\varepsilon}^{-2} t$, $\varepsilon \to 0$}; 
\node[draw,rectangle,fill=blue!25] (HE2) at (5,4) {$\partial_t T = \nabla ({\kappa}_2 (\gamma) \nabla T)$};
 \draw[->,>=latex] (GL) -- (HE2) node[sloped, pos=0.7]{$tN^2$, $N \to \infty$}; 
%node[sloped,draw, rectangle,fill=gray!25, scale=0.5,midway]{Linear response};
%\node[draw, rectangle,fill=gray!25] at (0,15)  {Fluctuation-dissipation theorem}; 
\end{tikzpicture}
\caption{The relation between the hydrodynamic limit, the weak coupling limit and the Green-Kubo expansion. The dotted arrow (hydrodynamic limits in the diffusive time scale) has not been proved. The weak coupling limit (vertical arrow) has been proved  in \cite{LO-0} (see also \cite{DolgoL}) and the diagonal arrow (hydrodynamic limits for a Ginzburg-Landau dynamics) has been obtained in \cite{Var0}  in some cases which however do not cover our cases. In \cite{BHLLO} it is argued that $\kappa (\ve, \gamma) \sim \ve^2 \kappa_2 (\gamma)$ as $\ve \to 0$.}
\label{fig:wclimit000}
\end{figure}
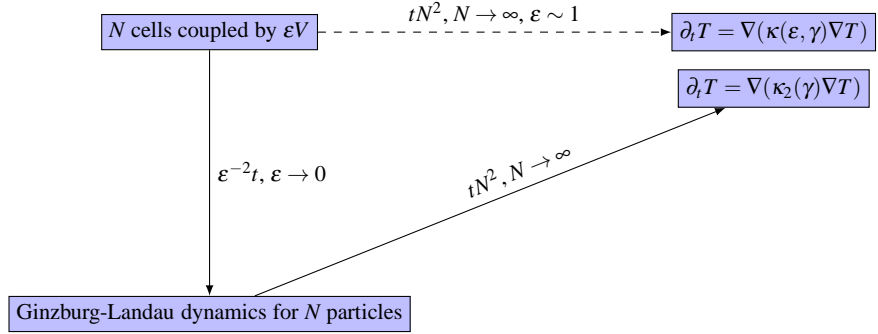
\end{center}
%}
According to the previous expansion it makes sense to define $\kappa_2 (\gamma)$ by
\begin{equation}
\label{eq:k2345}
\kappa_2 (\gamma)= \lim_{\ve \to 0} \lim_{\lambda \to 0} \left\{  \ll j_{0,1}, {\mc T}_0 j_{0,1} \gg_{\beta,\ve} + \ll j_{0,1}, \mc T_0 {\mc G} w_{\lambda,0} \gg_{\beta,\ve} \right\}
\end{equation}
if the limits exist. In fact, a priori, it is not even clear that the term ${\mc T}_0 {\mc G} w_{\lambda,0}$ makes sense since $w_{\lambda,0}$ is not well defined. In \cite{BHLLO} we argue that 
\begin{equation} 
\label{eq:gke}
\begin{split}
\kappa_2 (\gamma)  = & \left\langle  \gamma^2_{0,1}
\right\rangle_{\beta} -  \ll  \alpha_{0,1}  \, , \, (-{\mf L}_{GL})^{-1} \alpha_{0,1} \gg_{\beta}.  
  \end{split}
\end{equation}
Here $\langle \cdot \rangle_{\beta}$ and $\ll \cdot \gg_{\beta}$ refer to the scalar products w.r.t. $\rho_{\beta}$. In the special case $W=0$ {\footnote{If $W=0$ the variables $q_x$ have to take values in a compact manifold.}}, we prove rigorously in \cite{BHLLO} that we can make sense for any $\lambda,\ve$ of the term in the righthandside of (\ref{eq:k2345}) and that (\ref{eq:gke}) is valid, supporting the conjecture that (\ref{eq:gke}) is valid in more general situations. Observe that (\ref{eq:gke}) is the Green-Kubo formula for the diffusion coefficient of the Ginzburg-Landau dynamics. 

In specific examples, it is possible to study the behavior of $\kappa_2 (\gamma)$ defined by (\ref{eq:gke}) in the vanishing noise limit $\gamma \to 0$:
\begin{enumerate}
\item Harmonic chain: it is known that the conductivity of the (deterministic) harmonic chain is $\kappa(\ve,0)=\infty$. If $\gamma>0$, $\kappa (\ve,\gamma)= c \gamma^{-1} \ve^{-2}$, $c>0$ a constant, and we get thus that $\lim_{\gamma \to 0} \kappa_2 (\gamma)=\infty$.
\item Disordered pinned harmonic chain: $V$ is quadratic and the one-site potential $W$ is site-dependent given by $W_x (q)=\nu_{x} q^2$ where $\{\nu_x \, ; \, x \in \ZZ\}$ is a sequence of independent identically distributed positive bounded random variables {\footnote{Even if this model does not belong {\textit{stricto sensu}} to the class of models discussed above it is easy to generalize to this case, at least formally, the previous results.}}. It is known (\cite{BH}) that $\kappa (\ve, 0)=0$ so that $\kappa_2 (\ve, 0) =0$. It can be proved that $\kappa_2 (\gamma)$ vanishes as $\gamma$ goes to $0$.
\item Harmonic chain with quartic pinning potential: $V$ is quadratic and $W(q)=q^4$. Then it can be shown that $\limsup_{\gamma \to 0} \kappa_2 (\gamma) < \infty$. This upper bound does not prevent the possibility that $\lim_{\gamma \to 0} \kappa_2 (\gamma)=0$.  
\end{enumerate} 

To prove these results we use the upper bound $\kappa_2 (\gamma)  \le  \left\langle  \gamma^2_{0,1}
\right\rangle_{\beta}$. Recalling (\ref{gamma square}) we see that if we are able to compute ${\mc T}_0 j_{0,1}$ then we can estimate $\left\langle  \gamma^2 ({\mc E}_{0}, {\mc E}_1)
\right\rangle_{\beta}$. It is exactly what is done in \cite{BHLLO} for the specific cases above.

It would be highly interesting to have a rigorous derivation of the formal expansion above. Bypassing this problem, another relevant issue is to decide if genuinely $\lim_{\gamma \to 0} \kappa_2 (\gamma) $ is zero or not. Some authors (see \cite{dRH0} and references therein) conjecture that, in some cases, the conductivity of the deterministic chain $\kappa(\ve,0)$ has a trivial weak coupling expansion ($\kappa(\ve,0) ={\mc O} (\ve^n)$ for any $n \ge 2$). Showing that  $\kappa_2 (\gamma) \to 0$ as $\gamma \to 0$ would support this conjecture.

\subsection{Harmonic chain with velocity-flip noise}

In this section we assume that $V ( r )=r^2/2$.

\subsubsection{Hydrodynamic limits}
\label{subsec:hl-vf}
As explained in the beginning of this chapter an interesting problem consists to derive a diffusion equation for a chain of oscillators perturbed by an energy conserving noise. Consider a one dimensional unpinned chain of $N$ harmonic oscillators with periodic boundary conditions perturbed by the velocity flip noise in the diffusive scale. In other words let $\omega (t) =(\br (t), \bp(t))_{t \ge 0}$ be the process  with generator $N^2 {\mc L}_N=N^2 \left[ {\mc A}_N + \gamma {\mc S}_N \right]$ where ${\mc S}_N$ is given by (\ref{eq:cons-noise-flip}), $\ZZ^d$ being replaced by $\TT_N$, the discrete torus of length $N$, and ${\mc A}_N$ is the Liouville operator of a chain of unpinned harmonic oscillators with periodic boundary conditions. The system conserves two quantities: the total energy $\sum_{x \in \TT_N} {\mc E}_x$, ${\mc E}_x =\frac{p_x^2}{2} +\frac{r_x^2}{2}$, and the total deformation of the lattice $\sum_{x \in \TT_N} r_x$. Consequently, the Gibbs equilibrium measures $\nu_{\beta, \tau}$ are indexed by two parameters $\beta>0$, the inverse temperature, and $\tau \in \RR$, the pressure. They take the form
\begin{equation*}
d\nu_{\beta, \tau} (d\br, d \bp) = \prod_{x \in \TT_N}{\mc Z}^{-1} (\beta,\tau) \,  \exp \{ -\beta ({\mc E}_x -\tau r_x)\} dr_x dp_x
\end{equation*}
where
$${\mc Z} (\beta, \tau) =  \cfrac{2\pi}{\beta} \exp(\beta \tau^2 /2).$$
Observe the following thermodynamic relations 
$$\int {\mc E}_x \, d\nu_{\beta, \tau} =\beta^{-1} + \tfrac{\tau^2}{2}, \quad \int {r}_x \, d\nu_{\beta, \tau} = \tau$$
or equivalently
\begin{equation*}
\tau =\int {r}_x \, d\nu_{\beta, \tau}, \quad \beta = \left\{ \int {\mc E}_x \, d\nu_{\beta, \tau}  - \tfrac{ \Big( \int {r}_x \, d\nu_{\beta, \tau} \Big)^2}{2}\right\}^{-1}.
\end{equation*}

\begin{definition}
Let $\TT=[0,1)$ be the continuous torus. Let ${\mf e}_0: \TT \to \RR$ and ${\mf r}_0:\TT \to \RR$ be two continuous macroscopic profiles such that ${\mf e}_0 > \tfrac{{\mf r}_0^2}{2}$. A sequence of probability measures $(\mu^N)_{N \ge 1}$ on $(\RR \times \RR)^{\TT_N}$ is said to be a sequence of Gibbs local equilibrium states associated to the energy profile ${\mf e}_0$ and the deformation profile ${\mf r}_0$ if 
\begin{equation*}
d\mu^N (d\br, d\bp) = \prod_{x \in \TT_N} {\mc Z}^{-1} (\beta_0 (\tfrac{x}{N}) ,\tau_0 (\tfrac{x}{N}))\, \exp \{ -\beta_0 (x/N)  ({\mc E}_x -\tau_0 (x/N) r_x)\} dr_x dp_x
\end{equation*}
where the functions $\beta_0$ and $\tau_0$ are defined by
\begin{equation*}
\tau_0 = {\mf r}_0, \quad \beta_0 = \{ {\mf e}_0 - \tfrac{{\mf r}_0^2}{2}\}^{-1}.
\end{equation*}
\end{definition}

Once we have the microscopic fluctuation-dissipation equation (see (\ref{eq:fluct-diss-eq-harm})) and assuming the propagation of local equilibrium in the diffusive time scale it is easy to guess the hydrodynamic equations followed by the system. In \cite{Sim} the following theorem is proved.

\begin{theorem}[\cite{Sim}]
\label{thm:Sim1}
Consider the unpinned velocity-flip model with periodic boundary conditions. Let $(\mu^N)_N$ be a sequence of Gibbs local equilibrium states {\footnote{One can consider more general initial states, see \cite{Sim}.}} associated to a bounded energy profile ${\mf e}_0$ and a deformation profile ${\mf r}_0$. For every $t \ge 0$, and any test continuous functions $G,H: \TT \to \RR$, the random variables 
\begin{equation}
\label{eq:vf-empiricaldensities}
\Big(\frac1N\sum_{x \in\TT_N} {\displaystyle{ G(\tfrac{x}{N}) r_x (tN^2), \frac1N \sum_{x \in\TT_N} H(\tfrac{x}{N}) {\mc E}_x (tN^2)}} \Big)
\end{equation}
converge in probability as $N \to \infty$ to $$\Big(\int_{\TT} G(y) {\mf r} (t,y) dy, \int_{\TT} H(y) {\mf e} (t,y)dy \Big)$$ where $\mf r$ and $\mf e$ are the (smooth) solutions to the hydrodynamical equations
\begin{equation}
\label{eq:2:hl-re1}
\begin{cases}
\partial_t {\mf r} =\frac 1{\gamma}\,  \partial_y^2 \, {\mf r},\\
\partial_t {\mf e}= \frac1{2 \gamma} \,  \partial_y^2 \, \left[ {\mf e} \, + \, \frac{ {\mf r}^2}{2} \right]
\end{cases}
,\quad y \in \TT,
\end{equation} 
with initial conditions ${\mf r} (0,y) ={\mf r}_0 (y)$, ${\mf e}(0,y)={\mf e}_0 (y)$.
\end{theorem}

The proof of this theorem is based on Yau's relative entropy method (\cite{yau1}, \cite{OVY}). The general strategy is simple. Let $\mu_t^N$ be the law of the process at time $tN^2$ starting from $\mu^N$ and let ${\tilde \mu}^N_t$ be a sequence of Gibbs local equilibrium state corresponding to the deformation profile ${\mf r}_t (\cdot):={\mf r} (t,\cdot)$ and energy profile ${\mf e}_t (\cdot):={\mf e}(t,\cdot)$ solution of (\ref{eq:2:hl-re1}). We expect that since $\mf e$ and ${\mf r}$ are the hydrodynamic profiles, the probability measure of the process $\mu_t^N$ is close, in some sense, to the local Gibss state ${\tilde \mu}_t^N$. Yau's relative entropy method consists to show that the entropic distance {\footnote{There is some abuse of language here since the relative entropy is not a distance between probability measures.}}  
\begin{equation}
\label{eq:Hyauest}
H_N (t):=H(\mu_t^N | {\tilde \mu}_t^N) =o(N)
\end{equation}
between the two states is relatively small. Assuming (\ref{eq:Hyauest}), in order to prove for example the convergence of the empirical energy,  we use the entropy inequality {\footnote{It is a trivial consequence of the definition (\ref{eq:ent009}).}} which states that for any $\alpha>0$ and test function $\phi$
\begin{equation}
\label{eq:ent-inequ}
\int \phi d\mu_t^N \le \tfrac{H(\mu_t^N | {\tilde \mu}_t^N) }{\alpha} +\cfrac{1}{\alpha} \log \left( \int e^{\alpha \phi} d{\tilde \mu}_t^N \right).
\end{equation}
We take then $\alpha=\delta N$, $\delta>0$, and 
$$\phi=\left| \frac1N \sum_{x \in\TT_N} H({\tfrac{x}{N}}) {\mc E}_x  - \int_{\TT} H(y) {\mf e} (t,y)dy \right|.$$
Since ${\tilde \mu}_t^N$ is fully explicit and even product, by using large deviations estimates, it is possible to show that
\begin{equation}
\limsup_{N \to \infty} \cfrac{1}{\delta N} \log \left( \int e^{\delta N \phi} d{\tilde \mu}_t^N \right) =I(\delta)
\end{equation}
where $I(\delta ) \to 0$ as $\delta \to 0$. By using (\ref{eq:Hyauest}), we are done. It remains then to prove (\ref{eq:Hyauest}) and for this we rely on a Gronwall inequality for the entropy production ($C>0$ is a constant)
\begin{equation}
\label{eq:entprodbound}
\partial_t H_N \le C H_N (t) + o(N).
\end{equation}
The proof of (\ref{eq:entprodbound}) is quite evolved and we refer the interested reader to \cite{Sim}, \cite{BO-livre1} (see also \cite{KL} for some overview on the subject). It is in this step that the macro-ergodicity of the dynamics is used in order to derive the so-called one-block estimate.  

For non-gradient systems, i.e. systems such that the microscopic currents of the conserved quantities are not given by discrete gradients {\footnote{Observe that if a system is gradient then a microscopic fluctuation-dissipation equation (\ref{eq:2:fluc-diss}) holds with a zero fluctuating term.}}, the previous strategy has to be modified. Indeed, in order to have (\ref{eq:Hyauest}) it is necessary to replace the local equilibrium Gibbs state ${\tilde \mu}_t^N$ by a local equilibrium state with a first order correction term of the form
\begin{equation}
\label{eq:firstoderGibbs}
\begin{split}
&d{\hat \mu}_t^{N} (d\br, d\bp) \\
&= Z_{t,N}^{-1} \, \prod_{x \in \TT_N} \exp \left\{ -\beta_t (x/N)  ({\mc E}_x -\tau_t (x/N) r_x) +\tfrac{1}{N} F(t, x/N) (\theta_x g) (\br, \bp) \right \} dr_x dp_x
\end{split}
\end{equation}
where $Z_{t,N}$ is a normalization constant, 
\begin{equation*}
\tau_t = {\mf r}_t, \quad \beta_t = \{ {\mf e}_t - \tfrac{{\mf r}_t^2}{2}\}^{-1}
\end{equation*}
and the functions $F$ and $g$ are judiciously chosen. The choice is guided by the fluctuation-dissipation relation (\ref{eq:fluct-diss-eq-harm}) and done in order to obtain the first order "Taylor expansion" (\ref{tay}) below. 

Let $\Omega^N =(\RR \times \RR)^{\TT_N}$ be the configurations space and denote
 \begin{equation} 
 {\hat H}_N(t):=H\left(\mu_t^N \vert {\hat \mu}_t^N \right)=\int_{\Omega^N} f_t^N(\omega) \log \frac{f_t^N(\omega)}{\phi_t^N(\omega)}  d\nu_* (\omega)\ , 
 \label{eq:relentr2} 
 \end{equation}
 where  $f_t^N$ is the density of $\mu_t^N$ with respect to the Gibbs reference measure $\nu_*:=\nu_{1,0}$.  In the same way, $\phi_t^N$ is the density of ${\hat \mu}_t^N$ with respect to $\nu_*$ (which is fully explicit). The goal is to get (\ref{eq:entprodbound}) with $H_N$ replaced by ${\hat H}_N$ .
 
 We begin with the following entropy production bound. Let us denote by ${\mc L}_N^*=-\mc A_N+\gamma {\mc S}_N$ the adjoint of $\mc L_N$ in $\mathbb{L}^2(\nu_*)$. 
 
\begin{lemma}
\label{entropy}
\begin{equation*}
\partial_t {\hat H} _N(t) \leq  \int \frac{1}{\phi_t^N}\left(N^2\mc L_N^*\phi_t^N-\partial_t\phi_t^N\right) f_t^N d\nu_{*}= \int \left[  \frac{1}{\phi_t^N}\left(N^2\mc L_N^*\phi_t^N-\partial_t\phi_t^N\right) \right] d\mu_t^N \ .
\end{equation*}
\end{lemma}

\begin{proof}
We have that $f_t^N$ solves the Fokker-Plack equation $\partial_t f_t^N = N^2 {\mc L}_N^* f_t^N$. Assuming it is smooth to simplify, we have
\begin{equation*}
\begin{split}
\partial_t {\hat H} _N(t)& = \int \partial_t f^N_t [ 1+ \log f_t^N] d\nu_* - \int \partial_t f_t^N \log \phi_t^N d\nu_* - \int \partial_t \phi_t^N  \frac{f_t^N}{\phi_t^N} d\nu_*\\
&=N^2  \int {\mc L}_N^* f_t^N  [\log f_t^N -\log \phi_t^N] d\nu_* - \int \partial_t \phi_t^N  \frac{f_t^N}{\phi_t^N} d\nu_*\\
&= N^2 \int f_t^N {\mc L}_N [ \log \tfrac{f_t^N}{\phi_t^N}] d\nu_* - \int \partial_t \phi_t^N  \frac{f_t^N}{\phi_t^N} d\nu_*\\
&= N^2 \int \tfrac{f_t^N}{\phi_t^N} {\mc L}_N [ \log \tfrac{f_t^N}{\phi_t^N}] \phi_t^N d\nu_* - \int \partial_t \phi_t^N  \frac{f_t^N}{\phi_t^N} d\nu_*\\
& \le N^2 \int {\mc L}_N [ \tfrac{f_t^N}{\phi_t^N}] \phi_t^N d\nu_* - \int \partial_t \phi_t^N  \frac{f_t^N}{\phi_t^N} d\nu_*\\
&=N^2 \int  \cfrac{f_t^N}{\phi_t^N} \, {\mc L}_N^*\phi_t^N d\nu_* - \int \partial_t \phi_t^N  \frac{f_t^N}{\phi_t^N} d\nu_*
\end{split}
\end{equation*}
where we used that for any positive function $h$, $h {\mc L}_N (\log h) \le {\mc L}_N h$ (this is a consequence of Jensen's inequality). 
\cqfd\end{proof}

We define $\xi_x:=({\mc E}_x,r_x)$ and $\pi(t,q):=({\mf e}(t,q),{\mf r}(t,q)).$ If $f$ is a vectorial function, we denote its differential by $Df$. 

\begin{proposition}[\cite{Sim}]
\label{prop:M1} 
Let $(\lambda, \beta)$ be defined by $\beta= ({\mf e}- \tfrac{{\mf r}^2}{2})^{-1}$ and $\lambda= -\beta {\mf r}$. The term $ (\phi_t^N)^{-1}\left(N^2\mc L_N^*\phi_t^N-\partial_t\phi_t^N\right)$ can be expanded as
\begin{multline} 
(\phi_t^N)^{-1}\left(N^2\mc L_N^*\phi_t^N-\partial_t\phi_t^N\right)  \\
=\sum_{k=1}^5 \sum_{x \in \T_N} v_k\left(t,\frac{x}{N}\right) \left[J_x^k-H_k\left({\pi}\left(t,\frac{x}{N}\right)\right)-(DH_k)\left({\pi}\left(t,\frac{x}{N}\right)\right)\cdot \left({\xi}_x-{\pi}\left(t,\frac{x}{N}\right)\right)\right] \\ +o(N)
\label{tay} 
\end{multline} 
where 
\begin{equation*} 
\begin{array}{| c | c | c | c |}
\hline k & J_x^k & H_k(\mf{e},\mf{r}) & v_k(t,q) \\
\hline \hline  1 &  p_x^2+r_xr_{x-1}+2\gamma p_x r_{x-1} &  \mf{e}+{\mf r}^2/2 &  -(2\gamma)^{-1} \partial^2_q \beta(t,q) \\ 
2 & r_x+\gamma p_x & {\mf r} &  -{\gamma}^{-1} \partial^2_q\lambda(t,q) \\
3 & p_x^2\ (r_x+r_{x-1})^2 &  (2{\mf e}-{\mf r}^2)  \left({\mf e}+3{\mf r}^2/2\right) & (4\gamma)^{-1} [\partial_q \beta(t,q)]^2 \\
4 & p_x^2 \ (r_x+r_{x-1}) & {\mf r} \ (2{\mf e}-{\mf r}^2)&  {\gamma}^{-1} \partial_q\beta(t,q) \ \partial_q\lambda(t,q) \\
5 & p_x^2 &  {\mf e}-{\mf r}^2/2 &  {\gamma}^{-1} [\partial_q \lambda(t,q)]^2 \\ 
\hline
\end{array} 
\end{equation*}
\end{proposition}

Observe that $H_k (e,r)$ is equal to $\int J_x^k d\nu_{\beta,\tau}$ where $\beta,\tau$ are related to $e,r$ by the thermodynamic relations. Thus, the terms appearing in the righthand side of (\ref{tay}) can be seen as  first order ``Taylor expansion". The form of the first order correction in (\ref{eq:firstoderGibbs}) plays a crucial role in order to get such expansions. 

A priori  the first term on the right-hand side of (\ref{tay}) is of order $N$, but we want to take advantage of these microscopic Taylor expansions to show it is in fact of order $o(N)$. 

First, we need to cut-off large energies in order to work  with bounded variables only. To simplify, we assume they are bounded ab initio. 

Let $\ell$ be some integer (dividing $N$). We introduce some averaging over microscopic blocks of size $\ell$ and we will let $\ell \to \infty$ after $N\to \infty$. We decompose $\TT_N$ in a disjoint union of $p=N/\ell$ boxes $\Lambda_\ell (x_j)$ of length $\ell$ centered at $x_j$, $j \in \{ 1, \ldots,p\}$. The microscopic averaged profiles in a box of size $\ell$ around $y \in \TT_N$ are defined by 
\begin{equation*}
{\tilde \xi}_{\ell} (y) = \cfrac{1}{\ell} \sum_{x \in \Lambda_\ell (y)} \xi_x.
\end{equation*}
Similarly we define
$${\tilde J}^k_\ell (y) =\cfrac{1}{\ell} \sum_{x \in \Lambda_\ell (y)} J_x^k.$$
In (\ref{tay}) we rewrite the sum $\sum_{x \in {\TT}_N}$ as $\sum_{j=1}^p \sum_{x \in \Lambda_{\ell} (x_j)}$ and, by using the smoothness of the function $v_k$, $H_k$, it is easy to replace the term 
$$\cfrac{1}{N}\sum_{x \in \T_N} v_k\left(t,\frac{x}{N}\right) \left[J_x^k-H_k\left({\pi}\left(t,\frac{x}{N}\right)\right)-(DH_k)\left({\pi}\left(t,\frac{x}{N}\right)\right)\cdot \left({\xi}_x-{\pi}\left(t,\frac{x}{N}\right)\right)\right] $$
by 
$$\cfrac{1}{p} \sum_{j=1}^p v_k\left(t,\frac{x_j}{N}\right) \left[{\tilde J}^k_\ell (x_j)  \, - \, H_k\left({\pi}\left(t,\frac{x_j }{N}\right)\right)-(DH_k)\left({\pi}\left(t,\frac{x_j}{N}\right)\right)\cdot \left({\tilde \xi}_\ell (x_j) -{\pi}\left(t,\frac{x_j}{N}\right)\right)\right]$$ 
in the limit $N, \ell \to \infty$ with some error term of order $o(1)$.

Then, the strategy consists in proving the following crucial estimate, often called the {\textit{one-block estimate}}: we replace the empirical average current ${\tilde J}^k_\ell (x_j)$ which is averaged over a box centered at $x_j$ by its mean with respect to a Gibbs measure with the parameters corresponding to the microscopic averaged profiles ${\tilde \xi}_\ell (x_j)$, i.e. $H_k ({\tilde \xi}_\ell (x_j) )$. This non-trivial step is achieved thanks to some compactness argument and the macro-ergodicity of the dynamics.  

Consequently we have to deal with terms in the form
\begin{equation}
\label{eq:phild}
\begin{split}
\cfrac{1}{p} \sum_{j=1}^p v_k\left(t,\frac{x_j}{N}\right) & \left[ H_k \left( {\tilde \xi}_\ell (x_j) \right)  \, - \, H_k\left({\pi}\left(t,\frac{x_j }{N}\right)\right) \right.\\
& \left. -(DH_k)\left({\pi}\left(t,\frac{x_j}{N}\right)\right)\cdot \left({\tilde \xi}_\ell (x_j) -{\pi}\left(t,\frac{x_j}{N}\right)\right)\right]
\end{split}
\end{equation}

The final step consists then in applying the entropy inequality (\ref{eq:ent-inequ}) with respect to ${\hat \mu}_t^N$ with $\phi:=\phi_{\ell, N}$ given by (\ref{eq:phild}) and $\alpha=\delta N$, $\delta>0$ fixed but small. This will produce some term of order ${\hat H}_N (t) /N$ plus the term
\begin{equation*}
\limsup_{\ell \to \infty} \limsup_{N \to \infty}  \cfrac{1}{\delta N} \log \left( \int e^{\delta N \phi} d{\hat \mu}_t^N \right) =I(\delta).
\end{equation*} 
By using some large deviations estimates (observe that ${\hat \mu}_t^N$ is explicit and product at first order in $N$) one can show that $I(\delta)$ is nonpositive for $\delta$ sufficiently small. Thus we get the desired Gronwall inequality. 

There is some additional difficulty that we hid under the carpet in the sketch of the proof. Since the state space is non compact, a control of high energies is required for the initial cut-off. This is a highly non trivial problem {\footnote{A similar problem appears in \cite{OVY} where the authors derived Euler equations for a gas perturbed by some ergodic noise. There, to overcome this difficulty, the authors replace ab initio the kinetic energy by the relativistic kinetic energy.}}. In the harmonic case considered here this control is obtained thanks to the following remark: the set of mixtures of Gaussian probability measures {\footnote{A Gibbs local equilibrium state is a Gaussian state in the harmonic case.}} is preserved by the (harmonic) velocity-flip model. Since for Gaussian measures all the moments are expressed in terms of the covariance matrix, required bounds can be obtained by a suitable control of the covariance matrices appearing in the mixture. 

The extension of this result in the anharmonic case is a challenging open problem (see however \cite{Olla-Sasada} where equilibrium fluctuations are considered for an anharmonic chain perturbed by a conservative noise acting on the momenta and positions).

\subsubsection{Fourier's law}
Since in the harmonic case an exact fluctuation-dissipation equation is available Fourier's law can be obtained without too much work {\footnote{The {\textit{a posteriori}} simple but fundamental remark that an exact fluctuation-dissipation equation exists for the harmonic model (see (\ref{eq:fluct-diss-eq-harm})) is the real contribution of \cite{BO1}. }}. 

 \begin{theorem}[\cite{BO1,BO2}]
 \label{thm:BO1}
Consider the one-dimensional harmonic chain in contact with two heat baths and with forced boundary conditions as in Section \ref{subsec:lrGK}. Then Fourier's law holds:
\begin{equation}
\label{eq:2:flwertyui}
{\tilde J}_s:=\lim_{N \to \infty} N \langle j_{0,1}^e \rangle_{ss} = \cfrac{1}{2\gamma} \left\{ (T_\ell -T_r) +(\tau_\ell^2 -\tau_r^2)\right\}
\end{equation}
and we have
\begin{equation}
\label{eq:jthermh}
\begin{split}
{\hat J}_\ell=\lim_{N \to \infty} N(\langle p_1^2 \rangle_{ss} -T_\ell) =\cfrac{1}{2\gamma \gamma_{\ell}} \left[ (T_r -T_\ell) +(\tau_{\ell} -\tau_r)^2 \right],\\
{\hat J}_r=\lim_{N \to \infty}  N(T_r -\langle p_N^2 \rangle_{ss} )= \cfrac{1}{2\gamma \gamma_r} \left[ (T_r -T_\ell) - (\tau_{\ell} -\tau_{r})^2 \right].
\end{split}
\end{equation}
\end{theorem}

\begin{proof}
We divide the proof in two steps:
\begin{itemize}
\item We first prove that there exists a constant $C$ independent of $N$ such that $|\langle j_{0,1}^e \rangle_{ss} | \le C/N$. This is obtained by using the fluctuation-dissipation equation and the fact that $\langle j^e_{x,x+1}\rangle_{ss}$ is independent of $x$:
\begin{eqnarray}
\langle j^e_{0,1} \rangle_{ss} &=& \cfrac{1}{N-3} \sum_{x=2}^{N-2} \langle j^e_{x,x+1} \rangle_{ss}\\ \nonumber
&=& -\cfrac{1}{2\gamma} \cfrac{1}{N-3} \sum_{x=2}^{N-2} \left\langle\nabla \left[ p_{x}^2 + r_{x} r_{x+1}\right]\right\rangle_{ss}\\ \nonumber
&=& \cfrac{1}{2\gamma}{\frac{1}{N-3}} \left\{ (\langle p_2^2 \rangle_{ss} +\langle r_2 r_3\rangle_{ss} )-(\langle p_{N-1}^2 \rangle_{ss} +\langle r_{N-1} r_{N}\rangle_{ss} )\right\}.
\end{eqnarray}
By using simple computations, one can show that $(\langle p_2^2 \rangle_{ss} +\langle r_2 r_3\rangle_{ss} )-(\langle p_{N-1}^2 \rangle_{ss} +\langle r_{N-1} r_{N}\rangle_{ss} )$ is uniformly bounded in $N$ by a positive constant.
\item Now we have only to evaluate the limit of each term appearing in $(\langle p_2^2 \rangle_{ss} +\langle r_2 r_3\rangle_{ss} )-(\langle p_{N-1}^2 \rangle_{ss} +\langle r_{N-1} r_{N}\rangle_{ss} )$. Notice that assuming local equilibrium we easily get  the result. The first step implies that $\langle j_{0,1}^e \rangle_{ss}$ and $ \langle j_{N,N+1}^{e} \rangle_{ss}$ vanish as $N \to +\infty$. Since $V_s$ goes to $0$ by Lemma \ref{lem:velo}, one has that $\langle p_1^2 \rangle_{ss}$ and $\langle p_{N}^2 \rangle_{ss}$ converge respectively to $T_\ell$ and $T_r$. By using some ``entropy production bound'' one can propagate this local equilibrium information to the particles close to the boundaries and show (\ref{eq:2:flwertyui}).
\end{itemize}
\cqfd\end{proof}

It follows from this Theorem that the system can be used as a heater but not as a refrigerator. Assume for example that $T_r > T_{\ell}$.  The term ${\hat J}_\ell$ (resp. ${\hat J}_r$) is the macroscopic heat current from the left reservoir to the system (resp. from the system to the right reservoir). Whatever the values of $\tau_\ell, \tau_r$ are, ${\hat J}_{\ell} >0$ and we can not realize a refrigerator.  But if $(T_r -T_{\ell}) < (\tau_r -\tau_{\ell})^2$ then ${\hat J}_r <0$ and we realized a heater.\\

The proof of the validity of Fourier's law for anharmonic chains perturbed by an energy conserving noise is still open.

\subsubsection[Macroscopic Fluctuation Theory] {Macroscopic Fluctuation Theory for the energy conserving harmonic chain}

The macroscopic fluctuation theory (\cite{BdSGJLL1}) is a general approach developed by Bertini, De Sole, Gabrielli, Jona-Lasinio and Landim to calculate the large deviation functional of the empirical profiles of the conserved quantities of Markov processes in a NESS.  Its main interest is that it can be applied to a large class of boundary driven diffusive systems and does not require the explicit form of the NESS but only the knowledge of two thermodynamic macroscopic parameters of the system, the diffusion coefficient $D(\rho)$ and the mobility $\chi(\rho)$. This theory can be seen as an infinite dimensional generalization of the Freidlin-Wentzel theory \cite{Freidlin-Wentzel} and is based on the large deviation principle for the hydrodynamics of the system.

In  order to explain (roughly) the theory we consider for simplicity a Markovian system $\{ \eta (t) :=\{ \eta_x (t) \in \RR \, ; \, x \in \{1, \ldots N\} \}_{t \ge 0}$ with only one conserved quantity, say the density $\rho$, in contact with two reservoirs at each extremity. Here $N$ is the size of the system which will be sent to infinity. We denote by $\mu^N_{ss}$ the nonequilibrium stationary state of $\{\eta(t)\}_{t \ge 0}$. For any microscopic configuration $\eta:=\{\eta_x \, ;\, x \in \{1, \ldots,N\} \}$ let
\begin{equation*}
\pi^N (\eta, \cdot)=\sum_{x=1}^{N-1} \eta_x {\bf 1}_{\big[\tfrac{x}{N}, \tfrac{x+1}{N}\big)} (\cdot)
\end{equation*}
be the empirical density profile. In the diffusive time scale, we assume that $\pi^N (\eta (tN^2), \cdot)$ converges as $N$ goes to infinity to $\rho_t (\cdot):=\rho(t,\cdot)$ solution of
\begin{equation*}
\begin{cases}
\partial_t \rho = \partial_y (D(\rho) \partial_y \rho), \quad y \in [0,1], \quad t\ge 0,\\
\rho(t,0) =\rho_\ell, \quad \rho(t,1)=\rho_r,\quad  t \ge 0, \\
\rho(0, \cdot) =\rho_0 (\cdot) 
\end{cases}
\end{equation*}
where $\rho_0 (\cdot)$ is the initial density profile, $D(\rho)>0$ is the diffusion coefficient and $\rho_\ell$, $\rho_r$ the densities fixed by the reservoirs. As $t \to \infty$ the solution $\rho_t$ of the hydrodynamic equation converges to a stationary profile $\bar \rho :[0,1] \to \RR$ solution of $D(\bar \rho) \partial_y {\bar \rho} =J=const.$ with ${\bar \rho}(0)=\rho_\ell$, ${\bar \rho} (1)=\rho_r$. We assume that under $\mu_{ss}^N$, the empirical density profile $\pi^N (\eta, \cdot)$ converges to ${\bar \rho}$. This assumption is nothing but a law of large numbers for the random variables $\pi^N$. 

We are here interested in the corresponding large deviation principle. Thus, we want to estimate the probability that in the NESS $\mu_{ss}^N$ the empirical density profile $\pi^N$ is close to an atypical macroscopic profile $\rho(\cdot) \neq {\bar \rho}$. This probability typically is of order $e^{-N {\bb V} (\rho)}$ where $\bb V$ is the rate function:
\begin{equation*}
\mu_{ss}^N ( \pi^N (\eta, \cdot) \approx \rho(\cdot) ) \approx e^{-N {\bb V} (\rho)}.
\end{equation*}
The goal of the macroscopic fluctuation theory is to obtain information about this functional.

The condition to be fulfilled by the system to apply the theory of Bertini et al. is that it satisfies a {\textit{dynamical large deviation principle}} with a rate function which takes a quadratic form {\footnote{Such property has been proved to be valid for a large class of stochastic dynamics (\cite{KOV89}, \cite{KL}).}} like (\ref{eq:formI}). 

Let us first explain what we mean by dynamical large deviation principle. Imagine we start the system from a Gibbs local equilibrium state corresponding to the macroscopic profile $\rho_0$. We want to estimate the probability that the empirical density  $\pi^N (\eta (tN^2), \cdot)$ is close during the macroscopic time interval $[0,T]$, $T$ fixed, to a smooth macroscopic profile $\gamma (t, y)$ supposed to satisfy \footnote{This assumption avoids taking into account the cost to produce the initial profile, cost which is irrelevant for us.} $\gamma(0,\cdot)=\rho_0$. This probability is exponentially small in $N$ with a rate $I_{[0,T]} ( \gamma\, | \, {\rho_0})$
\begin{equation}
\label{eq:IIII}
{\PP} \left[ \pi^N ( \eta (tN^2), y) \approx \gamma (t,y), \; (t,y) \in [0,T] \times [0,1] \right] \sim e^{- N I_{[0,T]} (\gamma | \rho_0)}. 
\end{equation}
The rate function is assumed to be of the form 
\begin{equation}
\label{eq:formI}
\begin{split}
I_{[0,T]} (\gamma \, | \, \rho_0) & = \frac{1}{2} \int_0^T dt \int_0^1 dy \;  \chi (\rho (t,y))  \left[ (\partial_y H) (t,y) \right]^2 
\end{split}
\end{equation}
where $ \partial_y H$ is the extra gradient external field needed to produce the fluctuation $\gamma$, namely such that
\begin{equation}
\label{eq:Hgamma}
\partial_t \gamma = \partial_y \left [ D(\gamma) \partial_y \gamma \, - \, \chi(\gamma) \partial_y H \right].
\end{equation} 
Thus, $I_{[0,T]} (\gamma | \rho_0)$ is the work done by the external field $\partial_y H$ to produce the fluctuation $\gamma$ in the time interval $[0,T]$. The function $\chi$ appearing in (\ref{eq:Hgamma}) is the second thermodynamic parameter (with the diffusion coefficient $D$) mentioned in the beginning of this section. The two parameters $D$ and $\chi$ are in fact related together by the Einstein relation so that knowing one of them and the Gibbs states of the microscopic model is sufficient to obtain the second. 

To show this result the strategy is the following. We perturb the Markov process $\{\eta (t)\}_{t \ge 0}$ thanks to the function $H:=H(\gamma)$, which is solution of the Poisson equation (\ref{eq:Hgamma}), by adding locally a small space inhomogeneous drift provided by $\partial_y H$. In doing so we obtain a new Markov process $\{ \eta^H (t)\}_{t \ge 0}$ such that in the diffusive time scale $\pi^N (\eta^H (tN^2), \cdot)$ converges to $\gamma (\cdot)$. Let ${\bb P}^H$ (resp. $\PP^0$) be the probability measure on the empirical density paths space induced by $\{\eta^H (tN^2)\}_{t \in[0,T]}$ (resp. $\{\eta(tN^2)\}_{t \in [0,T]}$) . Then, by using hydrodynamic limits techniques similar to the ones explained in Section \ref{subsec:hl-vf} we show that in the large $N$ limit, under ${{\mathbb P}^H}$, the Radon-Nikodym derivative is well approximated by {\footnote{We use Girsanov transform to express the Radon-Nikodym derivative. A priori it is not a functional of the empirical density and we need to establish some {\textit{replacement lemma}} (see \cite{KL}). }}
\begin{equation*}
\cfrac{{\rm d}\PP^0}{{\rm d}{{\PP^H}}} (\pi ) \approx \exp\left\{ -N I_{[0,T]}( \pi | \rho_0)\right\}. 
\end{equation*} 
Here $\pi:=\{\pi (t,y)\, ; \, t \in[0,T], y \in [0,1]\}$ is any space-time density profile. Thus, since
\begin{equation*}
\PP^0 \left[\pi^N (\eta (tN^2), \cdot) \sim \gamma(t, \cdot),\; t\in [0,T]\right] 
={\EE^H} \left[ \cfrac{{\rm d}\PP^0}{{\rm d}{{\PP}^H}} (\pi) \,  {\bf 1}_{\left\{ \pi (t, \cdot) \sim \gamma (t, \cdot), t \in [0,T] \right\} }\right]
\end{equation*}
we obtain (\ref{eq:IIII}).

The macroscopic fluctuation theory claims that the large deviations functional ${\bb V} (\rho)$ of the empirical density in the NESS coincides with the quasi-potential ${\bb W} (\rho)$ defined by
\begin{equation*}
{\bb W} (\rho) = \inf_{\substack{\gamma: \gamma (-\infty) = {\bar \rho}\\ \gamma (0) =\rho}} I_{[-\infty, 0]} (\gamma | {\bar \rho}).
\end{equation*} 
Here $I_{[-\infty, 0]}$ is obtained from $I_{[0,T]}$ by a shift in time by $-T$, $T$ being sent to $+\infty$ afterwards. In words, the quasi potential determines the cost to produce a fluctuation equal to $\gamma$ at $t=0$ when the system is macroscopically in the stationary profile ${\bar \rho}$ at $t=-\infty$. 

Thus, the problem is reduced to computing $W$. It can be shown that $W$ solves (at least formally) the infinite-dimensional Hamilton-Jacobi equation
\begin{equation}
\frac{1}{2} \left \langle \partial_y \left[  \frac{\delta {\bb W}}{ \delta \rho} \right] \, , \chi( \rho) \,  \partial_y \left[  \frac{\delta {\bb W}}{ \delta \rho} \right] \right\rangle + \left\langle \frac{\delta {\bb W}}{ \delta \rho}\, , \, \partial_y \left[  D(\rho) \partial_y \rho \right] \right\rangle =0
\end{equation}
where $\langle \cdot, \cdot \rangle$ denotes the usual scalar product in ${\mathbb L}^2 ([0,1])$. Note that there is no uniqueness of solutions (${\bb W}=0$ is a solution) and up to now a general theory of infinite dimensional Hamilton-Jacobi equations is still missing. This implies that we have in fact to solve by hand the variational problem and the solution is only known for few systems. This is an important limitation of the macroscopic fluctuation theory. Even getting interesting qualitative properties on ${\bb W}$ is difficult. 

The rigorous implementation of this long program has only been carried for the boundary driven Symmetric Simple Exclusion Process and extended with less rigor to a few other systems (see \cite{BDSGLL-SSEP}, \cite{Bod-Giac}, \cite{Farfan} for rigorous results).   

Let us now try to apply this theory for the harmonic chain with velocity-flip noise. Since we have a fully explicit microscopic fluctuation-dissipation equation (even when some harmonic pinning is added) we can easily guess what is the form of the hydrodynamic equations under various boundary conditions by assuming that the propagation of local equilibrium in the diffusive time scale holds. Nevertheless, let us observe that a rigorous derivation is missing, the obstacle being a sufficiently good control of the high energies {\footnote{This control is only available in the case of periodic boundary conditions (\cite{Sim}).}}. The boundary conditions we impose to the system are the following. At the left (resp. right) end we put the chain in contact with a Langevin bath at temperature $T_\ell$ (resp. $T_r$) and consider the system with fixed boundary conditions or with forced boundary conditions with the same force $\tau$ at the two boundaries.   
Then, for the unpinned chain, the equations (\ref{eq:2:hl-re1}) are still valid but they are supplemented with the boundary conditions (\cite{BKLL1})
\begin{equation}
\left[{\mf e} -\cfrac{{\mf r}^2}{2} \right] (t,0)= T_\ell, \quad \left[{\mf e} -\cfrac{{\mf r}^2}{2} \right] (t,1)= T_r, 
\end{equation}
since the Langevin baths fix the temperatures at the boundaries and
\begin{equation}
\partial_y {\mf r} (t, 0) = \partial_y {\mf r} (t, 1)=0 
\end{equation} 
for fixed boundary conditions (the total length of the chain is constant {\footnote{Indeed, by (\ref{eq:2:hl-re1}), we have $\partial_t (\int_{0}^1 {\mf r} (t,y) dy) = \gamma^{-1} \int_{0}^1 {\partial}_y^2 {\mf r} (t,y) dy= \gamma^{-1} [  \partial_y {\mf r} (t, 1) -\partial_y {\mf r} (t,0) ] =0$.}}) and 
\begin{equation}
{\mf r} (t, 0) = {\mf r} (t, 1)= \tau
\end{equation} 
for forced boundary conditions. 

If the chain is pinned by the harmonic potential $W(q)=\nu q^{2} /2$ then only the energy is conserved and the macroscopic diffusion equation takes the form
\begin{equation}
\begin{cases}
\partial_t {\mf e} = \partial_{y} (\kappa \partial_y {\mf e}),\\
{\mf e}(0,y) = {\mf e}_0 (y),\\
{\mf e} (t,0)= T_\ell, \; {\mf e} (t,1)=T_r,
\end{cases}
 \quad y \in (0,1)
\end{equation}
where the conductivity $\kappa$ is equal to  (\cite{BKLL1})
\begin{equation}
\label{eq:conductivityuiop}
\kappa= \cfrac{1/\gamma}{2 + \nu^2 + \sqrt{\nu (\nu +4)}}.
\end{equation}

Assuming a good control of high energies, it is possible to derive the dynamical large deviations function of the empirical conserved quantities. The goal would be to compute the large deviation functional of the NESS which according to the macroscopic fluctuation theory coincides with the quasi potential. We recall that the quasi potential is defined by a variational problem and that it depends only on two thermodynamic quantities, the diffusion coefficient and the mobility (the latter are matrices if several conserved quantities are involved). 

Let us first consider the pinned velocity flip model where the energy is the only conserved quantity. It turns out that the mobility is a quadratic function. Consequently, the methods exposed in Theorem 6.5 of \cite{BdSGJLL2} apply and the variational formula can be computed. The quasi potential ${\bb V}(\cdot)$ is given by (\cite{BKLL1})
 \begin{equation}
 \label{eq:2:VKMP}
{\bb V} (e)=\int_0^1 dq \left[ \cfrac{e(q)}{F(q)} -1 -\log\left( \cfrac{e(q)}{F(q)}\right) - \log \left( \cfrac{F' (q)}{T_r -T_\ell }\right) \right]\, ,
\end{equation}
where $F$ is the unique non decreasing solution of
\begin{equation}
\begin{cases}
{F''} = \cfrac{F-e}{F^2} (F')^2\, ,\\
F(0)=T_\ell , \; F(1) =T_r\, .
\end{cases}
\end{equation}
Surprisingly, the function ${\bb V}$ is independent of the pinning value $\nu$ and of the intensity of the noise $\gamma$. It is thus natural to conjecture that in the NESS of the unpinned velocity flip model the large deviation function of the empirical energy profile coincides with ${\bb V}$ but we did not succeed to prove it. Observe that at equilibrium ($T_\ell=T_r$), $F(q)=T_\ell=T_r$ and the last term in (\ref{eq:2:VKMP}) disappears so that the quasi potential is local. On the other hand, if $T_\ell \ne T_r$, this is no longer the case and this reflects the presence of long-range correlations in the NESS. In particular, an approximation of the NESS by a Gibbs local equilibrium state in the form (\ref{eq:GLES-ness}) would not give the correct value of the quasi potential.   

For the unpinned chain we have two conserved quantities. Solving the variational problem of the quasi potential for these two conserved quantities is a very difficult open problem{\footnote{Here we do not have any exactly solvable model like the Symmetric Simple Exclusion Process which could give us some hints for the form of the quasi-potential.}} (see \cite{B2} for a partial result for some other stochastic perturbation of the harmonic chain).

%%%%%%
%
%%% CHAPTER 4: ANOMALOUS DIFFUSION
%
%%%%%%%%%%%%%%%%%%%%%%%%%%%

\section{Anomalous diffusion}
\label{ch:anomalous}

An anomalous large conductivity is observed experimentally in carbon nanotubes and numerically in chains of oscillators without pinning, where numerical evidence shows a conductivity diverging with the size $N$ of the system like $N^{\alpha}$, with $\alpha<1$ in dimension $d=1$, and like $\log N$ in dimension $d=2$. If some nonlinearity is present in the interaction, finite conductivity is observed numerically in all pinned case or in dimension $d \ge 3$ (\cite{Dhar},\cite{LLP1}). Consequently it has been suggested that conservation of momentum is an important ingredient for the anomalous conductivity in low dimensions (see however \cite{ZZWZ}). 

In Chapter \ref{ch:normal} we considered chains of oscillators perturbed by a noise conserving only energy and destroying the possible momentum conservation. In the harmonic case we obtained Fourier's law and in the anharmonic case we proved existence of the Green-Kubo formula for the thermal conductivity. 

In this chapter the added perturbation conserves both energy and momentum (energy and volume for the Hamiltonian systems considered in Section \ref{sec:intro-phd}). These systems qualitatively have the same behavior as Hamiltonian chains of oscillators (without any noise), i.e. anomalous transport for unpinned chain in dimension $d=1,2$ and normal transport otherwise. We could even be more optimistic and hope that they share with the deterministic systems common limits for the energy fluctuation fields, two point correlation functions ${\ldots}$ This is because one expects that the microscopic details of the dynamics are irrelevant. Therefore some {\textit{universality}} should hold. Recently H. Spohn (\cite{Sp-NFH}), by following ideas of \cite{VB}, used the nonlinear fluctuating hydrodynamics theory  to classify very precisely the different expected universality classes. The nonlinear fluctuating hydrodynamics theory is based on the assumption that the microscopic dynamics evolve in the Euler time scale according to a system of conservation laws. The theory is macroscopic in the sense that all the predictions are done starting from this system of conservation laws without further references to the microscopic dynamics. Since we have seen that the presence of the energy-momentum conserving noise does not change the form of the hydrodynamic equations, the theory claims in fact that the limit of the fluctuations fields of the conserved quantities for purely deterministic chains of oscillators and for noisy energy-momentum conserving chains are exactly the same.

\subsection{Harmonic chains with momentum exchange noise}

Getting some information on the behavior of the energy fluctuation field in the large scale limit remains challenging. So far, satisfactory but not complete results have only been obtained in the harmonic case. The anharmonic case is much more difficult. 

In \cite{BBO1}, \cite{BBO2} we explicitly compute  the time correlation current for a system of harmonic oscillators perturbed by an energy-momentum conserving noise  {\footnote{It is straightforward to adapt the proofs given in \cite{BBO1}, \cite{BBO2} to the case of the momenta exchange noise.}} and we find that it behaves, for large times, like $t^{-d/2}$ in the unpinned cases, and like $t^{-d/2-1}$ when on-site harmonic potential is present.   

These results are given in the Green-Kubo formalism. Their counterpart in the NESS formalism have been considered in \cite{LMMP} but a rigorous proof is still missing. Several variations of the  Green-Kubo formula can be found in the literature: one can start with the infinite system in the canonical ensemble, as we did in Subsection~\ref{subsec:GK0}, or with a finite system, in the canonical or micro-canonical ensembles, sending the size of the system to infinity. It is widely believed that all these definitions coincide (also in the case of infinite conductivity).  As shown in \cite{BBO2}, this is essentially true for the energy-momentum conserving harmonic chain. Here we consider the simplest possible definition avoiding to discuss the rigorous definition of the canonical ensemble in infinite volume and the problem of equivalence of ensembles. 

The set-up is the following. We consider a chain perturbed by the energy-momentum conserving noise (see (\ref{eq:cons-noise-exch-mom})) with periodic boundary conditions. Its Hamiltonian is given by ${\mc H}_N=\sum_{x \in \TT_N^d}  {\mc E}_x$ where the energy ${\mc E}_x$ of atom $x$ is
\begin{equation}
\label{eq:4-energy-site-x}
{\mc E}_x= \frac{|p_x|^2}2 + W(q_x) + \frac{1}{2} \sum_{\substack{|y-x|=1}} V(q_x - q_y). 
\end{equation}

The system is considered at equilibrium under the Gibbs grand-canonical measure
\begin{equation*}
d\mu_{N,T} =\frac{e^{-{\mc H}_N /T}}{Z_{N,T}} d\bq d \bp
\end{equation*}
where $Z_{N,T}$ is the renormalization constant.

The Green-Kubo formula for the thermal conductivity in the direction $e_k$, $1 \le k \le d$, is {\footnote{By symmetry arguments this is independent of $k$.}} the limiting variance of the energy current $J^{e,\gamma}_{x,x+e_k} ([0,t])$ up to time $t$ in the direction $e_k$ in a space-time box of size $N \times t$:
\begin{equation}
\label{eq:5;GKtyu}
\kappa (T) =  \frac{1}{2T^2}  \lim_{t \to +\infty} \lim_{N \to \infty}{\mathbb E}_{\mu_{N,T}} \left[ \left(\frac{1}{\sqrt{N^{d} t}}\sum_{x \in \TT_N^d} J^{e,\gamma}_{x,x+e_1} ([0,t]) \right)^2 \right].
\end{equation}
The energy currents $\{ J^{e,\gamma}_{x,x+e_k} ([0,t])\, ; \, k=1, \ldots, d\}$ are defined by the energy conservation law
\begin{equation*}
{\mc E}_x (t) -{\mc E}_x (0) = \sum_{k=1}^d \left( J^{e,\gamma}_{x-e_k,x} ([0,t]) -J^{e, \gamma}_{x,x+e_k} ([0,t]) \right).
\end{equation*}
The energy current up to time $t$ can be written as
\begin{equation}
J^{e, \gamma}_{x,x+e_k} ([0,t]) = \int_0^t {j}_{x,x+e_k}^{e,\gamma} (s) ds + M_{x,x+e_k} (t)
\end{equation}
where $M_{x,x+e_k} (t)$ is a martingale and ${j}^{e,\gamma}_{x,x+e_k}$ is the instantaneous current which has the form
\begin{equation}
\label{eq:curebtjean}
{j}_{x,x+e_k}^{e,\gamma} ={\tilde j}^e_{x,x+e_k}+ \gamma \left[ p_{x+e_k}^2 -p_x^2 \right], \quad {\tilde j}^e_{x,x+e_k}= -\frac{1}{2} V' (q_{x+e_k} -q_x) (p_{x+e_k} + p_x).  
\end{equation}
The term ${\tilde j}_{x,x+e_k}^e$ is the Hamiltonian contribution while the gradient term is due to the noise. 

We now expand the square in (\ref{eq:5;GKtyu}). Notice first that since we have periodic boundary conditions the gradient term appearing in (\ref{eq:curebtjean}) does not contribute. By a time reversal argument one can show that the cross term between the martingale and the time integral of the instantaneous current vanishes. Moreover a simple computation shows that the square of the martingale term gives a contribution equal to $\gamma$ (see \cite{BBO2} for details). Thus we obtain
\begin{equation}
\label{eq:5;GKtyu6}
\begin{split}
\kappa (T) &= T^{-2}  \lim_{t \to +\infty} \lim_{N \to \infty} \frac{1}{2N^{d} t}{\mathbb E}_{\mu_{N,T}} \left[ \left( \sum_{x \in \TT_N^d} \int_{0}^t  {\tilde j}_{x,x+e_k}^e (s) ds \right)^2 \right] \; + \;  \gamma \\
&=T^{-2}  \lim_{t \to +\infty} \lim_{N \to \infty} \sum_{x \in \TT_N^d} \int_0^{+\infty} ds \left(1- \frac{s}{t} \right)^{+}  {\mathbb E}_{\mu_{N,T}} \left[ {\tilde j}^e_{0,e_k} (0) \,{\tilde  j}_{x,x+e_k}^e (s) \right] \, ds \; + \: \gamma
\end{split}
\end{equation}
where the last line is obtained by time and space stationarity of the Gibbs measure and $u^+$ denotes $\max (u,0)${\footnote{Observe that replacing $(1-\tfrac{s}{t})^+$ by $e^{-s/t}$ and $\lim_{N \to \infty} \sum_{x \in \TT_N^d}$ by $\sum_{x \in \ZZ^d}$ we formally get an expression similar to the Green-Kubo formula of Theorem \ref{th:GK}.}}. It is then clear that the divergence of the Green-Kubo formula, {\textit{i.e.}} anomalous transport, is due to a slow decay of the time correlation function $C(t)$ defined by 
\begin{equation}
\label{eq:4:CTTT}
C(t)= \lim_{N \to \infty} \sum_{x \in \TT_N^d} {\mathbb E}_{\mu_{N,T}} \left[ {\tilde j}^e_{0,e_k} (0) \, {\tilde j}_{x,x+e_k}^e (t) \right].
\end{equation} 

\begin{theorem}[\cite{BBO2}]
\label{th:GK-harm-exch}
Consider the harmonic case: $V (r )=\alpha r^2$, $W(q)= \nu q^2$ where $\alpha>0$ and $\nu \ge 0$. \\
Then the limit defining $C(t)$ in (\ref{eq:4:CTTT}) exists and can be computed explicitly. In particular, we have that $C(t) \sim t^{-d/2}$ if $\nu=0$ and $C(t) \sim t^{-d/2 -1}$ if $\nu>0$. \\
Consequently, the limit (\ref{eq:5;GKtyu6}) exists in $(0,+\infty]$ and is finite if and only if $d \ge 3$ or $\nu>0$. When finite, $\kappa (T)$ is independent of $T$ and can be computed explicitly.
\end{theorem}

\begin{proof}
We compute the Laplace transform $L_N (z)=\int_0^{+\infty} e^{-zt} C_N (t) dt$, $z>0$, of  $C_N (t)=\sum_{x \in \TT_N^d} {\mathbb E}_{\mu_{N,T}} \left[ {\tilde j}_{0,e_k} (0) \, {\tilde j}_{x,x+e_k}^e (t) \right]$. Since we have
\begin{equation*}
L_N (z) = N^{-1} \mu_{N,T} \left[ \left(\sum_{x \in \TT_N^d} {\tilde j}_{x,x+,e_k}^e \right) \, (z-{\mc L}_N)^{-1} \left( \sum_{x \in \TT_N^d} {\tilde j}_{x,x+,e_k}^e \right) \right]
\end{equation*} 
it is equivalent to solve the resolvent equation $(z-{\mc L}_N) h_N = \sum_{x \in \TT_N^d} {\tilde j}_{x,x+e_k}^e$. Notice that ${\mc L}_N$ maps polynomial functions of degree $2$ into polynomial functions of degree $2$ and that $\sum_{x} {\tilde j}^e_{x,x+e_k}$ is a polynomial function of degree $2$. Thus, the function $h_N$ is a polynomial function of degree $2$. Moreover it has to be space translation invariant since $\sum_{x} {\tilde j}^e_{x,x+e_k}$ is. Therefore we can look for a function $h_N$ of the form
\begin{equation*}
h_N = \sum_{x,y} a (y-x) p_x p_y + \sum_{x,y} b (y-x) p_x q_y + \sum_{x,y} c (y-x) q_x q_y
\end{equation*}
where $a,b$ and $c$ are functions from $\TT_N^d$ into $\RR$. We compute explicitly $a,b$ and $c$ and we get $a=c=0$ while $b$ is the solution to 
$$(z+2\nu -\gamma \Delta) b = - \alpha (\delta_{e_k} -\delta_{-e_k})$$
where $\Delta$ is the discrete Laplacian. Then we deduce $L_N (z)$, hence $C_N (t)$ by inverse Laplace transform. The limit $C(t) = \lim_{N \to + \infty} C_N (t)$ follows.
\cqfd\end{proof}

Consequently in the unpinned harmonic cases in dimension $d = 1$ and $2$, the conductivity
of our model diverges as $N$ goes to infinity. Otherwise it converges as $N \to \infty$. In the anharmonic case we obtained some upper bounds showing that the divergence cannot be worse than in the harmonic case. These upper bounds also show that the conductivity cannot be infinite if $d\ge 3$ (see \cite{BBO2} for details and precise statements).

\subsection{A class of perturbed Hamiltonian systems}

In \cite{BS} is proposed a class of models for which anomalous diffusion is observed. These models have been introduced in Section \ref{sec:intro-phd} of Chapter \ref{ch:models}. The goal of \cite{BS} was to show that these systems have a behavior very similar to that of the standard one-dimensional chains of oscillators conserving momentum {\footnote{They could be defined in any dimension.}}. 

\subsubsection{Definition of thermodynamic variables}
\label{sec:def_thermo}

Let us fix a potential $V$ and consider the stochastic energy-volume conserving model defined by the generator ${\mc L}={\mc A} + \gamma {\mc S}$, $\gamma \ge 0$, where ${\mc A}$ and ${\mc S}$ are given by (\ref{eq:1:ASdynum2}). Recall that the Gibbs grand-canonical probability measures $\mu_{\beta, \lambda}$, $\beta >0$, $\lambda \in \RR$, defined on $\Omega$ by
\begin{equation*}
d\mu_{\beta,\lambda} (\eta) =\prod_{x \in \ZZ} Z(\beta,\lambda)^{-1} \exp \left\{ -\beta V(\eta_x) -\lambda \eta_x  \right\} d\eta_x 
\end{equation*}
form a family of invariant probability measures for the infinite dynamics. We assume that 
the partition function $Z$ is well defined on $(0,+\infty) \times {\mathbb R}$. The following thermodynamic relations relate the chemical potentials $\beta, \lambda$  to the mean volume $v$ and the mean energy $e$ under $\mu_{\beta,\lambda}$:
\begin{equation}
\label{eq:tr}
\begin{split}
&v(\beta,\lambda) =\mu_{\beta,\lambda} (\eta_x)= -\partial_{\lambda} \Big(\log Z(\beta,\lambda)\Big), \\
&e(\beta,\lambda)= \mu_{\beta,\lambda} (V(\eta_x))= -\partial_{\beta}\Big(\log Z(\beta,\lambda)\Big).
\end{split}
\end{equation} 
These relations can be inverted by a Legendre transform to express $\beta$ and~$\lambda$ 
as a function of~$e$ and~$v$. 
Define the thermodynamic entropy $S \, : \, (0,+\infty) \times \RR \to [-\infty,+\infty)$ as
\begin{equation*}
S(e,v)= \inf_{ \lambda \in \RR, \beta >0} \Big\{ \beta e + \lambda v + 
\log Z (\beta,\lambda) \Big\}.
\end{equation*}
Let ${\mc U}$ be the convex domain of $(0,+\infty) \times \RR$ where 
$S(e,v) >- \infty$ and $\mathring{\mc U}$ its interior. 
Then, for any $(e,v):=(e(\beta, \lambda),v(\beta,\lambda)) 
\in {\mathring{\mc U}}$, the parameters $\beta,\lambda$ can be obtained as 
\begin{equation}
\label{eq:14}
\beta = (\partial_e S ) (e,v), 
\qquad 
\lambda= (\partial_v S) (e,v).
\end{equation} 
We also introduce the tension 
$\tau(\beta,\lambda)= \mu_{\beta,\lambda} (V' (\eta_0))=-\lambda / \beta$.
The microscopic energy current $j^{e,\gamma}_{x,x+1}$ and volume current $j^{v,\gamma}_{x,x+1}$ are given by
\begin{equation}
\label{eq:current_gamma}
\begin{split}
&j^{e,\gamma}_{x,x+1}= - V'(\eta_x) V' (\eta_{x+1}) - \gamma \nabla [ V(\eta_x)],\\
&j^{v,\gamma}_{x,x+1}= -[ V' (\eta_x) +V' (\eta_{x+1})] - \gamma \nabla [ \eta_x].
\end{split}
\end{equation} 
With these notations we have 
\begin{equation}
\label{eq:averages_of_currents}
\mu_{\beta,\lambda} (j_{x,x+1}^{e, \gamma})=-\tau^2, 
\qquad 
\mu_{\beta,\lambda} (j_{x,x+1}^{v,\gamma})=-2\tau.
\end{equation}
In the sequel, with a slight abuse of notation, 
we also write $\tau$ for $\tau(\beta(e,v), \lambda(e,v))$ 
where $\beta(e,v)$ and $\lambda(e,v)$ are defined by relations~(\ref{eq:14}).

\subsubsection{Hydrodynamic limits}
\label{sec:hydroLimEuler}
Consider the finite \emph{closed} stochastic energy-volume dynamics with periodic boundary conditions, 
that is the dynamics generated by ${\mc L}_{N,{\rm{per}}} ={\mc A}_{N,{\rm{per}}} + \gamma {\mc S}_{N,{\rm{per}}}$ where
\begin{equation}
\label{eq:A_per}
\Big({\mathcal A}_{N,{\rm per}} f \Big)(\eta) = 
\sum_{x \in \TT_N} \left[V'(\eta_{x+1})-V'(\eta_{x-1}) \right] \partial_{\eta_x}f(\eta),
\end{equation}
and
\begin{equation*}
\Big({\mc S}_{N,{\rm per}} f\Big)(\eta) =\sum_{x \in \TT_N} \left[ f(\eta^{x,x+1}) -f(\eta) \right].
\end{equation*}
We choose to consider the dynamics on $\TT_N$ rather than on $\ZZ$ to avoid (nontrivial) technicalities. We are interested in the macroscopic behavior of the two conserved quantities on a macroscopic time-scale $Nt$ as $N \to \infty$. 

\begin{remark} The results of this section shall be compared to the results of Section \ref{subsec:hl-vf}. For the velocity-flip model, the hydrodynamic limits where trivial in the Euler time scale. It was only in the diffusive time scale that some evolution of the profiles was observed and the hydrodynamic limits were given by parabolic equations (see (\ref{eq:2:hl-re1}). Here, the evolution is not trivial in the Euler time scale and the hydrodynamic limits are given by hyperbolic equations (see below (\ref{eq:syslim}). 
\end{remark}

We assume that the system is initially distributed according to a local Gibbs equilibrium  
state corresponding to a given energy-volume profile ${X_0} : \TT \to {\mathring{\mc U}}$: 
\begin{equation*}
X_0=\left(
\begin{array}{c}
{\mf e}_0\\
{\mf v}_0
\end{array} 
\right), 
\end{equation*}
in the sense that, for a given system size~$N$, 
the initial state of the system is described by the following product probability measure: 
\begin{equation}
\label{eq:Ges}
d\mu_{{\mf e}_0, {\mf v}_0}^N(\eta) = 
\prod_{x \in \TT_N} \frac{ \exp\left\{ -  {\mf \beta}_0 (x/N) V(\eta_x) 
-  {\mf \lambda}_0 (x/N) \eta_x \right\}}{Z ( {\mf \beta}_0 (x/N), {\mf \lambda}_0 (x/N))}\,
d\eta_x,
\end{equation} 
where $(\beta_0 (x/N), \lambda_0 (x/N))$ is actually a function of 
$({\mf e}_0 (x/N), {\mf v}_0 (x/N))$ through relations~(\ref{eq:14}).

Starting from such a state, we expect the state of the system at time $N t$ to be 
close, in a suitable sense, to a local Gibbs equilibrium measure 
corresponding to an energy-volume profile
\begin{equation*}
X(t,\cdot)=\left(
\begin{array}{c}
{\mf e}(t, \cdot)\\
{\mf v} (t,\cdot)
\end{array} 
\right),
\end{equation*}
satisfying a suitable partial differential equation with initial condition~$X_0$
at time $t = 0$. In view of~(\ref{eq:averages_of_currents}), and assuming propagation of local 
equilibrium, it is not difficult
to show that the expected partial differential equation 
is the following system of two conservation laws: 
%(note that the stochastic part 
%does not contribute, see the proof of Theorem~\ref{th:hl} 
%below for an explanation of this fact):
\begin{equation}
\label{eq:syslim}
\begin{cases}
&\partial_t {\mf e} -\partial_q \tau^2 =0,\\
&\partial_t {\mf v} - 2 \partial_{q} \tau=0,
\end{cases}
\end{equation} 
with initial conditions ${\mf e}(0,\cdot) = {\mf e}_0(\cdot), {\mf v}(0,\cdot) = {\mf v}_0(\cdot)$.
We write~(\ref{eq:syslim}) more compactly as  
\begin{equation*}
\partial_t X + \partial_q {\mf J}(X) =0, \qquad X (0,\cdot)=X_0 (\cdot),
\end{equation*}
with 
\begin{equation}
\label{eq:JJJ}
{\mf J} (X)= \left( 
\begin{array}{c}
-\tau^2 ( {\mf e}, {\mf v})\\
-2\tau ({\mf e}, {\mf v})
\end{array}
\right).
\end{equation}

The system of conservation laws (\ref{eq:syslim}) has other 
nontrivial conservation laws. In particular, the thermodynamic entropy~$S$ 
is conserved along a smooth solution of~(\ref{eq:syslim}):
\begin{equation}
\label{eq:consentropy}
\partial_t S ({\mf e},{\mf v})= 0.
\end{equation}
Since the thermodynamic entropy is a strictly concave function 
on $\mathring{\mc U}$, the system~(\ref{eq:syslim}) is strictly 
hyperbolic on $\mathring{\mc U}$ (see~\cite{Serre}). The two real 
eigenvalues of $(D{\mf J})({\bar \xi})$ are $0$ 
and $-\left[ \partial_{e} (\tau^2) + 2 \partial_v (\tau)\right]$, 
corresponding respectively to the two eigenvectors
\begin{equation}
\left(
\begin{array}{c}
-\partial_v \tau\\
\partial_e \tau
\end{array}
\right), \qquad
\left(
\begin{array}{c}
\tau\\
1
\end{array}
\right).
\end{equation}

It is well known that classical solutions to systems of $n \ge 1$ conservation laws in general develop shocks in finite times, even when starting from smooth initial conditions. If we consider weak solutions rather than classical solutions, then a criterion is needed to select a unique, relevant solution among the weak ones. For scalar conservation laws ($n=1$), this criterion is furnished by the so-called entropy inequality and existence and uniqueness of solutions is fully understood. If $n\ge2$, only partial results exist (see~\cite{Serre}). This motivates the fact that we restrict our analysis to smooth solutions before the appearance of shocks.

We assume that the potential $V$ satisfies the following

\begin{assumption}
\label{ass:V}
The potential $V$ is a smooth, non-negative function such that the partition function $Z(\beta, \lambda) = \int_{-\infty}^{\infty} \exp\left( -\beta V(r) -\lambda r \right)\, dr$ is well defined for $\beta>0$ and $\lambda \in \RR$ and  there exists a positive constant $C$ such that 
\begin{equation}
\label{eq:ass-pot1}
0 < V'' (r) \le C, 
\end{equation}
and
\begin{equation}
\label{eq:ass-pot2}
\limsup_{| r| \to + \infty} \frac{r V' (r)}{V(r)} \in(0, +\infty),
\end{equation}
\begin{equation}
  \label{eq:ass-pot3}
  \limsup_{| r| \to + \infty} \frac{[V' (r)]^2}{V(r)} < +\infty.  
\end{equation}
\end{assumption}

\emph{Provided} we can prove that the infinite volume
dynamics is macro-ergodic, then we can rigorously prove (even if $\gamma=0$),
using the relative entropy method of Yau (\cite{yau1}), that~(\ref{eq:syslim}) is
indeed the hydrodynamic limit in the smooth regime, \textit{i.e.}
for times~$t$ up to the appearance of the first shock (see for
example~\cite{KL,TV}). Observe that the expected hydrodynamic limits do not depend on $\gamma$. We need to assume $\gamma>0$ to ensure the macro-ergodicity of the dynamics.

\begin{remark}
As argued in~\cite{TV}, it turns out that the conservation of thermodynamic entropy (\ref{eq:consentropy}) is fundamental for Yau's method where, in the expansion of the time derivative of relative entropy, the cancelation of the linear terms is a consequence of the preservation of the thermodynamic entropy.
\end{remark}

Averages with respect to the empirical energy-volume measure are defined, for continuous functions $G,H : \TT \to \RR$, as (similarly to (\ref{eq:vf-empiricaldensities}))
\[
\left(\begin{array}{c} 
{\mc E}_N (t,G) \cr 
{\mc V}_N (t,H) 
\end{array} \right) 
%= \int_\TT \left( \begin{array}{c} G \\ H \end{array}\right) d\pi^N(t,\cdot)
= \left( \begin{array}{c}
\displaystyle \frac1N \sum_{x \in \TT_N} G\left(\frac{x}{N}\right) \, V(\eta_x (t) ) \cr
\displaystyle \frac{1}{N} \sum_{x \in \TT_N} H\left(\frac{x}{N}\right) \, \eta_x (t)
\end{array} \right).
\]   
We can then state the following result.

\begin{theorem}[\cite{BS}]
\label{th:4:hl}
Fix some $\gamma > 0$ and consider 
the dynamics on the torus ${\bb T}_N$ generated by ${\mc L}_{N,{\rm per}}$ where the potential $V$ satisfies Assumption~\ref{ass:V}. Assume that the system is initially distributed according to a local Gibbs state~(\ref{eq:Ges}) with smooth energy profile ${\mf e}_0$ and volume profile ${\mf v}_0$. Consider a positive time $t$ such that the solution $({\mf e}, {\mf v})$ to (\ref{eq:syslim}) belongs to~${\mathring{\mc U}}$ and is smooth on the time interval~$[0,t]$. Then, for any continuous test functions $G,H:\TT \to \RR$, the following convergence in
probability holds as $N \to +\infty$:
\[
\Big({\mc E}_N (tN, G), {\mc V}_N (tN,H)\Big) \longrightarrow 
\left(\int_{\TT} G(q) {\mf e} (t,q) dq,  \int_{\TT} H(q) {\mf v} (t,q) dq\right).
\]
\end{theorem}

The derivation of the hydrodynamic limits beyond the shocks for systems of $n \geq 2$ conservation laws is very difficult and is one of the most challenging problems in the field of hydrodynamic limits. The first difficulty is of course our poor understanding of the solutions to such systems. Recently, J. Fritz proposed in~\cite{F0} to derive hydrodynamic limits for hyperbolic systems (in the case $n=2$) by some extension of the compensated-compactness approach~\cite{Tartar} to stochastic microscopic models. This program has been achieved in~\cite{FT} (see also the recent paper \cite{MR2807137}), where the authors derive the classical $n=2$ Leroux system of conservation laws. In fact, to be exact, only the convergence to the set of entropy solutions is proved, the question of uniqueness being left open. It nonetheless remains the best result available at this time. The proof is based on a strict control of entropy pairs at the microscopic level by the use of logarithmic Sobolev inequality estimates. It would be very interesting to extend these methods to systems such as the ones considered here.

\subsubsection{Anomalous diffusion}

We investigate now the problem of anomalous diffusion of energy for these models.

If $V ( r )=r^2$ then Theorem \ref{th:GK-harm-exch} is mutatis mutandis valid and we get the same conclusions: the time-space correlations for the current behave for large time $t$ like $t^{-1/2}$ . Thus the system is super-diffusive (see \cite{BS} for the details).

For generic anharmonic potentials, we can only provide numerical evidence of the super-diffusivity.
However, it is difficult to estimate numerically the time autocorrelation functions of the 
currents because of their expected long-time tails, and because statistical errors are
very large (in relative value) when $t$ is large. 
Also, for finite systems (the only ones we can simulate on a computer), 
the autocorrelation is generically exponentially decreasing for anharmonic potentials,
and, to obtain meaningful results, the thermodynamic limit $N \to \infty$ should be taken before 
the long-time limit.

A more tenable approach consists in studying a nonequilibrium system in its steady-state.
We consider a finite system of length $2N+1$ in contact with two 
thermostats which fix the value of the energy at the boundaries. 
The generator of the dynamics is given by
\begin{equation}
  \label{eq:generator_open}
        {\mc L}_{N} = {\mc A}_{N} + \gamma {\mc S}_N + \lambda_\ell {\mc B}_{-N,T_\ell} + 
        \lambda_{r}  {\mc B}_{N,T_r},
\end{equation}
where ${\mc A}_N$ and ${\mc S}_N$ are defined  by 
\begin{equation*}
\begin{split}
&({\mathcal A}_{N}f)(\eta) =\sum_{x =-(N-1)}^{N-1} 
\Big(V'(\eta_{x+1})-V'(\eta_{x-1})\Big)(\partial_{\eta_x} f)(\eta) \\
&{\phantom{({\mathcal A}_{N}f)(\eta) =}}-V' (\eta_{N -1}) \, (\partial_{\eta_N} f)(\eta) + V' (\eta_{-N+1}) \, (\partial_{\eta_{-N}} f)(\eta),\\
&({\mc S}_N f)(\eta) =\sum_{x=-N}^{N-1} \left[ f(\eta^{x,x+1}) -f(\eta) \right],
\end{split}
\end{equation*}
and
${\mc B}_{x,T} = T \partial_{\eta_x}^2 -V' (\eta_x) \partial_{\eta_x}$. 
The positive parameters $\lambda_\ell$ and $\lambda_r$ are the intensities of the thermostats and $T_\ell, T_r$ the ``temperatures'' of the thermostats. 

The generator ${\mc B}_{x,T}$ is a thermostatting mechanism. In order to fix the energy at site $-N$ (resp. $N$) to the value $e_\ell$ (resp. $e_r$), we have to choose $\beta_\ell= T_{\ell}^{-1}$ (resp. $\beta_r = T_r^{-1}$) such that $e(\beta_\ell,0)=e_\ell$ (resp. $e(\beta_r, 0)=e_r$). We denote by $\langle \cdot \rangle_{\rm ss}$ the unique stationary state for the dynamics generated by ${\mc L}_{N}$.

The energy currents $j^{e,\gamma}_{x,x+1}$, which are such that
${\mc L}_{N,{\rm open}} (V (\eta_x)) = -\nabla j_{x-1,x}^{e,\gamma}$ (for $x=-N, \ldots,N-1$), 
are given by the first line of (\ref{eq:current_gamma})
for $x = -N+1,\dots,N-1$ while
\begin{equation*}
\begin{split}
&j^{e,\gamma}_{-N-1,-N} = \lambda_\ell \left[ T_\ell V'' (\eta_{-N}) -(V' (\eta_{-N}))^2 \right],\\
&j^{e,\gamma}_{N,N+1}   = -\lambda_r \left[ T_r V'' (\eta_{N}) -(V' (\eta_{N}))^2 \right].
\end{split}
\end{equation*}
Since $\langle {\mc L}_{N,{\rm open}} (V (\eta_x)) \rangle_{\rm ss} = 0$,
it follows that, for any $x = -N,\ldots,N+1$, 
$\langle j_{x,x+1}^{e,\gamma} \rangle_{\rm ss}$ is equal to a 
constant $J_N^\gamma(T_\ell,T_r)$ independent of $x$. In fact,
\begin{equation}
\label{eq:total_current_NESS}
J^{\gamma}_N(T_\ell,T_r) = \left\langle \mathcal{J}_N^\gamma 
\right\rangle_{\rm ss}, 
\qquad
\mathcal{J}_N^\gamma = \frac{1}{2N}\sum_{x=-N-1}^{N} j_{x,x+1}^{e,\gamma}.
\end{equation} 
The latter equation is interesting from a numerical viewpoint since it allows to 
perform some spatial averaging, hence reducing the statistical error of the results. We estimate the 
exponent $\delta \ge 0$ such that 
\begin{equation}
  \label{eq:def_delta}
  \kappa (N) :=N J_N^{\gamma} \sim N^{\delta}
\end{equation}
using numerical simulations.  If $\delta=0$, the system is a normal conductor of energy.
If on the other hand $\delta > 0$, it is a superconductor. 

The numerical simulations giving the value of $\delta$ are summarized in Table 1. %{\ref{tabtabexponents}} 
They have been performed for the harmonic chain $V ( r ) =r^2 /2$, the quartic potential  $V ( r ) =r^2/2 + r^4/4$ and the exponential potential $V ( r ) =  e^{-r}+r-1$. In Section~\ref{sec:exp.case} we will motivate our interest in the exponential potential.
\begin{table}
\captionsetup{justification=centering}
\label{tabtabexponents} 
\begin{center}
\begin{tabular}{|c|c|c|c|}
\hline
$\gamma$ & $V ( r ) =r^2 /2$   & $V ( r ) =r^2/2 + r^4/4$  & $V ( r ) =  e^{-r}+r-1$ \\
\hline
$0$               & $1$     & $0.13$ & $1$    \\
$0.01$\phantom{0} & --    & 0.14 & 0.12 \\
$0.1$\phantom{00} & 0.50  & 0.27 & 0.25 \\
$1$\phantom{0.00} & 0.50  & 0.43 & 0.33 \\ 
\hline
\end{tabular}
\end{center}

\caption{Conductivity exponents}
\end{table}

Exponents in the harmonic case agree with their expected values. For nonlinear potentials, except for the singular value $\delta = 1$ when $\gamma = 0$ and $V ( r ) = e^{-r}+r-1$, the exponents seem to be monotonically increasing with $\gamma$. A similar behavior of the exponents is observed for Toda chains~\cite{ILOS} with a momentum conserving noise. This strange behavior casts some doubts on the convergence of conductivity exponents $\delta$ with respect to system size $N$ (see the comment after Theorem 3 in \cite{BG}). A detailed study, including the nonlinear fluctuating hydrodynamics predictions, is available in \cite{Spohn-Stoltz}. 

Note also that the value found for $\gamma = 0$ with the anharmonic FPU potential $V(r) = r^2/2 + r^4/4$ is smaller than the corresponding value for standard oscillators chains, which is around 0.33 (see~\cite{MDN07}). We performed also numerical simulations for a ``rotor'' model, $V ( r ) = 1- \cos ( r )$, and we found $\delta \approx 0.02$, i.e. a normal conductivity. A similar picture is observed for the usual rotor {\footnote{The variable $r$ has  to be interpreted as an angle and belongs to the torus $2 \pi \TT$.}} model which is composed of a chain of unpinned oscillators with interaction potential $V ( r )=1-\cos ( r )$. The normal behavior is conjectured to be due to the absence of long waves (that carry energy ballistically) because some rotors turning fast in between will break them (\cite{ILOS-11}). See \cite{Spohn2014-rotators} and references therein for a recent study of the rotors model. 

%\begin{table}
%\caption{\label{tab:exponents} Conductivity exponents $\delta$.} 
%\begin{indented}
%\item[]\begin{tabular}{@{}llll}
%\br
%$\mathbf{\gamma}$ & $V ( r ) =r^2 /2$   & $V ( r ) =r^2/2 + r^4/4$  & $V ( r ) =  e^{-r}+r-1$ \\
%
%\mr
%$0$               & 1     & 0.13 & 1    \\
%$0.01$\phantom{0} & --    & 0.14 & 0.12 \\
%$0.1$\phantom{00} & 0.50  & 0.27 & 0.25 \\
%$1$\phantom{0.00} & 0.50  & 0.43 & 0.33 \\ 
%\hline
%\end{tabular}
%\end{indented}
%\end{table

\subsection[Harmonic interactions]{Fractional superdiffusion for a harmonic chain with bulk noise}

In this section we consider the energy-volume conserving model with quadratic potential. Fix $\lambda \in \bb R$ and $\beta>0$, and consider the process $\{\eta (t) ; t \geq 0\}$ generated by (\ref{eq:1:ASdynum2}) with $V(\eta)= \eta^2 /2$ and with initial distribution $\mu_{\beta, \lambda}$. Notice that the distribution of the process $\{\eta (t) +\rho; t \geq 0\}$ with initial measure $\mu_{\beta, \rho+\lambda}$ is the same for all values of $\lambda \in \RR$. Therefore, we can assume, without loss of generality, that $\lambda = 0$. We write $\mu_\beta= \mu_{\beta,0}$ to simplify notation, and denote by $\bb P$ the law of $\{\eta (t) ; t \geq 0\}$ and by $\bb E$ the expectation with respect to $\bb P$. The {\em energy correlation function} $\{S_t (x); x \in \bb Z, t \geq 0\}$ is defined as
\begin{equation}
S_t(x) = \tfrac{\beta^2}{2} \;   \bb E\big[\big(\eta_0(0)^2-\tfrac{1}{\beta} \big) \big( \eta_x (t)^2 -\tfrac{1}{\beta} \big)\big]
\end{equation}
for any $x \in \bb Z$ and any $t \geq 0$. The constant $\frac{\beta^2}{2}$ is just the inverse of the variance of $\eta_x^2-\frac{1}{\beta}$ under $\mu_\beta$. By translation invariance of the dynamics and the initial distribution $\mu_\beta$, we see that
\begin{equation}
\tfrac{\beta^2}{2}\,  \bb E\big[ \big(\eta_x(0)^2 -\tfrac{1}{\beta} \big) \big(\eta_y (t)^2 -\tfrac{1}{\beta}\big)\big]=S_t(y-x)
\end{equation}
for any $x, y \in \bb Z$.

\begin{theorem} [\cite{BGJ}]
\label{t1}
Let $f,g: \bb R \to \bb R$ be smooth functions of compact support. Then,
\begin{equation}
\label{ec1.12}
\lim_{n \to \infty} \tfrac{1}{n}\sum_{x, y  \in \bb Z} f\big( \tfrac{x}{n} \big) g\big( \tfrac{y}{n}\big) S_{tn^{3/2}}(x-y) = \iint f(x)g(y) P_t(x-y) dx dy,
\end{equation}
where $\{P_t(x); x \in \bb R, t \geq 0\}$ is the fundamental solution {\footnote{Since the skew fractional heat equation is linear, it can be solved explicitly by Fourier transform.}} of the skew fractional heat equation on $\RR$
\begin{equation}
\label{ec1.13}
\partial_t u =  -\tfrac{1}{\sqrt{2}}\big\{ (-\Delta)^{3/4} - \nabla (-\Delta)^{1/4}\big\} u.
\end{equation}
\end{theorem}

A fundamental step in the proof of this theorem is the analysis of the correlation function $\{{\bb S}_t(x,y); x \neq y \in \bb Z, t \geq 0\}$ given by
\begin{equation}
{\bb S}_t(x,y) = \tfrac{\beta^2}{2} \bb E \big[ \big(\eta_0 (0)^2-\tfrac{1}{\beta} \big) \eta_x(t) \eta_y(t)\big]
\end{equation}
for any $t \geq 0$ and any $x \neq y \in \bb Z$. Notice that this definition makes perfect sense for $x =y$ and, in fact, we have ${\bb S}_t(x,x) = S_t(x)$. For notational convenience we define ${\bb S}_t(x,x)$ as equal to $S_t(x)$. However, these quantities are of different nature, since $S_t(x)$ is related to {\em energy fluctuations} and ${\bb S}_t(x,y)$ is related to {\em volume fluctuations} (for $x \neq y$).

\begin{remark}
It is not difficult to see that with a bit of technical work the techniques actually show that the distribution valued process $\{ {\mc E}_t^n (\cdot) \, ; \, t\ge 0\}$ defined for any test function $f$ by
$${\mc E}_{t}^n (f) = \cfrac{1}{\sqrt n} \sum_{x \in \ZZ} f \big (\tfrac{x}{n} \big) \{ \eta_x (tn^{3/2})^2 - \tfrac{1}{\beta}\}$$
converges as $n$ goes to infinity to an infinite dimensional $3/4$-fractional Ornstein-Uhlenbeck process, i.e. the centered Gaussian process with covariance prescribed by the right hand side of (\ref{ec1.12}).
\end{remark}

\begin{remark}
It is interesting to notice that $P_t$ is the maximally asymmetric $3/2$-Levy distribution.
It has power law as $|x|^{-5/2}$ towards the diffusive peak and stretched exponential as $\exp[-|x|^3]$ towards the exterior of the sound cone (\cite[Chapter 4]{UZ}). As mentioned to us by H. Spohn, this reflects the expected physical property that  no propagation beyond the sound cone occurs.
\end{remark}

\begin{remark}
With a bit of technical work the proof of this theorem can be adapted to obtain a similar statement for a harmonic chain perturbed by the momentum exchanging noise (see \cite{JKO2} where such statement is proved for the Wigner function). In this case the skew fractional heat equation is replaced by the (symmetric) fractional heat equation.
\end{remark}

\begin{proof}
Denote by $\mc C_c^\infty(\bb R^d)$ the space of infinitely differentiable functions $f: \bb R^d \to \bb R$ of compact support. Then, $\| f \|_{2,n}$ denotes the weighted $\ell^2(\bb Z^d)$-norm
\begin{equation}
\|f\|_{2,n} = \sqrt{ \vphantom{H^H_H}\smash{\tfrac{1}{n^d} \sum_{x \in \bb Z^d} f\big( \tfrac{x}{n} \big)^2}}.
\end{equation}

Let $g \in \mc C_c^\infty(\bb R)$ be a fixed function.
For each $n \in \bb N$ we define the field $\{\mc S_t^n; t \geq 0\}$ as
\begin{equation}
\mc S_t^n(f) = \tfrac{1}{n}\!\! \sum_{x, y \in \bb Z} g\big(\tfrac{\vphantom{y}x}{n} \big) f\big( \tfrac{y}{n}\big) S_{tn^{3/2}}(y-x)
\end{equation}
for any $t \geq 0$ and any $f \in \mc C_c^\infty(\bb R)$. 
\noindent
By the Cauchy-Schwarz inequality we have the {\em a priori} bound
\begin{equation}
\label{ec3.4}
\big| \mc S_t^n(f) \big| \leq \|g\|_{2,n} \| f\|_{2,n}
\end{equation}
for any $t \geq 0$, any $n \in \bb N$ and any $f,g \in \mc C_c^\infty(\bb R)$. For a function $h \in \mc C_c^\infty(\bb R^2)$ we define $\{ Q_t^n(h); t \geq 0\}$ as
\begin{equation}
Q_t^n(h) = \tfrac{1}{n^{3/2}} \sum_{x \in \ZZ} \; \sum_{y\ne z \in \bb Z} g\big(\tfrac{\vphantom{y}x}{n} \big) h \big(\tfrac{y}{n}, \tfrac{\vphantom{y} z}{n}\big) {\bb S}_{t n^{3/2}}(y-x,z-x).
\end{equation}
Notice that $Q_t^n(h)$ depends only on the symmetric part of the function $h$. Therefore, we will always assume, without loss of generality, that $h(x,y) = h(y,x)$ for any $x,y \in \bb Z$. We point out that $Q_t^n(h)$ does not depend on the values of $h$ at the diagonal $\{x=y\}$. We have the {\em a priori} bound
\begin{equation}
\label{ec3.6}
\big| Q_t^n(h) \big| \leq 2 \|g\|_{2,n} \|{\tilde h} \|_{2,n},
\end{equation}
where $\tilde h$ is defined by ${\tilde h} \big(\tfrac{\vphantom{y}x}{n} , \tfrac{y}{n}\big) = h \big(\tfrac{\vphantom{y}x}{n} , \tfrac{y}{n}\big) \, {\bf 1}_{x \ne y}.$

For a function $f \in \mc C_c^\infty(\bb R)$, we define  $\Delta_n f: \bb R \to \bb R$ as
\begin{equation}
\Delta_n f \big( \tfrac{x}{n}\big) = n^2\Big( f\big(\tfrac{x\plus 1}{n} \big) + f\big( \tfrac{x\minus 1}{n} \big) - 2 f\big(\tfrac{x}{n}\big) \Big).
\end{equation}
In other words, $\Delta_n f$ is a discrete approximation of the second derivative of $f$. We also define $\nabla_n f \otimes \delta : \sfrac{1}{n} \bb Z^2 \to \bb R$ as
\begin{equation}
\label{eq:3.9}
\big(\nabla_n f \otimes \delta\big) \big( \tfrac{x}{n}, \tfrac{y}{n}\big) =
\begin{cases}
\frac{n^2}{2}\big(f\big(\tfrac{x+1}{n}\big) - f\big(\tfrac{x}{n}\big)\big); & y =x\plus 1\\
\frac{n^2}{2}\big(f\big(\tfrac{x}{n}\big) - f\big(\tfrac{x-1}{n}\big)\big); & y =x\minus 1\\
0; & \text{ otherwise.}
\end{cases}
\end{equation}
Less evident than the interpretation of $\Delta_n f$, $\nabla_n f \otimes \delta$ turns out to be a discrete approximation of the (two dimensional) distribution $f'\!(x) \otimes \delta(x=y)$, where $\delta(x=y)$ is the $\delta$ of Dirac at the line $x=y$. We have that
\begin{equation}
\label{ec3.10}
\tfrac{d}{dt} \mc S_t^n(f) = -2Q_t^n(\nabla_n f \otimes \delta) +  \mc S_t^n(\tfrac{1}{\sqrt n}\Delta_n f).
\end{equation}
In this equation we interpret the term $Q_t^n(\nabla_n f \otimes \delta)$ in the obvious way.
By the {\em a priori} bound \eqref{ec3.4}, the term $ \mc S_t^n(\frac{1}{\sqrt n} \Delta_n f)$ is negligible, as $n \to \infty$.We describe now the equation satisfied by $Q_t^n(h)$. For this we need some extra definitions. For $h \in \mc C_c^\infty(\bb R^2)$ we define $\Delta_n h : \bb R^2 \to \bb R$ as
\begin{equation}
\Delta_n h\big( \tfrac{\vphantom{y}x}{n}, \tfrac{y}{n}\big) = n^2\Big( h\big( \tfrac{\vphantom{y}x+1}{n}, \tfrac{y}{n}\big)+h\big( \tfrac{\vphantom{y}x-1}{n}, \tfrac{y}{n}\big)+h\big( \tfrac{\vphantom{y}x}{n}, \tfrac{y+1}{n}\big)+ h\big( \tfrac{\vphantom{y}x}{n}, \tfrac{y-1}{n}\big) - 4 h\big( \tfrac{\vphantom{y}x}{n}, \tfrac{y}{n}\big)\Big).
\end{equation}
In words, $\Delta_n h$ is a discrete approximation of the $2d$ Laplacian of $h$. We also define $\mc A_n h: \bb R \to \bb R$ as
\begin{equation}
\mc A_n h\big( \tfrac{\vphantom{y}x}{n}, \tfrac{y}{n}\big) = n\Big(h\big( \tfrac{\vphantom{y}x}{n}, \tfrac{y-1}{n}\big)+ h\big( \tfrac{\vphantom{y}x-1}{n}, \tfrac{y}{n}\big)- h\big( \tfrac{\vphantom{y}x}{n}, \tfrac{y+1}{n}\big)-h\big( \tfrac{\vphantom{y}x+1}{n}, \tfrac{y}{n}\big)\Big).
\end{equation}
The function $\mc A_n h$ is a discrete approximation of the directional derivative $(-2,-2) \cdot \nabla h$. Let us define $\mc D_n h : \sfrac{1}{n} \bb Z \to \bb R$ as
\begin{equation}
\mc D_n h\big( \tfrac{x}{n} \big) = n \Big( h \big(\tfrac{x}{n}, \tfrac{x+1}{n}\big) - h \big( \tfrac{x-1}{n}, \tfrac{x}{n} \big) \Big)
\end{equation}
and $\widetilde {\mc D}_n h :\sfrac{1}{n} \bb Z^2 \to \bb R$ as
\begin{equation}
\widetilde{\mc D}_n h (\tfrac{x}{n},\tfrac{y}{n}) =
\begin{cases}
n^2 \big(h\big(\tfrac{x}{n}, \tfrac{x+1}{n}\big)-h\big(\tfrac{x}{n}, \tfrac{x}{n}\big)\big); & y =x+1\\
n^2 \big(h\big(\tfrac{x-1}{n}, \tfrac{x}{n}\big)-h\big(\tfrac{x-1}{n}, \tfrac{x-1}{n}\big)\big); & y=x-1\\
0; & \text{ otherwise.}
\end{cases}
\end{equation}
The function $\mc D_n h$ is a discrete approximation of the directional derivative of $h$ along the diagonal $x=y$, while $\widetilde{\mc D}_n h$ is a discrete approximation of the distribution $\partial_y h(x,x) \otimes \delta(x=y)$.
Finally we can write down the equation satisfied by the field $Q_t^n(h)$:
\begin{equation}
\label{ec3.15}
\tfrac{d}{dt} Q_t^n(h) = Q_t^n\big(n^{-1/2}\Delta_n h+ n^{1/2}\mc A_n h\big) -2 \mc S_t^n\big( \mc D_n h\big) + 2 Q_t^n\big(n^{-1/2}\widetilde{\mc D}_n h\big).
\end{equation}

Given $f \in \mc C_c^\infty(\bb R)$, if we choose $h:=h_n (f)$ such that
$$n^{-1/2}\Delta_n h+ n^{1/2}\mc A_n h = 2 \nabla_n f \otimes \delta$$
then summing  \eqref{ec3.10} and \eqref{ec3.15} we get
\begin{equation*}
\tfrac{d}{dt} \mc S_t^n(f) = - \tfrac{d}{dt} Q_t^n(h) +  \mc S_t^n(\tfrac{1}{\sqrt n}\Delta_n f) -2 \mc S_t^n\big( \mc D_n h\big) + 2 Q_t^n\big(n^{-1/2}\widetilde{\mc D}_n h\big).
\end{equation*}
We integrate in time the previous expression. By the a priori bounds, the term $\int_0^t \mc S_s^n(\tfrac{1}{\sqrt n}\Delta_n f) ds $ is small as well as $\int_0^t \tfrac{d}{ds} Q_s^n (h) ds =Q_t^n(h) - Q_0^n (h)$. The term $$\int_0^t Q_s^n( n^{-1/2} \widetilde{\mc D}_n h)\,  ds$$ is quite singular since it involves an approximation of a distribution but it turns out to be negligible, although this does not follow directly from the a priori bounds (see \cite{BGJ}). By using Fourier transform one can see that $-2 \mc D_n h$ converges to $ -\tfrac{1}{\sqrt{2}}\big\{ (-\Delta)^{3/4} - \nabla (-\Delta)^{1/4}\big\} f$ and we are done.

\cqfd\end{proof}

\subsection[Exponential interactions]{Anomalous diffusion for a perturbed Hamiltonian system with exponential interactions}
\label{sec:exp.case}
We investigate here in more details the exponential case $V_{\rm{exp}} ( r ) = e^{-r} -1 +r$. The deterministic system with generator (\ref{eq:A_per}) and with the exponential potential above is well known in the integrable systems literature {\footnote{It seems that although the Hamiltonian structure of the Kac-van-Moerbecke system was known, the interpretation of the latter as a chain of oscillators with exponential kinetic energy and exponential interaction was not observed before \cite{BS}.}}. It has been introduced in \cite{KVM} by Kac and van Moerbecke and was shown to be completely integrable. Consequently, using Mazur's inequality, it is easy to show that the energy transport is ballistic (\cite{BS}). 

As we will see the situation dramatically changes when the momentum exchange noise is added: the energy transport is no more ballistic but superdiffusive. Thus the situation is similar to the harmonic case. Nevertheless we expect the time autocorrelation of the current to decay like $t^{-2/3}$. We are not able to show this but we proved in \cite{BG} lower bounds sufficient to imply superdiffusivity.

The results are stated in infinite volume: we consider the stochastic energy-volume conserving dynamics $\{\eta (t)\}_{t \ge 0}$ with potential $V:=V_{\rm{exp}}$. Its generator is given by ${\mc L}={\mc A} + \gamma {\mc S}$ where ${\mc A}$ and ${\mc S}$ are defined by (\ref{eq:1:ASdynum2}). Since the exponential potential grows very fast as $r\to -\infty$, some care has to be taken to show that the infinite dynamics is well defined (see \cite{BG}). We recall that grand canonical Gibbs measures are denoted by $\mu_{\beta,\lambda}$ and  take the form
\begin{equation*}
d\mu_{\beta,\lambda} (\eta) = \prod_{x \in \ZZ} \frac{{\rm e}^{- \beta V(\eta_x) -\lambda \eta_x} }{Z(\beta,\lambda)} d\eta_x, \quad \beta>0, \; \lambda + \beta <0.
\end{equation*}
In this section, $\beta$ and $\lambda$ are fixed and we denote by $e$ (resp. $v$) the average energy (resp. volume) w.r.t. $\mu_{\beta,\lambda}$ (see (\ref{eq:tr})). 

The microscopic energy current ${j}^{e,\gamma}_{x,x+1}$ and volume current ${j}^{v,\gamma}_{x,x+1}$ are given by
\begin{equation*}
{j}^{e,\gamma}_{x,x+1}(\eta)=-e^{-(\eta_x + \eta_{x+1})}+(e^{- \eta_x} +e^{- \eta_{x+1}})- \gamma \nabla (V (\eta_{x}))
\end{equation*}
and
\begin{equation*}
{j}^{v,\gamma}_{x,x+1}(\eta)=e^{-\eta_x} +e^{-\eta_{x+1}} -\gamma \nabla \eta_x .
\end{equation*}

We will use the compact notations  
\begin{equation*}
\omega_x= \left( \begin{array}{c} V(\eta_x) \\ \eta_x \end{array} \right), \quad J_{x,x+1}= \left( \begin{array}{c} j^{e,\gamma}_{x,x+1} \\ j^{v,\gamma}_{x,x+1} \end{array} \right).
\end{equation*}

In the hyperbolic scaling, the hydrodynamical equations for the energy profile ${\mf e}$ and the volume profile ${\mf v}$  take the form
\begin{equation}
\label{eq:hl1euler}
\begin{cases}
\partial_t {\mf e} -\, \partial_q ( ({\mf e} - {\mf v})^2) =0\\
\partial_{t} {\mf v} + 2 \, \partial_q ({\mf e -{\mf v}}) =0.
\end{cases}
\end{equation}
They can be written in the compact form $\partial_t {{X}} + \partial_q {{ \mf J}} ({{X}}) =0$ with
\begin{equation}\label{eq:hl-ss00}
{X}=
\left(
\begin{array}{c}
{\mf e}\\
{\mf v}
\end{array}
\right)
, \quad
{ \mf J }({ {X}})=
\left(
\begin{array}{c}
- ( {\mf e} - {\mf v})^2\\
2  ( {\mf e} - {\mf v})
\end{array}
\right).
\end{equation}
The differential matrix of ${ \mf J}$ is given by
\begin{equation*}
(\nabla{ \mf J})({X})=
2 \left(
\begin{array}{cc}
- ( {\mf e} - {\mf v})& {\mf e} - {\mf v} \\
1 & -1
\end{array}
\right).
\end{equation*}
For given $({e}, {v})$ we denote by $({ T}^{+}_t )_{t \ge 0}$ (resp. $({T}^{-}_t)_{t \ge 0}$) the semigroup on $S(\RR) \times S(\RR)$ generated by the linearized system
\begin{equation}\label{linearized system 1}
\partial_t \varepsilon + {M}^T\, \partial_q \varepsilon =0, \quad ({\text{resp.}} \; \partial_t \varepsilon - {M}^T \, \partial_q \varepsilon =0 ),
\end{equation}
where
\begin{equation*}
{M} := {M} ({e}, {v})= [\nabla { \mf J}] (\omega), \quad \omega= \left(
\begin{array}{c}
{ e}\\
{v}
\end{array}
\right).
\end{equation*}
We omit the dependence of these semigroups on $(e,v)$ for lightness of the notations. Above $S(\RR)$ denotes the Schwartz space of smooth rapidly decreasing functions.

The first result of \cite{BG} gives a lower bound on the time-scale for which a non-trivial evolution of the energy-volume fluctuation field can be observed. 

We take the infinite system at equilibrium under the Gibbs measure $\mu_{\bar \beta,\bar \lambda}$ corresponding to a mean energy $\bar e$ and a mean volume $\bar v$. Our goal is to study
the energy-volume fluctuation field in the time-scale $tn^{1+\alpha}$, $\alpha \ge0$:
\begin{equation}
  \label{eq:YY}
  \mathcal{Y}^{n,\alpha}_t  (\bG) =\frac{1}{\sqrt{n}} \sum_{x\in \ZZ}
  \bG\left(x/n\right) \cdot \left({\omega}_x (tn^{1+\alpha})  - {\bar{\omega}}\right),
\end{equation}
where for $q\in{\mathbb{R}}$,  $x \in \ZZ$,
\[
\bG(q) = \left(
\begin{array}{c}
G_{1} (q)\\
G_2 (q)
\end{array}
\right), \quad
{\omega}_x= \left(
\begin{array}{c}
V (\eta_x) \\
\eta_x
\end{array}
\right)
\]
and $G_1, G_2$ are test functions belonging to $S(\RR)$.

We need to introduce some notation. For each $z \ge 0$, let $H_z (x) = (-1)^z e^{x^2} \tfrac{d^z}{dx^z} e^{-x^2}$ be the Hermite polynomial of order $z$ and $h_z (x) =(z! {\sqrt{2\pi}})^{-1} H_{z} (x) e^{-x^2}$ the Hermite function. The set $\{ h_z, z\ge 0\}$ is an orthonormal basis of ${\bb L}^2 (\RR)$. Consider in ${\bb L}^2(\RR)$ the operator $K_0 = x^2-\Delta$, $\Delta$ being the Laplacian on $\RR$. For an integer $k \ge 0$, denote by ${\bb H}_k$ the Hilbert space obtained by taking the completion of  $S (\RR)$ under the norm induced by the scalar product $\langle \cdot,\cdot\rangle_{k}$ defined by $\langle f, g \rangle_k= \langle f, K_0^k g \rangle_0$, where $\langle \cdot,\cdot\rangle_0$ denotes the inner product of $\LL^2 (\RR)$ and denote by ${\bb H}_{-k}$ the dual of ${\bb H}_k$, relatively to this inner product. Let $\langle\cdot\rangle$ represent the average with respect to the Lebesgue measure.

If $E$ is a Polish space then $D(\RR^+, E)$ (resp, $C(\RR^+, E)$) denotes the space of $E$-valued functions, right continuous with left limits (resp. continuous), endowed with the Skorohod (resp. uniform) topology. Let $Q^{n,\alpha}$ be the probability measure on ${D}(\RR^+,{\bb H}_{-k} \times {\bb H}_{-k})$ induced by the fluctuation field ${\mc Y}^{n,\alpha}_t$ and $\mu_{\beta,\lambda}$. Let $\mathbb{P}_{\mu_{ \beta,\lambda}}$ denote the probability measure on ${ D}(\RR^+, \RR^{\ZZ})$ induced by $(\eta(t))_{t\geq{0}}$ and $\mu_{\beta, \lambda}$. Let $\mathbb{E}_{\mu_{\beta,\lambda}}$ denote the expectation with respect to $\mathbb{P}_{\mu_{\beta,\lambda}}$.

\begin{theorem}[\cite{BG}]
\label{th:fluct-hs}
Fix an integer $k>2$. Denote by $Q$ the probability measure on $C(\RR^+, {\bb H}_{-k} \times {\bb H}_{-k})$ corresponding to a stationary Gaussian process with mean $0$ and covariance given by
\begin{equation*}
{\mathbb E}_{Q} \left[ \mathcal{Y}_t (\mb H) \, \mathcal{ Y}_s (\mb G) \right] =  \langle \,{T}_t^{-} \mb H\;  \cdot \;   \chi \;  {T}_s^{-}\mb G \,  \rangle
\end{equation*}
for every $0 \le s \le t$ and $\mb H, \mb G$ in ${\bb H}_k \times {\bb H}_k$. Here ${\chi}:={\chi} ({ \beta}, {\lambda})$ is the equilibrium covariance matrix of ${ \omega}_0$. Then, the sequence $(Q^{n,0})_{n \ge 1}$ converges weakly to the probability measure $Q$.
\end{theorem}

The theorem above means that in the hyperbolic scaling the fluctuations are trivial: the initial fluctuations are transported by the linearized system of (\ref{eq:hl1euler}). To see a nontrivial behavior we have to study, in the transport frame, the fluctuations at a longer time scale $t n^{1+\alpha}$, with $\alpha>0$. Thus, we consider the fluctuation field ${\widehat {\mc Y}}_{\cdot}^{n,\alpha}$, $\alpha>0$, defined, for any $\bG \in S(\RR) \times S(\RR)$, by
\begin{equation}\label{longer density field}
{\widehat {\mc Y}}_t^{n,\alpha} (\bG)= {\mc Y}_t^{n, \alpha} \left( {T}^+_{t n^{\alpha}} \bG \right).
\end{equation}

Our second main theorem shows that the correct scaling exponent $\alpha$ is greater than $1/3$:

\begin{theorem}[\cite{BG}]
\label{th:fluct-ds}
Fix an integer $k>1$ and $\alpha<1/3$.  Denote by $Q$ the probability measure on $C(\RR^+, {\bb H}_{-k} \times {\bb H}_{-k})$ corresponding to a stationary Gaussian process with mean $0$ and covariance given by
\begin{equation*}
{\mathbb E}_{Q} \left[ \mathcal{Y}_t (\mb H) \, \mathcal{ Y}_s (\mb G) \right] =  \langle \,  \mb H \;\cdot    \;  {\chi} \;  \mb G   \rangle
\end{equation*}
for every $0 \le s \le t$ and $\mb H, \mb G$ in ${\bb H}_k \times {\bb H}_k$. Then, the sequence $(Q^{n,\alpha})_{n \ge 1}$ converges weakly to the probability measure $Q$.
\end{theorem}

The proofs of these theorems can be reduced to the proof of a so-called {\textit{equilibrium Boltzmann-Gibbs principle}}. Let us explain what it means. Observables can be divided into two classes: non-hydrodynamical and hydrodynamical. The first ones are the non conserved quantities and they fluctuate on a much faster scale than the conserved ones. Hence, they should average out and only their projection on the hydrodynamical variables should persist in the scaling limit. 
For any local function $g:= g(\eta)$, the projection ${\mc P}_{e, {v}} \, g$ of $g$ on the fields of the conserved quantities is defined by
\begin{equation*}
({\mc P}_{{ e}, { v}} g)(\eta)=g (\eta) -{\tilde g} ( e,  v) -(\nabla {\tilde g})(e, v) \cdot ({\omega_0}-{\omega})
\end{equation*}
where ${\tilde g} (e, v) = \langle g \rangle_{\mu_{\beta,\lambda}}$ and  $\nabla {\tilde g}$ is the gradient of the function ${\tilde g}$. 
As explained above we expect that in the Euler time scale, for any test function $\bH \in S(\RR) \times S(\RR)$, the space-time variance
\begin{equation}
\label{eq:4:BGvar}
\lim_{n \to \infty} {\mathbb E}_{\mu_{\beta, \lambda}} \left[ \left( \int_0^{t} \, \cfrac{1}{\sqrt{n}} \sum_{x \in \ZZ} \mb H\left( x/n \right)\cdot \left[\theta_x {\mc P}_{e, v}\,  g  \, (\eta (sn) )\right] \, ds\right)^2\right] =0
\end{equation}
vanishes as $n$ goes to infinity. In fact it suffices to show (\ref{eq:4:BGvar}) for the function $g=J_{0,1}${ \footnote{For Theorem \ref{th:fluct-ds}, the Boltzmann-Gibbs principle has to be proved in the longer time scale $t n^{1+ \alpha}$ and in the transport frame.}}. Thus let us first define the {\textit{normalized}} currents by
\begin{equation}
\label{norm. currents}
\begin{split}
{\hat J}_{x,x+1}& = \left[\theta_x {\mc P}_{\bar e, \bar v}\,  J_{0,1} \right]= \left(
\begin{array}{c}
{j}^{e,\gamma}_{x,x+1}(\eta)\\
{j}^{v,\gamma}_{x,x+1}(\eta)
\end{array}
\right)
- {{\mf J}} ({ \omega}) - (\nabla {{\mf J}}) ({\omega})
\left(
\begin{array}{c}
V_{\rm{exp}} (\eta_x) -{e}\\
\eta_x -{v}
\end{array}
\right).
\end{split}
\end{equation}

To estimate the space-time variance involved we use the following inequality (see \cite{KLO-book}):
\begin{equation}
\label{eq:4:inequS}
\begin{split}
 {\mathbb E}_{\mu_{\beta,\lambda}} \left[ \left( \int_0^t f (\eta (s n^{1+\alpha})) \, ds \right)^2\right] \, &\le  \, \cfrac{C t} {n^{1+\alpha}} \left\langle  f\;, \left( \cfrac{1}{t n^{1+\alpha}} - \gamma {\mc S} \right)^{-1} f \right\rangle_{\mu_{\beta,\lambda}}
 \end{split}
\end{equation}
where $f= \cfrac{1}{\sqrt{n}} \sum_{x \in \ZZ} \mb H\left( x/n \right)\cdot {\hat J}_{x,x+1}$. Due to the very simple form of the operator ${\mc S}$ the RHS of (\ref{eq:4:inequS}) can be estimated and shown to vanish as $n$ goes to infinity. Nevertheless it has to be done with some care since ${\mc S}$ is very degenerate so that without the term $\frac{1}{t n^{1+\alpha}}$ the RHS of (\ref{eq:4:inequS}) blows up.

Theorem \ref{th:fluct-ds} does not exclude the possibility of normal fluctuations, i.e. the convergence in law of the fluctuation field of the two conserved quantities to an infinite dimensional Ornstein-Uhlenbeck process in the diffusive time scale ($\alpha=1$). To see that it is not the case we will show that the diffusion coefficient ${\mc D}:={\mc D} ({e}, { v})$ appearing in this hypothetical limiting process would be infinite, excluding thus this possibility. Up to a constant matrix coming from a martingale term (due to the noise) and thus irrelevant for us (see \cite{BBO2}, \cite{BS}), the matrix coefficient ${\mc D}$ is defined by the Green-Kubo formula
\begin{equation}
\label{eq:GK00}
{\mc D} = \int_{0}^{\infty} {\mathbb E}_{\mu_{\beta, \lambda}} \left[ \sum_{x \in \ZZ} {\hat J}_{x,x+1} (t) \left[ {\hat J}_{0,1} (0) \right]^T \right] \, dt.
\end{equation}
The signature of the superdiffusive behavior of the system is seen in the divergence of ${\mc D}$, i.e. in a slow decay of the current-current correlation function. To study the latter we introduce its Laplace transform
\begin{equation*}
\begin{split}
{\mc F} (\gamma, z) = \int_{0}^{\infty} e^{-z t} \, {\mathbb E}_{\mu_{\beta, \lambda}} \left[ \sum_{x \in \ZZ} {\hat J}_{x,x+1} (t) \left[  {\hat J}_{0,1} (0)\right]^T  \right] \, dt
\end{split}
\end{equation*}
which is well defined for any $z>0$. This can be rewritten as
\begin{equation*}
\begin{split}
{\mc F} (\gamma, z) = \ll {\hat J}_{0,1}, (z-{\mc L})^{-1} {\hat J}_{0,1} \gg_{\beta,\lambda}
\end{split}
\end{equation*}
where $\ll \cdot, \cdot \gg_{\beta,\lambda}$ is the semi-inner product defined with respect to $\mu_{\beta,\lambda}$ in the same way as in (\ref{eq:scalarproduct>>}).

Our third theorem is the following lower bound on ${\mc F} (\gamma, z)$. Observe that ${\mc F} (\gamma, z)$ is a square matrix of size $2$ whose entry $(i,j)$ is denoted by ${\mc F}_{i,j}$.
\begin{theorem}[\cite{BG}]
\label{th:diffusivity}
Fix $\gamma>0$. There exists a positive constant $c:=c(\gamma)>0$ such that
\begin{equation*}
{\mc F}_{1,1} (\gamma, z) \ge c z^{-1/4}
\end{equation*}
and
\begin{equation*}
{\mc F}_{i,j} (\gamma, z)=0, \quad (i,j) \ne (1,1).
\end{equation*}
Moreover, there exists a positive constant $C:=C(\gamma)$ such that for any $z>0$,
\begin{equation}
\label{eq:F11c}
C^{-1} {\mc F}_{1,1} (1,z/\gamma) \le {\mc F}_{1,1} (\gamma, z) \le C {\mc F}_{1,1} (1,z/\gamma).
\end{equation}
\end{theorem}

%Notice that ${\mc F}_{i,j} (\gamma, z)=0$ for $(i,j) \ne (1,1)$ implies actually that the volume fluctuation field evolves in a diffusive time scale.
The last part of the theorem follows easily by a scaling argument and is in fact also valid for general potentials $V$ and for generic ``standard" anharmonic chains of oscillators.  In \cite{BDLLOP, ILOS, BS}, numerical simulations indicate a strange dependence w.r.t. the noise intensity $\gamma>0$ of the exponent $\delta$ in the energy transport coefficient $\kappa (N) \sim N^{\delta}$ ($N$ is the system size, see (\ref{eq:def_delta}) for the definition of $\kappa (N)$): $\delta:=\delta(\gamma)>0$ is increasing with the noise intensity $\gamma$. This is very surprising since the more stochasticity in the model is introduced, the more the system is superdiffusive!  The inequality (\ref{eq:F11c}) shows that the time decay of the current autocorrelation function is independent of $\gamma$ (up to possible slowly varying functions corrections, i.e. in a Tauberian sense). It is common folklore that there should be a simple relationship between the slow long-time tail decay of the autocorrelation of the current in the Green-Kubo formula (described by some power law decay) and the divergence of the thermal conductivity of open systems in their steady states. The argument is that the autocorrelation should be integrated over times of order~$N$. If we believe in this argument it means that the numerical simulations of \cite{BDLLOP, ILOS, BS} are not converged. There is however no clear mathematical result backing up this belief.\\

The proof of the first part of Theorem \ref{th:diffusivity} is based on the three following arguments. 
\begin{itemize}
\item  The first idea consists in performing the microscopic change of variables $\xi_x = e^{-\eta_x}$, $x \in \ZZ$, that defines a new Markov process $\{\xi (t)\}_{t \ge 0}= \{ \xi_x (t) \, ; \, x \in \ZZ\}_{t \ge 0}$ with state space $(0,+\infty)^{\ZZ}$ and conserving $\sum_x \xi_x$ and $\sum_{x} \log \xi_x$. Its generator is given by ${\tilde {\mc L}} ={\tilde {\mc A}} +\gamma {\tilde{ \mc S}}$ where for any local differentiable function $f$, 
\begin{equation*}
\begin{split}
&({\tilde {\mc A}} f)(\xi) = \sum_{x \in \ZZ} \xi_x (\xi_{x+1} -\xi_{x-1} ) (\partial_{\xi_x} f )(\xi),\\
&({\tilde {\mc S}} f)(\xi) = \sum_{x \in \ZZ} \left[ f(\xi^{x,x+1}) -f(\xi)\right].
\end{split}
\end{equation*} 
The invariant measures for $(\xi (t))_{t \ge 0} $ are obtained from the Gibbs measures $\mu_{\beta, \lambda}$ by the change of variables above. They form a family $\{ \nu_{\rho, \theta} \}_{\rho, \theta}$ of translation invariant product measures indexed by two parameters $\rho$ and $\theta$ which satisfy 
$$\rho =\nu_{\rho, \theta} (\xi_x), \quad \theta= \nu_{\rho, \theta} (\log \xi_x). $$
In fact the marginal of $\nu_{\rho, \theta}$ is a Gamma distribution. The parameters $(\rho, \theta)$ are in a one-to-one explicit correspondence with the parameters $(e,v)$.  

Rewriting ${\hat J}_{x,x+1}$ with these new variables we see that it is sufficient to prove a similar statement for the process $(\xi (t) )_{t \ge 0}$ under the equilibrium probability measure $\nu_{\rho,\theta}$. Introducing the inner product $\ll \cdot, \cdot \gg$ defined, for any local functions $f,g$ on $(0,+\infty)^{\ZZ}$ by 
\begin{equation*}
\ll f, g \gg =\sum_{x \in \ZZ} \left\{  \nu_{\rho, \theta} (f \,  \theta_x g) - \nu_{\rho,\theta} (f) \nu_{\rho, \theta} (g) \right\}
\end{equation*}
we can show that the proof of the first claim of Theorem \ref{th:diffusivity} reduces to showing that there exists a positive constant $c$ such that for any $z>0$,
\begin{equation}
\label{eq:4:w01}
\ll W_{0,1} , (z -{\tilde {\mc L}})^{-1} W_{0,1} \gg  \ge c z^{-1/4}
\end{equation}
where $W_{0,1} (\xi) =(\xi_0 -\rho) (\xi_1 -\rho)$. 

\item The second step consists in using a variational formula to express the LHS of (\ref{eq:4:w01}). Indeed we have
\begin{equation*}
\begin{split}
\ll W_{0,1} , (z -{\tilde {\mc L}})^{-1} W_{0,1} \gg &=\sup_{g} \left\{ 2 \ll W_{0,1}, g \gg - \ll g, (z -\gamma {\tilde{\mc S}}) g  \gg \right.\\
& \quad \quad \left.  - \ll {\tilde{\mc A}} g, (z-\gamma {\tilde{\mc S}})^{-1} {\tilde{\mc A}}g \gg \right\}
\end{split}
\end{equation*}
where the supremum is taken over local compactly supported smooth functions~$g$. To get a lower bound it is sufficient to find a function $g$ for which one can show that 
$$2 \ll W_{0,1}, g \gg - \ll g, (z -\gamma {\tilde{\mc S}}) g  \gg - \ll {\tilde{\mc A}} g, (z-\gamma {\tilde{\mc S}})^{-1} {\tilde{\mc A}}g \gg \; \ge \; c z^{-1/4}.$$

\item Let ${\bb H}$ be the Hilbert space obtained by completion of the set of local functions w.r.t. the inner product $\ll \cdot, \cdot \gg$. Since $\nu_{\rho, \theta}$ is a product of Gamma distributions, the set of multivariate Laguerre polynomials form an orthogonal basis of ${\bb H}$. It is then possible to decompose $\bb H$ as an orthogonal sum $\oplus_{n \in \N} {\bb H}_n$ of subspaces ${\bb H}_n$ such that 
\begin{equation*}
{\tilde{\mc S}}: {\bb H}_n \to {\bb H}_n, \quad {\tilde{\mc A}}:{\bb H}_n \to {\bb H}_{n-1} \oplus {\bb H}_n \oplus {\bb H}_{n+1}.
\end{equation*}
The function $W_{0,1}$ belongs to ${\bb H}_2$. Then we restrict the variational formula to functions $g \in {\bb H}_2$ and we estimate the corresponding new variational problem which is still infinite dimensional but involves only functions belonging to ${\bb H}_1 \oplus {\bb H}_2 \oplus {\bb H}_3$. To solve this variational problem we adapt ideas developed first in the context of Asymmetric Simple Exclusion Process (\cite{B1}, \cite{LQSY}) and exploited later for other models. One of the difficulties comes again from the fact that the noise is degenerate.  
\end{itemize}

The extension of Theorem \ref{th:diffusivity} to other interacting potentials is a challenging problem. The general strategy presented here could be carried out but the orthogonal basis (formed by Laguerre polynomials in the exponential case) is no longer explicit and only defined by some recurrence relations.

\begin{acknowledgement}
This work has been supported by the Brazilian-French Network in Mathematics and the French Ministry of Education through the ANR grant EDNHS. The referees deserve thanks for careful reading and many useful comments.
\end{acknowledgement}

\bibliography{revbibPSPDE-BERNARDIN}{}
\bibliographystyle{plain}
%\bibliography{}

%\addcontentsline{toc}{chapter}{Bibliography}

%\input{referencBernardin}
\end{document}